\documentclass[12pt]{amsart} 
\usepackage[margin=.85 in]{geometry}
\usepackage{thmtools}
\usepackage{tikz-cd}
\usepackage{amscd}
\usepackage{amssymb}
\usepackage{dsfont}
\usepackage{bbm}
\usepackage{mathrsfs}
\usepackage{scalerel}
\usepackage{mathbbol}

\theoremstyle{plain}
\newtheorem{lem}{Lemma}[subsection]
\newtheorem{prop}[lem]{Proposition}

\newtheorem{cor}[lem]{Corollary}

\theoremstyle{definition}
\newtheorem{const}[lem]{Construction}
\newtheorem{defi}[lem]{Definition}

\theoremstyle{remark}
\newtheorem*{remark}{Remark}
\newtheorem{example}{Example}[section]
\newtheorem{nota}[example]{Notation}
\newcommand{\trop}{\textnormal{trop}}

\numberwithin{equation}{section}

\newcommand{\bbG}{\mathbb{G}}

\newcommand{\bbN}{\mathbb{N}}

\newcommand{\bbR}{\mathbb{R}}

\newcommand{\bbV}{\mathbb{V}}

\newcommand{\bbZ}{\mathbb{Z}}

\newcommand{\cM}{\mathcal{ M}}

\newcommand{\cX}{\mathcal{ X}}
\newcommand{\cY}{\mathcal{ Y}}

\newcommand{\op}{\textnormal{op}}

\newcommand{\rc}{\operatorname{rc}}

\newcommand{\Hom}{\operatorname{Hom}}

\newcommand{\Top}{\operatorname{Top}}

\newcommand{\Lin}{\operatorname{Lin}}

\newcommand{\Id}{\operatorname{Id}}
\newcommand{\add}{\operatorname{Add}}

\newcommand{\colim}{\operatorname{colim}}

\newcommand{\pr}{\operatorname{pr}}

\newcommand{\Star}{\operatorname{Star}}

\newcommand{\drgdf}{\operatorname{drgdf}}
\newcommand{\rgdfst}{\operatorname{rgdfst}}

\newcommand{\val}{\operatorname{val}}

\newcommand{\Mtrop}{\mathrm{M}^{\trop}}
\renewcommand{\Im}{\operatorname{Im}}

\newcommand{\ord}{\operatorname{ord}}

\newcommand{\Aut}{\operatorname{Aut}}

\newcommand{\POIC}{\operatorname{POIC}}
\newcommand{\POICCmplxs}{\operatorname{pCmplxs}}
\newcommand{\POICSpcs}{\operatorname{pSpcs}}
\newcommand{\LinPOICCmplxs}{\operatorname{Lin-pCmplxs}}
\newcommand{\dist}{\operatorname{dist}}
\newcommand{\ST}{\operatorname{ST}}
\newcommand{\tint}{\textnormal{int}}
\newcommand{\pcomplexes}{\textnormal{c}}
\newcommand{\pspaces}{\textnormal{s}}
\newcommand{\st}{\operatorname{st}}

\newcommand{\Mf}{\operatorname{Mf}}

\newcommand{\ft}{\operatorname{ft}}

\newcommand{\Subd}{\operatorname{Subd}}
\newcommand{\fish}{\ensuremath{\operatorname{\begin{tikzpicture}
\draw[] (-0.5em,0.25em)--(-0.25em,0)--(-0.5em,-0.25em);
\draw[] (0,0) circle (0.25em);
\end{tikzpicture}}}}
\newcommand{\tailcirc}{\ensuremath{\operatorname{\begin{tikzpicture}
\draw[] (-0.5em,0.25em)--(-0.25em,0)--(-0.5em,-0.25em);
\draw[] (-0.25em,0)--(0.25em,0);
\draw[] (0.5em,0) circle (0.25em);
\end{tikzpicture}}}}
\newcommand{\eyetwolashes}{\ensuremath{\operatorname{\begin{tikzpicture}
\draw[] (-0.5em,0.25em)--(-0.25em,0);
\draw[] (0,0) circle (0.25em);
\draw[] (0.25em,0)--(0.5em,0.25em);

\end{tikzpicture}}}}
\title{A combinatorial extension of tropical cycles}
\author{Diego A. Robayo Bargans}
\address{FB Mathematik, RPTU Kaiserslautern, 67663 Kaiserslautern, Germany}
\email{robayo@mathematik.uni-kl.de}

\begin{document}

\begin{abstract}   
This article discusses a combinatorial extension of tropical intersection theory to spaces given by glueing quotients of partially open convex polyhedral cones by finitely many automorphisms. This extension is done in terms of linear poic-complexes and poic-fibrations, mainly motivated by the case of the moduli spaces of tropical curves of arbitrary genus and marking. We define tropical cycles of a linear poic-complex and of a poic-fibration, and discuss the pushforward maps in these situations. In the context of moduli spaces of tropical curves, we also discuss ``clutching morphisms'' and ``forgetting the marking'' morphisms. In a subsequent article we apply this framework to moduli spaces of discrete admissible covers and study the loci of tropical curves that appear as the source of a degree-$d$ discrete admissible cover of a genus-$h$ $m$-marked tropical curve, for fixed $d$, $h$ and $m$.
\end{abstract}



\maketitle

\section{Introduction}

Tropical intersection theory is an established research field whose foundations are closely linked to enumerative applications (for instance \cite{Mikhalkin1}, \cite{GathmannMarkwigKFWDVVETG}, \cite{AllermannRau}). It initially dealt with cycles in tropical fans (\cite{GathmannKerberMarkwigTFMSTC}) and tropical varieties, and has been more generally extended to tropical spaces (\cite{CavalieriGrossMarkwig}). Some particular spaces of interest are the moduli spaces $\cM_{g,A}^{\trop}$ of $A$-marked genus-$g$ tropical curves, where $g\geq 0$ is an integer and $A$ is a finite set (subject to the condition $2g-2+\#A>0$). If $g=0$, these are well-behaved and fully understood tropical varieties. In the general case, these spaces are neither tropical varieties nor tropical spaces, and sadly the existing theory cannot be directly applied to them. In this article, we propose a combinatorial extension of the classical intersection theory to spaces such as $\cM_{g,A}^{\trop}$. This project arose as the technical background developed while seeking to generalize the results of \cite{VargasDraisma}. Their approach relies on a deformation argument and a thorough examination of certain paths in $\cM_g^\trop$, as well as the combinatorics of \emph{DT-morphisms}. In a subsequent article, we apply our framework to moduli spaces of discrete admissible covers and obtain a generalization of their results.

In what follows, a tropical curve simply means a discrete graph (it may have legs) with a given metric. No non-trivial weight functions on the vertices are considered and, for the present purposes, it is assumed that a discrete graph is connected with all its vertices at least $3$-valent. Fixing the underlying discrete graph and varying the metric gives rise to an open cone, where the faces of its closure correspond to contractions of edges of the graph. We consider the partially open polyhedral cone, containing this open cone and the additional faces that correspond to genus preserving edge contractions. Given an integer $g\geq 0$ and a finite set $A$ with $2g-2+\#A>0$, the space $\cM_{g,A}^\trop$ is obtained by taking all the partially open integral cones associated to (isomorphism classes of) genus-$g$ $A$-marked discrete graphs and identifying points representing isometric tropical curves. The presence of possible non-trivial automorphisms at the associated cone of a discrete graph of positive genus forces this presentation to not be a partially open fan (or partially open polyhedral complex), in contrast to the trivial genus case where the space is naturally a simplicial fan. The gist of our approach is to describe tropical cycles in $\cM^\trop_{g,A}$ by means of rational curves through ``cutting'' and ``glueing'' edges. More precisely, let $\mathbbm{g}=\{1,\dots,g,1^*,\dots,g^*\}$ and consider the map
\begin{equation*}
    \st_{g,A}\colon\cM^\trop_{0,A\sqcup \mathbbm{g}}\times \bbR^g_{>0}\to\cM^\trop_{g,A}, 
\end{equation*}
where $\st_{g,A}(\Gamma,\delta_1,\dots,\delta_g)$ is the graph given by connecting the $i$- and $i^*$-legs by an interval of length $\delta_i>0$ (and forgetting the $i$- and $i^*$-marked legs). Combinatorially, we are regarding the input tree as a spanning tree of the target graph. This map is an example of a poic-fibration, and we make use of it to describe cycles of $\cM^\trop_{g,A}$ as particular subcycles of $\cM^\trop_{0,A\sqcup \mathbbm{g}}\times \bbR_{>0}^g$. More precisely, we speak of $\st_{g,A}$-equivariant cycles. 

The previous construction is our motivating example from nature and, as it shows, prompts us directly to partially open integral cones (poics), partially open fans (as in \cite{GathmannOchseMSCTV}), and more generally to introduce poic-complexes. In simple terms, these are to partially open fans, as cone complexes (\cite{GrossITTTE}) are to fans. The underlying categorical point of view for this definition is further endorsed by the subsequent introduction of poic-fibrations and the role of automorphisms in the whole. Additional approaches to study $\cM_{g,A}^\trop$ have been done by considering non-trivial weight functions on the underlying discrete graphs, and using a barycentric subdivisions to understand its structure as an extended cone complex \cite{AbramovichCaporasoPayneTTMSC}. However, following our original motivation and keeping in mind subsequent constructions involving tropical admissible covers, we do not consider non-trivial weight functions. Nevertheless, the intrinsic interaction of the combinatorics of barycentric subdivisions and graph automorphisms is present in a different guise in our work through the compatibility of a subdivision with a given poic-fibration. 

 In \cite{CavalieriGrossMarkwig} a different approach to tropical intersection theory in $\cM_g^\trop$ (more generally in \emph{tropical spaces}) has been proposed, with the goal of introducing $\psi$-classes for moduli spaces of tropical curves of arbitrary genus. There the authors take a stacky approach to this space in the site of tropical spaces, and also allow for non-trivial weight functions on the underlying discrete graphs. This yields a non-geometric stack, and impedes the introduction of tropical cycles on these moduli spaces. They do introduce line bundles and divisors thereon, and describe how to intersect them with tropical cycles in spaces equipped with a morphism to the stack. Which is of special importance due to their interest in introducing tropical $\psi$-classes. We content ourselves with exhibiting the relationship between our methods and the tropical cycles on their cycle rigidification of Mumford curves.

A synopsis of the paper is as follows: in section $2$ the preliminaries of partially open polyhedral cones, discrete graphs, and the relevant notions concerning tropical curves are introduced. The definition of a discrete graph follows that of \cite{ChanGalatiusPayneTCGCTWCMg} and \cite{LenUlirschZakharovATC}, which is recalled for the sake of convenience. The necessary categories of discrete graphs are introduced in \ref{ssec: our categories of graphs}, with a detailed explanation of the corresponding morphisms. The section finishes with a presentation of the moduli space of $A$-marked genus-$g$ curves. Section $3$ deals with the notion of (linear) poic-complexes and the tropical intersection theory thereof. Examples and comparisons to the classical tropical intersection theory are realized in parallel to the introduction of these structures. In section $4$ the notions of poic-spaces and poic-fibrations are introduced. The notion of poic-spaces mimicks the presentation of $\cM_g^\trop$. Namely, a poic-space is space glued out of quotients of partially open integral cones by finitely many automorphisms. A poic-fibration is a special kind of morphism from a (linear) poic-complex to a poic-space with certain lifting properties. This is the central notion of the article and, together with the previous section, forms the core of this work. Additionally, the aforementioned poic-fibration over the moduli space of tropical curves of a fixed genus and marking is introduced, along with the ``forgetting the marking'' and ``clutching'' morphisms. In the spirit of clarity and better presentation, multiple technical constructions are postponed to the appendix. 

\subsection*{Acknowledgments:} I would like to thank A. Gathmann for several helpful and inspiring discussions that made possible this project and brought it into existence, as well as for proofreading and patience. Interactions made possible by the SFB-TRR 195 of the DFG have also benefited this project. I also thank A. Vargas and J. M. P\'erez for helpful conversations. This project has been carried out and realized through financial support of the DAAD. 

\section{Preliminaries} In the interest of clarity and readability, we fix notation and state precise definitions for our basic objects of interest: partially open integral (polyhedral) cones, and discrete graphs. We then describe a category of graphs and associate to each graph in this category a partially open integral (polyhedral) cone. These are the cones of metrics, and together with the previous category give rise to a presentation of the moduli space of tropical curves of a fixed genus and marking.

\subsection{Partially open integral cones.} Suppose $N$ is a free abelian group of finite rank and consider the vector space $N_\bbR:=N\otimes_\bbZ\bbR$. The dual of $N$ is the free abelian group (of finite rank) $N^\vee = \Hom_\bbZ(N,\bbZ)$, and there is a natural morphism $N^\vee\to N_\bbR^\vee$ that gives rise to an isomorphism $N^\vee \otimes_\bbZ\bbR\cong N_\bbR^\vee$. In this way, an element $f\in N^\vee$ gives rise to a closed half-space and an open half-space of $N_\bbR$ correspondingly:
\begin{align*}
    H_f &:= \{ x\in N_\bbR \colon f(x)\geq 0\},&
    H_f^o &:= \{ x\in N_\bbR\colon  f(x)>0\}.
\end{align*} 

\begin{defi}\label{defi: poic of Nr}
A \emph{partially open integral cone of $N_\bbR$} is a non-empty subset $\sigma\subset N_\bbR$ given as a finite intersection of open and closed half-spaces, i. e.
\begin{equation}
    \sigma = \bigcap_{i=0}^nH_{f_i} \cap \bigcap_{j=0}^m H_{g_j}^o\label{eq: poic}
\end{equation} 
where $f_i\in N^\vee$ and $g_j\in N^\vee$ for $0\leq i\leq n$ and $0\leq j\leq m$.
\end{defi}

\begin{nota}
If $\sigma$ is a partially open integral cone of $N_\bbR$ given by \eqref{eq: poic}, then the relative interior of $\sigma$ is the partially open integral cone of $N_\bbR$ given by
    \begin{equation*}
        \sigma^o = \bigcap_{\substack{i=0, \\ \sigma \subset \{f_i=0\}}}^nH_{f_i}\cap\bigcap_{\substack{i=0,\\
        \sigma\not\subset \{f_i=0\}}}^nH_{f_i}^o \cap \bigcap_{j=0}^m H_{g_j}^o.
    \end{equation*}
    Naturally if $\sigma$ is a partially open integral cone of $N_\bbR$, then $\sigma^o$ is also a partially open integral cone. In addition, we denote denote by $\Lin(\sigma)$ the linear subspace generated by $\sigma$, and by $d(\sigma)$ the dimension of $\sigma$ (namely $d(\sigma) = \dim \Lin(\sigma)$).
\end{nota}

\begin{defi}
Suppose $\sigma$ is a partially open integral cone of $N_\bbR$. A \emph{face} $\tau$ of $\sigma$ is a partially open integral cone given by $\tau = \sigma\cap H_f$, where $f\in N^\vee$ is such that $\sigma\subset H_{-f}$. We denote this situation by $\tau\leq \sigma$. If $\tau\leq \sigma$ and $\tau\neq \sigma$, then $\tau$ is called a \emph{proper face} and we emphasize this distinction by $\tau\lneq\sigma$.
\end{defi}

\begin{remark}
If $\tau\leq \sigma$, then $\Lin(\tau)$ is a subspace of $\Lin(\sigma)$ and $d(\tau)\leq d(\sigma)$. 
\end{remark} 

\begin{defi}\label{defi: poic}
A \emph{partially open integral cone} consists of a pair $(\sigma, N)$ where $N$ is a free abelian group of finite rank and $\sigma$ is a full-dimensional partially open polyhedral cone of $N_\bbR$. A \emph{face} of a partially open integral cone $(\sigma,N)$ is $(\tau,\Lin_N(\tau))$ where $\tau$ is a face of $\sigma$ in $N_\bbR$ and $\Lin_N(\tau) = N\cap \Lin(\tau)$. A \emph{subcone} of a partially open integral cone $(\sigma,N)$ is a partially open integral cone $(\xi,M)$ such that $\xi\subset \sigma$ and $M = N\cap \Lin(\xi)$.
\end{defi}

\begin{remark}
    If $(\sigma,N)$ is a partially open integral cone, then $(\sigma^o,N)$ is also a partially open integral cone.
\end{remark}

Hereafter, we will refer to a partially open integral cone simply as a poic. Observe that Definition \ref{defi: poic of Nr} is different from Definition \ref{defi: poic}: The latter consists of a pair, whereas the former merely refers to a subset of a specified vector space. Due to subsequent sections where we want to consider multiple cones that may not naturally inhabit a common vector space, we insist on this difference and draw attention to it.

When no risk of confusion seems near, the following notation is adopted:
\begin{nota}
If the pair $(\sigma,N)$ is a poic, then the lattice $N$ will be denoted by $N^\sigma$ and the pair simply by $\sigma$.    
\end{nota}

\begin{const}
Suppose $(\sigma,N)$ and $(\xi,M)$ are two poics. There is a natural isomorphism $N_\bbR\oplus M_\bbR \cong (N\oplus M)_\bbR$. Separately, the product $\sigma\times \xi$ is a partially open integral cone of $N_\bbR\oplus M_\bbR$. Abusing the aforementioned natural identification, let us denote by $\sigma\times\xi$ the corresponding cone of $(N\oplus M)_\bbR$, so that the pair $(\sigma\times\xi,N\oplus M)$ is also a poic.
\end{const}

\begin{defi}
Let $\sigma$ and $\xi$ be two poics. A \emph{morphism} $f\colon\sigma\to \xi$ consists of a linear morphism $f\colon  N^\sigma\to N^\xi$, such that the induced linear map $f_\bbR:N^\sigma_\bbR\to N^\xi_\bbR$ sends $\sigma$ to $\xi$. A morphism $f\colon \sigma\to \xi$ is called a \emph{face-embedding}, if $f$ identifies $\sigma$ with a face of $\xi$ (i. e. if $f_\bbR$ is injective and $f_\bbR(\sigma)$ is a face of $\xi$).
\end{defi}

\begin{example}
    If $\sigma$ and $\xi$ are two poics, then $\sigma\times \xi$ is also a poic and both projections $\sigma\times\xi\to\sigma$ and $\sigma\times\xi\to \xi$ are morphisms of poics.
\end{example}

\begin{example}
If $\tau$ is a face of $\sigma$ then the inclusion $\tau\to \sigma$ is a morphism (and undoubtedly a face-embedding) where the morphism between the lattices is the natural inclusion.
\end{example}

The data of poics and morphisms thereof, together with natural composition of maps, defines a category, which we denote by $\POIC$. A poic $\sigma$ carries naturally an underlying topological space: the cone itself. More precisely, if $(\sigma,N^\sigma)$ is a poic, let $|\sigma|$ denote the topological space given by $\sigma\subset N_\bbR$ with the Euclidean topology. If $\xi$ is also a poic and $f\colon \sigma\to\xi$ is a morphism, then $f$ induces a linear map $f\colon N^\sigma_\bbR\to N^\xi_\bbR$ which restricts to a continuous map
$|f|\colon |\sigma|\to|\xi|$. This association gives rise to a functor
\begin{equation}
    |\bullet|\colon \POIC\to \Top,\sigma\mapsto |\sigma|,\label{eqdefi: realization of poics}
\end{equation}
which we call the \emph{realization functor}. Furthermore, if $\sigma$ is a poic then $|\sigma|$ is called the \emph{realization of $\sigma$} (analogously with morphisms).

\subsection{Discrete graphs.} We use the definition of discrete graph with legs as in \cite{LenUlirschZakharovATC}.
 
\begin{defi}
    A \emph{discrete graph (possibly with legs)} $G$ consists of the data of a triple $(F(G),r_G,\iota_G)$ where:
    \begin{itemize}
        \item $F(G)$ is a finite set. 
        \item $r_G\colon F(G)\to F(G)$ is a map of sets, which we call the root map.
        \item $\iota_G\colon F(G)\to F(G)$ is an involution subject to $\iota_G\circ r_G=r_G$. 
    \end{itemize}
    The \emph{set of vertices} of $G$ is the subset of $F(G)$ given by $\Im(r_G)$. It will be denoted by $V(G)$ and its complement by $H(G)$. 
\end{defi}

Suppose $G$ is a discrete graph. Since the involution $\iota_G$ satisfies $\iota_G\circ r_G=r_G$, it restricts to an involution of $H(G)$ and hence partitions this set into orbits of size $2$ or $1$. We make use of the following definitions:

\begin{itemize}
    \item A \emph{leg} of $G$ is a size-$1$ orbit of $H(G)$. The set of legs will be denoted by $L(G)$.
    \item For $\ell\in L(G)$ given by $\ell=\{L\}\subset H(G)$, the \emph{boundary} is $\partial \ell \colon =\{ r_G(L)\}$.
    \item An \emph{edge} of $G$ is a size-$2$ orbit of $H(G)$. The set of edges will be denoted by $E(G)$.
    \item For $e\in E(G)$ given by $e=\{E,E^\prime\}\subset H(G)$, the \emph{boundary} is $\partial e\colon = \{r_G(E),r_G(E^\prime)\}$.
    \item A vertex is \emph{incident} to an edge or a leg if it is in its boundary.
    \item  The graph $G$ is \emph{connected}, if for any two different vertices $V,W\in V(G)$ there exists a sequence $e_1,\dots,e_n\in E(G)$ such that:
        \begin{enumerate}
            \item $V\in \partial e_1$ and $W\in \partial e_n$.
            \item For any $1\leq i\leq n-1$, $\partial e_i\cap \partial e_{i+1} \neq \varnothing$.
        \end{enumerate}
    \item The \emph{genus} of a connected graph $G$ is the number $g(G)\colon =\#E(G)-\#V(G)+1$.
    \item A \emph{tree} is a connected graph of genus $0$.
    \item A \emph{subgraph} $K$ of $G$ is a graph $(F(K),r_K,\iota_K)$ such that $F(K)\subset F(G)$ with $r_G|_{F(K)}=r_K$ and $\iota_G|_{F(K)} =\iota_K$.
    \item The \emph{valency} of $V\in V(G)$ is the number $\val(v) := \# r_G^{-1}(V)-1$. 
\end{itemize}

\begin{example}
    The theta graph (see Figure \ref{fig: theta graph}) is given by two vertices $A$ and $B$, and three edges $e_1$, $e_2$ and $e_3$ such that $\partial e_i = \{A,B\}$ (for $1\leq i\leq 3)$. To present this as a discrete graph we consider the set (the elements are just names)
    \begin{equation*}
        F(G) \colon = \{AB_1,AB_2,AB_3,BA_1,BA_2,BA_3,A,B\},
    \end{equation*}
    together with the root map
    \begin{equation*}
        r_G\colon F(G)\to F(G),\begin{array}{cc}
            r_G(AB_i)= A, & 1\leq i\leq 3 , \\
             r_G(BA_i)= B,& 1\leq i\leq 3, \\
             r_G(A) = A, & \\
             r_G(B) = B,
        \end{array}
    \end{equation*}
    and involution 
    \begin{equation*}
        \iota_G\colon F(G)\to F(G),\begin{array}{cc}
             \iota_G(AB_i)=BA_i, &1\leq i\leq 3,  \\
             \iota_G(BA_i)=AB_i, &1\leq i\leq 3,  \\
             \iota_G(A)=A, &  \\
             \iota_G(B)=B. &
        \end{array}
    \end{equation*}
    It is immediate to check that the previous data gives rise to a graph. In addition, $V(G)= \{A,B\}$, $L(G) = \varnothing $, and $E(G) = \{e_1,e_2,e_3\}$ where for $1\leq i\leq 3$
    \begin{equation*}
        e_i \colon =\{AB_i,BA_i\}.
    \end{equation*}
    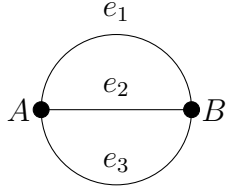
\begin{figure}[htbp]
        \centering
        \begin{tikzpicture}
            \draw[fill=none](0,0) circle (1);
            \draw[fill=black](-1,0) circle (3pt);
            \draw[fill=black](1,0) circle (3pt);
            \draw[](-1,0)--(1,0);
            \node[] at (-1.3,0) {$A$};
            \node[] at(1.3,0) {$B$};
            \node[] at (0,1.3) {$e_1$};
            \node[] at (0,0.3) {$e_2$};
            \node[] at (0,-0.7) {$e_3$};
        \end{tikzpicture}
        \caption{The theta graph}
        \label{fig: theta graph}
    \end{figure}
    
\end{example}

\begin{nota}
    Whenever we depict a graph, we will do so as in Figure \ref{fig: theta graph}. Namely, both the vertices and edges will be labeled.
\end{nota}

More than maps between graphs, we care for the bijective maps between them. However, it is neither burdensome nor  fruitless to state a general definition for what we call maps of graphs.

\begin{defi}\label{defi: maps of graphs}
 Suppose $G_1$ and $G_2$ are discrete graphs. A \emph{map of graphs} $f\colon G_1\to G_2$ is a map of sets $f\colon F(G_1)\to F(G_2)$ such that:
\begin{itemize}
    \item It commutes with the root maps (i. e. $r_{G_2}\circ f = f\circ r_{G_1}$).
    \item It commutes with the involution maps (i. e. $\iota_{G_2}\circ f = f\circ \iota_{G_1}$).
\end{itemize}
If the map $f\colon F(G_1)\to F(G_2)$ is a bijection, then the map of graphs is additionally called \emph{bijective}. In this case, $f$ induces a set bijection
$E(f):E(G_1)\to E(G_2)$.
\end{defi}
As will be apparent afterwards, edge contractions of graphs are of crucial interest for our endeavors. Hence, we carry out this construction separately, dealing meticulously with notation.

\begin{const}\label{const: edge contraction}
    Let $G$ be a discrete graph, and suppose $e\in E(G)$ is an edge of $G$. Let $F(G/e)$ denote the set obtained from $F(G)$ by identifying the subset $e\cup\partial e\subset F(G)$ into a single element $V_e$. This identification gives rise to a natural surjective map $p_e\colon F(G)\to F(G/e)$. Furthermore, the following maps are well-defined
    \begin{align*}
        &r_{G/e}\colon  F(G/e)\to F(G/e),& H\mapsto \begin{cases}
            p_e(r_G(H)), &\textnormal{ if }H\neq V_e, \\
             V_e,&\textnormal{ if }H=V_e,\\
        \end{cases}\\
        &\iota_{G/e}\colon F(G/e)\to F(G/e),&H \mapsto \begin{cases}
            p_e(\iota_G(H)), &\textnormal{ if }H\neq V_e,\\
            V_e,&\textnormal{ if }H=V_e,\\
        \end{cases}\\
    \end{align*}
    and these maps satisfy $\iota_{G/e}\circ r_{G/e} = r_{G/e}$.
\end{const}

\begin{defi}
    For a discrete graph $G$ with $e\in E(G)$, the \emph{contraction of $e$ in $G$} is the graph $G/e$ given by the triple $(F(G/e),r_{G/e},\iota_{G/e})$.
\end{defi}

We leave the proof of the next lemma to the reader. 

\begin{lem}\label{lem: contraction of two edges}
    Let $G$ be a discrete graph, with $e_1,e_2\in E(G))$ two different incident edges. Then $\left(G/e_1\right)/e_2$ is naturally bijective to $\left(G/e_2\right)/e_1$.
\end{lem}

Suppose $K$ is a legless subgraph of $G$. It is clear from Lemma \ref{lem: contraction of two edges} that contracting the edges of $K$ in $G$ in any order will produce naturally bijective distinct graphs. The distinction between them only comes from the newly introduced vertices in the last step of the successive contractions. The following definition is made in order to dump these natural identifications and replace them by a single object:

\begin{defi}
Let $K$ be a legless subgraph of $G$ with $K=K_1\sqcup \dots \sqcup K_m$ its connected components. The \emph{contraction of $K$ in $G$} is the graph $G/K$ obtained by the succesive contractions of the edges of $K_1, \dots ,K_m$, where the new vertex obtained by the contraction of all the edges of the component $K_i$ is denoted by $V_{K_i}$.
\end{defi}

We remark that if $K$ is subgraph of $G$ then the contraction of $K$ in $G$ is a new graph $G/K$ whose set of vertices consists of $V(G)\backslash V(K)$ and one additional vertex from each component of $K$. Some immediate computations show that if $G$ is connected, then so is $G/K$ (for arbitrary $K$), and when $K$ is connected
\begin{equation*}
g(G/K) = g(G)-g(K).
\end{equation*}

\begin{lem}\label{lem: map of graphs is contraction}
    Suppose $f\colon G_1\to G_2$ is a map of graphs such that:
    \begin{itemize}
        \item For every $V\in V(G_2)$, the inverse image $f^{-1}(V)\subset F(G_1)$ is a legless tree.
        \item For every $e\in E(G_2)\cup L(G_2)$, its inverse image under $f$ consists of a single edge or leg of $G_1$.
    \end{itemize}
    Then there is a natural bijective map of graphs $G_1/f^{-1}\left(V(G_2)\right)\to G_2$.
\end{lem}
\begin{proof}
    We show the lemma in the case where there is a vertex $V\in V(G_2)$ such that for all $W\in V(G_2)\backslash \{V\}$ the inverse image $f^{-1}(W)$ is just a vertex (a trivial legless tree). This is sufficient, since every map of graphs as in the statement can be decomposed as a composition of these. In this case, if $f^{-1}(V)$ is just a vertex, then $f$ is a bijective map between $G_1$ and $G_2$. Therefore, suppose $f^{-1}(V)$ is a non-trivial legless tree. The second condition implies that $f$ restricts to a bijection between $F(G_1)\backslash f^{-1}(V)$ and $F(G_2)\backslash\{V\}$, and hence induces a bijection $F(G_1/f^{-1}(V))\to F(G_2)$. The commutation with respect to root and involution maps follows directly from that of $f$. 
\end{proof}

The following definition is motivated by Lemma \ref{lem: map of graphs is contraction}.

\begin{defi}\label{defi: morphs GgA}
    A map of graphs $f\colon G_1\to G_2$ is a \emph{contraction}, if for every $V\in V(G_2)$ the subgraph defined by the inverse image $f^{-1}(V)$ is a legless tree of $G_1$.
\end{defi}

\subsection{Some categories of graphs and their structures.}\label{ssec: our categories of graphs} Suppose $A$ is a finite set. A \emph{discrete graph with $A$-marked legs} is a discrete graph $G$ with a bijection 
\begin{equation*}
\ell_\bullet(G)\colon A\to L(G),a\mapsto \ell_a(G),
\end{equation*} 
which we will call the \emph{marking}. For an element $a\in A$, the \emph{$a$-leg of $G$} is simply $\ell_a(G)\in L(G)$. A \emph{map of $A$-marked graphs} is simply a map of graphs that commutes with the respective markings.

It is clear that a composition of contractions is a contraction. Suppose $2g+\#A-2>0$ and let $\bbG_{g,A}$ denote the category specified by:
\begin{itemize}
    \item The objects of $\bbG_{g,A}$ consist of the connected genus-$g$ discrete graphs with $A$-marked legs whose vertices are at least $3$-valent.
    \item For two objects $G_1$ and $G_2$ of $\bbG_{g,A}$, the set of morphisms $\Hom_{\bbG_{g,A}}(G_1,G_2)$ consists of the contractions $f\colon G_1\to G_2$.
    \item Composition of morphisms is just composition of maps.
\end{itemize}

\begin{remark}
    If $A$ is empty, then this category is equivalent to the category $\mathbb{\Gamma}_g$ of \cite{ChanGalatiusPayneTCGCTWCMg}. 
\end{remark}

\begin{nota}
    In case $A=\{1,\dots,n\}$ for $n\in \bbN$, then the category $\bbG_{g,A}$ will simply be denoted by $\bbG_{g,n}$. In this case, an $A$-marked graph $G$ will simply be called an $n$-marked graph, and for $1\leq i\leq n$ the \emph{$i$th leg} will be $\ell_i(G)$.
\end{nota}

\begin{example}
    Let $g=2$ and $n=1$. Analogous to the previous example, Figure \ref{fig: G21} depicts a skeleton of $\bbG_{2,1}$, which excludes the automorphisms of the objects and where the arrows denote the contracted edge.
    \begin{figure}[htbp]
        \centering
        \begin{tikzpicture}
            \draw[fill=none](-5,10) circle (1);
            \draw[fill=black](-5,11) circle (3pt);
            \draw[fill=black](-5,9) circle (3pt);
            \draw[fill=black](-4,10) circle (3pt);
            \draw[](-5,11)--(-5,9) node [midway, left] {$c_2$};
            \draw[dashed, line width = 1pt](-4,10)--(-3.5,9);
            \node at (-6.3,10) {$c_1$};
            \node at (-4.1,10.9) {$c_3$};
            \node at (-4.1,9.1) {$c_4$};

            \draw[fill=none](-1.5,10) circle (1);
            \draw[fill=none](1.5,10) circle (1);
            \draw[fill=black](-0.5,10) circle (3pt);
            \draw[fill=black](0.5,10) circle (3pt);
            \draw[fill=black](2.5,10) circle (3pt);
            \draw[](-0.5,10)--(0.5,10) node [midway, above] {$d_2$};
            \draw[dashed, line width = 1pt](2.5,10)--(2.5,9);
            \node at (-1.5,11.3) {$d_1$};
            \node at (1.5,11.3) {$d_3$};
            \node at (1.5,9.3) {$d_4$};

            \draw[fill=none](4,10) circle (1);
            \draw[fill=none](8,10) circle (1);
            \draw[fill=black](5,10) circle (3pt);
            \draw[fill=black](7,10) circle (3pt);
            \draw[fill=black](6,10) circle (3pt);
            \draw[dashed,line width=1pt](6,10)--(6,9);
            \draw[](5,10)--(6,10) node [midway, above] {$e_2$};
            \draw[](6,10)--(7,10) node [midway, above] {$e_3$};
            \node at (4,11.3) {$e_1$};
            \node at (8,11.3) {$e_4$};

            \draw[to-,line width = 2pt,black](0,6.5)--(0,8.5) node [midway, left] {$d_2$};
            \draw[to-,line width = 2pt,black](-1.5,6.5)--(-2.5,8.5) node [midway, left] {$c_2$};
            \draw[to-,line width = 2pt,black](-2.5,6.5)--(-3.5,8.5)node [midway,left] {$c_1$};
            \draw[to-,line width = 2pt,black](-5.5,6.5)--(-5.5,8.5) node [midway, left] {$c_3$};
            \draw[to-,line width = 2pt,black](-4.5,6.5)--(-4.5,8.5) node [midway, left] {$c_4$};
            \draw[to-,line width = 2pt,black](3,6.5)--(2,8.5) node [midway, left]{$d_3$};
            \draw[to-,line width = 2pt,black](4,6.5)--(3,8.5) node [midway, left] {$d_4$};
            \draw[to-,line width = 2pt,black](6.5,6.5)--(6.5,8.5) node [midway, left] {$e_3$};
            \draw[to-,line width = 2pt,black](5.5,6.5)--(5.5,8.5) node [midway, left] {$e_2$};

            \draw[fill=none](-5,5) circle (1);
            \draw[fill=black](-6,5) circle (3pt);
            \draw[fill=black](-4,5) circle (3pt);
            \draw[](-6,5)--(-4,5) node [midway, above] {$i_2$};
            \draw[dashed, line width = 1pt](-4,5)--(-4,4);
            \node at (-5,6.3) {$i_1$};
            \node at (-5,4.3) {$i_3$};
            
            \draw[fill=none](-1,5) circle (1);
            \draw[fill=none](1,5) circle (1);
            \draw[fill=black](0,5) circle (3pt);
            \draw[fill=black](2,5) circle (3pt);
            \draw[dashed, line width = 1pt](2,5)--(2,4);
            \node at (-1,6.3) {$j_1$};
            \node at (1,6.3)  {$j_2$};
            \node at (1,4.3)  {$j_3$};

            \draw[fill=none](4.5,5) circle (1);
            \draw[fill=none](7.5,5) circle (1);
            \draw[fill=black](5.5,5) circle (3pt);
            \draw[fill=black](6.5,5) circle (3pt);
            \draw[dashed,line width=1pt](6.5,5)--(6.5,4);
            \draw[](5.5,5)--(6.5,5) node [midway,above] {$k_2$};
            \node at (4.5,6.3) {$k_1$};
            \node at (7.5,6.3) {$k_3$};

            \draw[to-,line width = 2pt,black](-4,1.5)--(-5,3.5) node [midway, left]{$i_1$};
            \draw[to-,line width = 2pt,black](-3,1.5)--(-4,3.5) node [midway, left]{$i_2$};
            \draw[to-,line width = 2pt,black](-2,1.5)--(-3,3.5) node [midway, left]{$i_3$};
            \draw[to-,line width = 2pt,black](-0.5,1.5)--(-0.5,3.5)node [midway, left]{$j_2$};
            \draw[to-,line width = 2pt,black](0.5,1.5)--(0.5,3.5)node [midway, left]{$j_3$};
            \draw[to-,line width = 2pt,black](3,1.5)--(5,3.5) node [midway, above left]{$k_2$};

            \draw[fill=none](-1,0) circle (1);
            \draw[fill=none](1,0) circle (1);
            \draw[fill = black](0,0) circle (3pt);
            \draw[dashed,line width=1pt](0,0)--(0,-1);
            \node at (-1,1.2) {$m_1$};
            \node at (1,1.2) {$m_2$};
            
        \end{tikzpicture}
        \caption{Skeleton of $\bbG_{2,1}$ without automorphisms of objects.}
        \label{fig: G21}
    \end{figure}
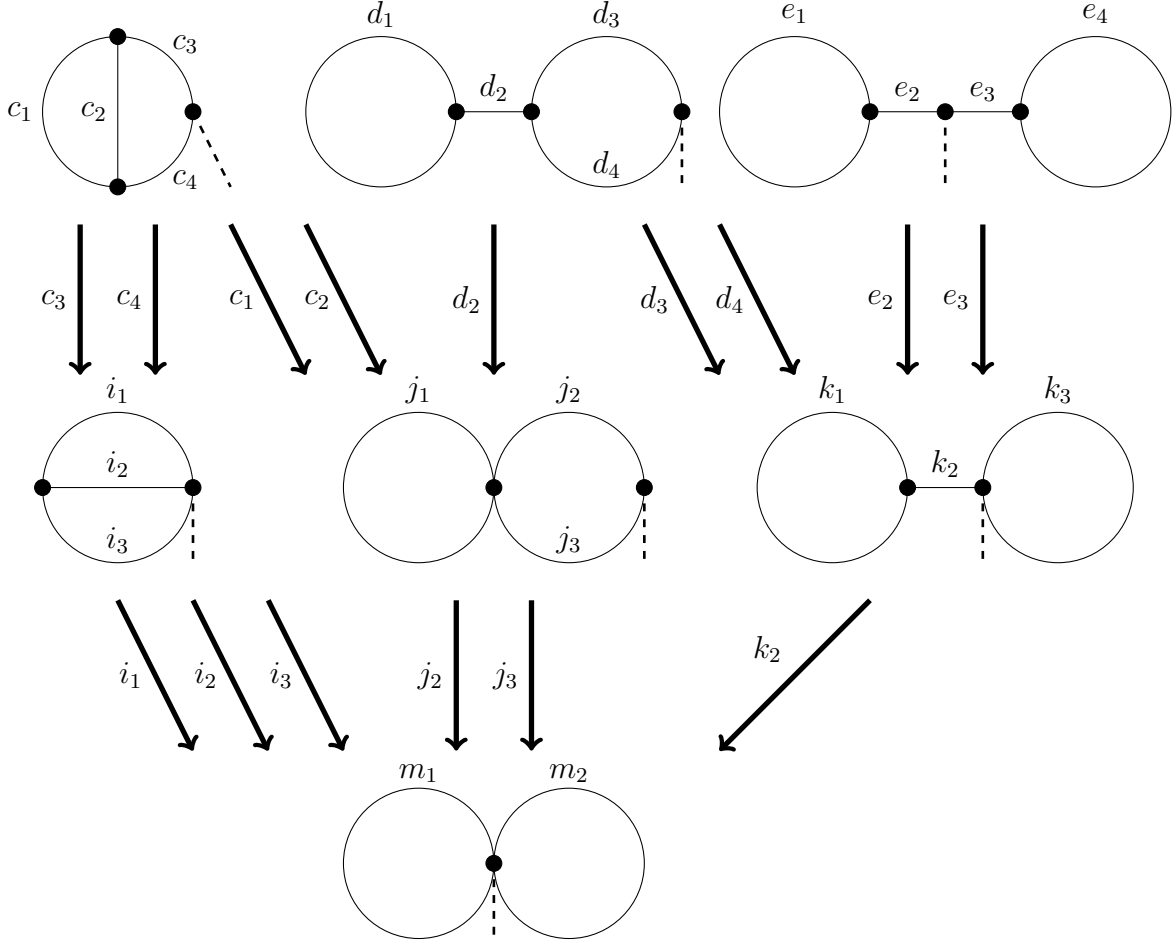
\end{example}

\begin{prop}\label{prop: G0n iso classes}
    Any object of $\bbG_{0,A}$ has no non-trivial automorphisms, and any skeleton of this category is a finite poset. Furthermore, for any integer $g\geq 0$ with $2g+\#A-2\geq 0$, the category $\bbG_{g,A}$ has finitely many isomorphism classes. 
\end{prop}
\begin{proof}
    Both statements are well known facts, that can also be readily checked.
\end{proof}

\subsection{Tropical curves and their moduli spaces.} A \emph{metric} on a discrete graph $G$ is a function $\delta\colon  E(G)\to \bbR_{>0}$. Such a pair gives rise to a metric space: to each edge an interval of length prescribed by $\delta$ is associated, to each leg a half line $\bbR_{\geq 0}$ is associated, and all these are glued according to the incidence relations of $G$. For a finite set $A$, an \emph{$A$-marked tropical curve} $\Gamma$ is a metric space so arising from a pair $(G,\delta)$, where $G$ is an $A$-marked discrete graph and $\delta$ is a metric on $G$. 

Naturally, many different pairs can yield isometric spaces, and hence some further restrictions are necessary to study such spaces. More interestingly, we want to understand and study some invariants of the space parametrizing the connected genus-$g$ $A$-marked tropical curves. This is the moduli space of tropical curves $\cM_{g,n}^{\trop}$, and our previously introduced category $\bbG_{g,A}$ provides a neat way of constructing (and studying) it.

\begin{nota}
For a finite set $A$, we let $\bbR_{\geq0}^A$ denote the poic given by the non-negative orthant in the vector space $(\bbZ^A)_\bbR$ and the lattice $\bbZ^A$. Analogously, we let $\bbR_{>0}^A$ denote the poic given by the positive orthant in this same vector space and the same lattice. It is clear that subsets $B\subset A$ give rise to face-embeddings $\bbR_{\geq0}^B\to \bbR_{\geq0}^A.$   
\end{nota}

\begin{const}\label{const: face-embeddings}
   Suppose $A$ is a fixed finite set, and $g$ is a non-negative number with $2\#A+g-2>0$. Let $G$ be an object of $\bbG_{g,A}$. Unmarked subgraphs of $G$ that do not have vertices as connected components correspond to subsets of $E(G)$, and hence to particular face-embeddings of the poic $\bbR_{\geq0}^{E(G)}$. If $K$ is an unmarked subgraph of $G$ whose connected components are non-trivial trees, then $G/K$ is an object of $\bbG_{g,A}$, $E(G/K)=E(G)\backslash E(K)$, and the inclusion $E(G/K)\subset E(G)$ gives rise to a face-embedding which we denote by $f_K\colon \bbR^{E(G/K)}_{\geq0}\to\bbR_{\geq0}^{E(G)}$.
\end{const}

\begin{defi}
    Let $G$ be as above. The \emph{cone of metrics} of $G$ is the poic $\sigma_G$ contained in $\bbR^{E(G)}_{\geq0}$, given by $\bbR^{E(G)}_{>0}$ and all the images $f_K(\bbR^{E(G/k)}_{>0})$ where $K$ is an unmarked subgraph of $G$ whose connected components are non-trivial trees. 
\end{defi}
\begin{remark}
     Recall that a metric on $G$ is a function $\delta\colon E(G)\to \bbR_{>0}$, hence it corresponds uniquely to a point of the underlying polyhedral cone of $\bbR_{>0}^{E(G)}$.
\end{remark}
\begin{example}
    In Figure \ref{fig:cones of metrics} the cones of metrics of a full set of representatives of $\bbG_{1,2}$ are depicted. For the sake of notational simplicity, we do not explicitize the lattice in these cases. Non-present facets are depicted by dashed lines, and non-present vertices are depicted by unfilled points.
    \begin{figure}[htbp]
        \centering
        \begin{tikzpicture}
            \draw[fill=none](-5,5) circle (1);
            \draw[fill=black](-6,5) circle (3pt);
            \draw[fill=black](-4,5) circle (3pt);
            \draw[dashed](-6,5)--(-7,4) node [pos=1.2]{$1$};
            \draw[dashed](-4,5)--(-3,4) node [pos=1.2]{$2$};
            \node[] at (-5,6.2) {$a$};
            \node[] at (-5,3.8) {$b$};

            \draw[line width=2pt](0,3.5)--(0,7) node [midway, left]{$b$};
            \draw[line width=2pt](0,3.5)--(3.5,3.5) node [midway, below] {$a$};
            \draw[fill=gray,opacity=0.3,line width=0pt] (0,3.5)--(0,7)--(3.5,7)--(3.5,3.5);
            \draw[line width=1pt,fill=white](0,3.5) circle (3pt);

            \draw[fill=none](-4,0) circle (1);
            \draw[fill=black](-5,0) circle (3pt);
            \draw[](-5,0)--(-6,0)node [midway, above]{$x$};
            \draw[dashed](-6,0)--(-7,1)node [pos=1.2]{$1$};
            \draw[dashed](-6,0)--(-7,-1)node [pos=1.2]{$2$};
            \node[] at (-2.8,0) {$y$};

            \draw[line width=2pt](0,-1.5)--(0,2)node [midway, left]{$y$};
            \draw[dashed,line width=2pt](0,-1.5)--(3.5,-1.5)node [midway, below]{$x$};
            \draw[fill=gray,opacity=0.3,line width=0pt] (0,-1.5)--(0,2)--(3.5,2)--(3.5,-1.5);
            \draw[line width=1pt,fill=white](0,-1.5) circle (3pt);

            \draw[fill=none](-5,-5) circle (1);
            \draw[fill=black](-6,-5) circle (3pt);
            \draw[dashed](-6,-5)--(-7,-4)node [pos=1.2]{$1$};
            \draw[dashed](-6,-5)--(-7,-6)node [pos=1.2]{$2$};
            \node[] at (-3.8,-5) {$t$};

            \draw[line width=2pt](0,-5)--(3.5,-5)node [midway, below]{$t$};
            \draw[line width=1pt,fill=white](0,-5) circle (3pt);

            \draw (-7.5,8.25)--(4,8.25);
            \draw (-7.5,7.5)--(-1.75,7.5) node [midway, above] {Object of $\bbG_{1,2}$};
            \draw (-1.75,7.5)--(4,7.5) node [midway, above] {Associated cone of metrics};
            \draw (-7.5,2.75)--(4,2.75);
            \draw (-7.5,-3.25)--(4,-3.25);
            \draw (-7.5,8.25)--(-7.5,-6.5);
            \draw (4,8.25)--(4,-6.5);
            \draw (-1.75,8.25)--(-1.75,-6.5);
            \draw (-7.5,-6.5)--(4,-6.5);
        \end{tikzpicture}
        \caption{Examples of cones of metrics.}
        \label{fig:cones of metrics}
    \end{figure}
\end{example}
Let $A$, $g$ and $G$ be as in Construction \ref{const: face-embeddings}, and suppose $G^\prime$ is an additional object of $\bbG_{g,A}$ and $f\in \Hom_{\bbG_{g,A}}(G^\prime,G)$ (so $f$ is a contraction). By definition $f^{-1}(V(G))$ is an unmarked subgraph of $G^\prime$. Let $K$ denote the subgraph of $G^\prime$ obtained from $f^{-1}(V(G))$ by removing the isolated vertices. From Lemma \ref{lem: map of graphs is contraction}, it follows that $G^\prime/K\cong G$, and therefore the contraction $f$ gives rise to:
\begin{itemize}
        \item A face-embedding $\iota_K\colon \sigma_{G^\prime/K}\to \sigma_{G^\prime}$, given by the natural inclusion $E(G^\prime/K)\subset E(G^\prime)$.
        \item A bijection $E(G)\to E(G^\prime/K)$ which gives an isomorphism of poics $L_f\colon \sigma_G\to \sigma_{G^\prime/K}$.
    \end{itemize}
Notice that the composition $\iota_K\circ L_f$ is then a face-embedding, which we denote by $\iota_f\colon\sigma_G\to\sigma_{G^\prime}$. 
\begin{lem} \label{lem: moduli space of n-marked genus g graphs}
    For a non-negative integer $g$ and a finite set $A$ with $2g+\#A-2>0$, the cone of metrics construction gives rise to a functor
    \begin{equation} \Mtrop_{g,A}\colon \bbG_{g,A}^{\op}\to \textnormal{POIC},\end{equation}
    defined as follows: 
    \begin{itemize}
        \item At an object $G$ of $\bbG_{g,A}$, the poic $\Mtrop_{g,A}(G)$ is simply $\sigma_G$,
        \item At a morphism $f\in\Hom_{\bbG_{g,A}}(G^\prime,G)$, the poic morphism $\Mtrop_{g,A}(f)$ is the face-embedding $\iota_f\colon \sigma_G\to \sigma_{G^\prime}$.
    \end{itemize}
    Moreover if $G$ is an object of $\bbG_{g,A}$ and $\tau$ is a face of $\Mtrop_{g,A}(G)$, then there exists up to isomorphism a unique $h\in \Hom_{\bbG_{g,A}}(G,H)$ such that $\Mtrop_{g,A}(h)\colon \sigma_H\to\sigma_G$ is isomorphic to $\tau\leq \sigma_G$.
\end{lem}
\begin{proof}
To show the functoriality, it suffices to show that if $f\colon G^\prime\to G$ and $g\colon G^{\prime\prime}\to G^\prime$ are two contractions, then $\iota_{f\circ g} = \iota_g\circ\iota_f$. Let $K^\prime\subset G$ be the maximal subgraph contained in $f^{-1}(V(G))$ whose connected components are non-trivial trees. Analogously, let $K^{\prime\prime}\subset G$ denote the maximal subgraph contained in $g^{-1}(V(G^\prime))$ whose connected components are non-trivial trees. From Lemma \ref{lem: map of graphs is contraction} it follows that $f$ and $g$ induce respective bijective maps of graphs
\begin{align}
    &G^\prime/K^\prime\to G, &G^{\prime\prime}/K^{\prime\prime}\to G^\prime. \label{eq: induced bijective map of graphs 1}
\end{align}
Analogously, let $H^{\prime\prime}\subset G^{\prime\prime}$ be the maximal subgraph contained in $(f\circ g)^{-1}(V(G))$ whose connected components are non-trivial trees. Hence, $f\circ g$ induces a bijective map of graphs
\begin{equation}
    G^{\prime\prime}/H^{\prime\prime}\to G. \label{eq: induced bijective map of graphs 2}
\end{equation}
There are the following containment relations from above:
\begin{align}
    &E(G^\prime/K^\prime) \subset E(G^\prime), &E(G^{\prime\prime}/H^{\prime\prime})\subset E(G^{\prime\prime}/K^{\prime\prime})\subset E(G^{\prime\prime}). \label{eq: containment relations}
\end{align}
By definition of the face-embeddings $\iota_f$, $\iota_g$, and $\iota_{f\circ g}$, the equation $\iota_{f\circ g}=\iota_g\circ \iota_f$ follows from the following commutative diagram of sets induced by \eqref{eq: containment relations}, and the induced maps between edge sets of \eqref{eq: induced bijective map of graphs 1} and \eqref{eq: induced bijective map of graphs 2}
\begin{equation*}
    \begin{tikzcd}
        E(G) &\arrow{l}{\sim}[swap]{\eqref{eq: induced bijective map of graphs 1}}E(G^\prime/K^\prime)\arrow[hookrightarrow]{d}{\eqref{eq: containment relations}}\\
        E(G^{\prime\prime}/H^{\prime\prime})\arrow{u}{\eqref{eq: induced bijective map of graphs 2}}[swap]{\wr}\arrow[hookrightarrow]{dr}[swap]{\eqref{eq: containment relations}}&E(G^\prime)\\
        &E(G^{\prime\prime}/K^{\prime\prime})\arrow{u}{\wr}[swap]{\eqref{eq: induced bijective map of graphs 1}}
    \end{tikzcd}.
\end{equation*}
For the second statement of the lemma, consider an object $G$ of $\bbG_{g,A}$. A face $\tau\leq \Mtrop_{g,A}(G) = \sigma_G$ corresponds, by definition, to a unmarked subgraph $K$ of $G$ whose components are non-trivial trees. In particular $G/K$ is an object of $\bbG_{g,A}$, and by definition, $\tau\leq\Mtrop_{g,A}(G)$ is given by the contraction $G\to G/K$. The identification is up to isomorphism, as the only possible alternatives are given by a bijective map of graphs of $G$ or $G/K$.
\end{proof}

This section is closed with our definition of the moduli space of genus-$g$ $A$-marked tropical curves. We remark that since $\bbG_{g,A}^{\op}$ has finitely many isomorphism classes, it is necessarily equivalent to a finite category. More interestingly, as small colimits are representable in $\Top$, this implies that the colimit of any functor from $\bbG_{g,A}^{\op}$ to $\Top$ is representable.

\begin{defi}
    The \emph{moduli space of $A$-marked genus-$g$ tropical curves} is the topological space 
    \begin{equation*}
        \cM_{g,A}^\trop \colon = \colim_{\bbG_{g,A}^{\op}} \left|\Mtrop_{g,A}\right|.\end{equation*}    
\end{defi}

After making a choice for a full set of representatives for $\bbG_{g,A}^{\op}$, the above construction shows that $\cM_{g,A}^\trop$ consists of taking the cones of metrics and identifying points representing isometric curves. This presentation is useful for our purposes as it emphasizes the combinatorial relations between cones of metrics that lie at the core of our program: face relations by edge contractions and point identifications due to isometries. 

\section{Poic-complexes and their cycles}
In this section we describe poic-complexes, which are one of the main objects of our interest. As was previously stated, poic-complexes are a small generalization of partially open fans and defined in somewhat categorical terms. We adopt this perspective mainly for two reasons:
\begin{enumerate}
    \item Some objects of interest do not naturally lie in a finite category. Rather, they are naturally presented in a broad manner and hence cleaner to deal with when we do not have to make a choice of isomorphism classes or the like. 
    \item The subsequent section deals with poic-fibrations, which we found simpler and more pleasant to describe in categorical terms. 
\end{enumerate}
This categorical perspective comes at the expense of some notational complexity and technicalities, which, however subtle, are nevertheless relevant. Hence, we dedicate some space to establish these notions and make them explicit, leaving some proofs to the appendix in the interest of readability. We introduce poic-complexes to discuss tropical cycles on them, but this requires the introduction of more structure. More precisely, tropical cycles are not defined on a general poic-complex but rather on a linear poic-complex. In addition, we discuss the pushforward of (weights) cycles through (weakly) proper maps, the moduli space of rational tropical curves as a poic-complex, and we close the section relating these constructions with the classical theory.

\subsection{Poic-complexes.} We regard any finite poset as a category in the usual way. Recall that a category is said to be:
\begin{itemize}
    \item \emph{thin}, if between any two objects there is at most one morphism,
    \item \emph{essentially finite}, if it is equivalent to a category whose class of objects is a finite set.
\end{itemize}
Thus, a finite poset gives a finite and thin category, and any finite and thin category is equivalent to a finite poset.

\begin{defi}
A \emph{poic-complex} $\Phi$ is a functor
\begin{equation*}
    \Phi\colon C_\Phi\to \POIC,
\end{equation*} 
where $C_\Phi$ is an essentially finite thin category, and is subject to:
\begin{enumerate}
    \item The functor $\Phi$ maps morphisms of $C_\Phi$ to face-embeddings.
    \item If $q$ is an object of $C_\Phi$ and $\tau$ is a face of $\Phi(q)$, then there exists a unique up to isomorphism $p\to q$ in $C_\Phi$, such that $\Phi(p\to q)\cong \tau\leq \Phi(q)$. 
\end{enumerate}
\end{defi}

\begin{nota}
    Let $\Phi$ be a poic-complex and $q$ an object of $C_\Phi$. We will denote the dimension $d(\Phi(q))$ simply by $d(q)$. In addition, we avoid cumbersome notation and denote the lattice $N^{\Phi(q)}$ simply by $N^q$ (or by $N_\Phi^q$ to emphasize the domain of $q$). Furthermore, if $p\to q$ is a morphism of $C_\Phi$, then $\Phi(p\to q)$ is poic-morphism $(\Phi(p),N^p)\to(\Phi(q),N^q)$, which in particular carries an integral linear map $N^p\to N^q$. We abuse notation and denote this integral linear map by $\Phi(p\to q)$.
\end{nota}

When no confusion seems near, we will refer to objects of $C_\Phi$ as cones of $\Phi$, and to the morphisms of $C_\Phi$ as morphisms of $\Phi$. We denote the set of isomorphism classes of $C_\Phi$ by $[\Phi]$, and the isomorphism class of a cone $p$ of $\Phi$ by $[p]$. The set of isomorphism classes of $k$-dimensional cones of $\Phi$ is denoted $[\Phi](k)$, and these form a partition of $[\Phi]$. The poic-complex $\Phi$ is called \emph{pure of dimension $n$}, if every cone $p$ of $\Phi$ appears as the source of a morphism to an $n$-dimensional cone of $\Phi$.

\begin{example}\label{ex: poic-complexes coming from a poic}
    Consider a poic $\sigma=(\sigma,N^\sigma)$ and the following finite posets (where the order relation is face inclusion):
    \begin{itemize}
        \item The set $\underline{\sigma}$ of all faces of $\sigma$.
        \item The set $\partial \sigma$ of all proper faces of $\sigma$, namely, $\underline{\sigma}$ without $\sigma$.
        \item The one-point set $\underline{\sigma^o}$ consisting of the interior of $\sigma$.
    \end{itemize}
   Associating each cone to itself provides a poic-complex structure for each. These poic-complexes will be respectively denoted by $\underline{\sigma},\partial\sigma$ and $\underline{\sigma^o}$.
\end{example}
\begin{example}
A relevant and special case of the previous example is the following. Suppose $N$ is a free abelian group of finite rank. The pair $(N_\bbR,N)$ is naturally a poic, and $\{(N_\bbR,N)\}$ is vacuously a poic-complex (the functor being tautological in this case), because the cone $N_\bbR$ has no faces. This poic gives rise to the poic-complex $\underline{N_\bbR}$. If $M$ is an additional free abelian group of finite rank, then it is clear that any integral map $N\to M$ gives rise to a morphism of poic-complexes $\underline{N_\bbR}\to\underline{M_\bbR}$. Furthermore, any morphism of poic-complexes $\underline{N_\bbR}\to\underline{M_\bbR}$ comes from a unique integral map $N\to M$. In addition, by means of our previous identifications we have that $\underline{N_\bbR}\times \underline{M_\bbR} = \underline{(N\oplus M)_\bbR}$.
\end{example}
\begin{example} A natural source of examples are rational fans. Let $N$ be a finite rank lattice and suppose $\Sigma$ is a (partially open) rational polyhedral fan of the vector space $V=N\otimes \bbR$. By definition $\Sigma$ is a poset, and we can then regard this as a poic-complex, where to each cone $\sigma\in\Sigma$ we associate the poic $(\sigma,\textnormal{Lin}(\sigma)\cap N)$. We denote this poic-complex just by $\Sigma$.
\end{example}
\begin{example}\label{ex: m0n is a poic-complex}
    Let $A$ be an arbitrary finite set. We remark that Lemma \ref{lem: moduli space of n-marked genus g graphs} shows that the functor $\Mtrop_{0,A}$ is a poic-complex. 
\end{example}

\begin{const} Suppose $\Phi$ and $\Psi$ are poic-complexes. Since product of poics is a poic, the functors $\Phi$ and $\Psi$ give rise to the functor
\begin{equation*}
{\Phi\times\Psi}\colon C_\Phi\times C_\Psi\to \POIC, (q,s)\mapsto \Phi(q)\times\Psi(s).\end{equation*}
\end{const}

\begin{lem}
Following the notation of the above construction, the product $\Phi\times\Psi$ is a poic-complex.
\end{lem}
\begin{proof}
Observe that $C_{\Phi}\times C_{\Psi}$ is essentially finite, because $C_\Phi$ and $C_\Psi$ are essentially finite. Additionally, the definition of the product category implies that the set of morphisms between any two objects $(q_1,s_1)$ and $(q_2,s_2)$ of $C_\Phi\times C_\Psi$ corresponds to $\Hom_\Phi(q_1,s_2)\times\Hom_\Psi(q_1,s_2)$. Therefore, either it is  empty, or consists of a single element. In particular $C_\Phi\times C_\Psi$ is also thin. Observe that if $\sigma$ and $\xi$ are two poics, then the faces of $\sigma\times\xi$ consist of products of faces of both $\sigma$ and $\xi$. Moreover, the product of any two faces of $\sigma$ and $\xi$ gives a face of $\sigma\times\xi$. From this, both conditions on poic-complexes follow directly. 
\end{proof}

Let $\Phi$ denote a poic-complex. There are the following full subcategories of $C_\Phi$:
\begin{enumerate}
    \item For a cone $p$ of $\Phi$, the category $C_\Phi(p,)$ consisting of morphisms $p\to q$ of $C_\Phi$ where maps are commuting triangles. Since $C_\Phi$ is thin, it follows that this is a full subcategory.
    \item For a non-negative integer $k$, the full subcategory $C_\Phi(\leq k)$ of objects with cones of dimension at most $k$.
\end{enumerate}

Let $\Phi(\leq k)$ denote the restriction of $\Phi$ to $C_\Phi(\leq k)$.
\begin{remark}
    In general $\Phi |_{C_\Phi(p,)}$ is not a poic-complex.
\end{remark}
\begin{lem}\label{lem: natural poic subcomplexes}
    Let $\Phi$ be a poic-complex, and let $k\geq0$ be an integer. The functor ${\Phi(\leq k)}$ is a poic-complex. 
\end{lem}
\begin{proof}
    Since $C_\Phi(\leq k)$ is a full subcategory of an essentially finite thin category, then it must also be essentially finite and thin. The additional conditions of a poic-complex follow directly from those of $\Phi$. 
\end{proof}

\begin{defi}
    A poic-complex $\Upsilon$ is a \emph{poic-subcomplex} of $\Phi$, if $C_\Upsilon$ is a full subcategory of $C_\Phi$ and $\Upsilon=\Phi|_{C_\Upsilon}$.
\end{defi}

\begin{remark}
    By definition, in the situation of Lemma \ref{lem: natural poic subcomplexes}, the poic-complex $\Phi(\leq k)$ is a subcomplex of $\Phi$.
\end{remark}

\begin{defi}
Suppose $\Phi$ and $\Psi$ are poic-complexes. A \emph{morphism of poic-complexes} $F\colon \Phi\to\Psi$ consists of the following data:
\begin{itemize}
    \item A functor $F\colon C_\Phi\to C_\Psi$.
    \item A natural transformation $\eta_F\colon \Phi\implies \Psi\circ F$.
\end{itemize}
These are subject to the additional condition: for a cone $p$ of $\Phi$, the image $\eta_{F,p}\left(\Phi(p)\right)$ does not lie in a proper face of $\Psi(F(p))$.\\
Composition of morphisms of poic-complexes is defined as follows: let $\Omega$ denote an additional poic-complex and $G\colon \Psi\to\Omega$ a morphism. Then $G\circ F\colon \Phi\to \Omega$ is given by:
\begin{itemize}
    \item The functor $G\circ F\colon C_\Phi\to C_\Omega$.
    \item The natural transformation $\eta_{G\circ F}\colon \Phi\implies {\Omega}\circ (G\circ F)$, which is defined at a cone $s$ of $\Phi$ by
\begin{equation*}
    \eta_{G\circ F,s} = \eta_{G,F(s)}\circ \eta_{F,s}.
\end{equation*}
\end{itemize}
\end{defi}
\begin{nota}
    Let $F\colon\Phi\to\Psi$ be a morphism of poic-complexes and let $s$ be a cone of $\Phi$. To simplify notation, we will omit $\eta_F$ and denote the poic-morphism $\eta_{F,s}$ simply by $F_s$.
\end{nota}
The data of poic-complexes with their morphisms and the compositions thereof forms a category. We denote this category by $\POICCmplxs$. The realization functor \eqref{eqdefi: realization of poics} on poics can be extended naturally to poic-complexes. Namely, let $\Phi$ be a poic-complex. Since $C_\Phi$ is essentially finite, the colimit of $|\bullet|\circ \Phi$ is representable in $\Top$, and hence gives rise to a space $|\Phi|$. Properties of colimits show that from a morphism of poics $F\colon \Phi\to\Psi$, we obtain a continuous map $|F|\colon |\Phi|\to |\Psi|$ and this association gives rise to a functor
\begin{equation}
    |\bullet|\colon  \POICCmplxs\to\Top. \label{eqdefi: realization of poiccomplexes}
\end{equation}

\subsection{Linear poic-complexes and Minkowski weights.}\label{ssec: lin pcmplxs mkwsk wghts} 
\begin{defi}\label{defi: general weights}
Let $\Phi$ be a poic-complex and $k\geq0$ an integer. A \emph{$k$-dimensional weight $\omega$ on $\Phi$} is a function
\begin{equation*}
    \omega\colon  [\Phi](k)\to \bbZ.
\end{equation*}
Since $[\Phi](k)$ is a finite set, the $k$-dimensional weights on $\Phi$ form a free abelian group of finite rank, which we denote by $W_k(\Phi)$.  
\end{defi}

\begin{const}
    Suppose $\Upsilon$ is a subcomplex of a poic-complex $\Phi$. The natural inclusion of $C_\Upsilon$ into $C_\Phi$ and the identity morphisms at cones of $\Upsilon$ give rise to a poic-morphism $I\colon \Upsilon\to\Phi$. Since $\Upsilon$ is a poic-subcompex of $\Phi$, any $\omega\in W_k(\Upsilon)$ can be extended by zero to a weight on $\Phi$, namely  
    \begin{equation*} I_!\omega\colon  [\Phi](k)\to\bbZ, \sigma\mapsto \begin{cases}\omega([s]),&\textnormal{ if }[s]\in[\Upsilon](k),\\
        0, &\text{otherwise}.
    \end{cases}\end{equation*}
\end{const}

\begin{lem}\label{lem: extension by zero weights}
    Following the notation from the construction above, extension by zero gives an injective linear map
    \begin{equation*} I_!\colon W_k(\Upsilon)\to W_k(\Phi).\end{equation*}
\end{lem}

\begin{proof}
    Linearity of the map is clear. Since $[\Upsilon](k)\subset [\Phi](k)$, the injectivity follows directly.
\end{proof}

\begin{defi}
A \emph{linear poic-complex} $\Phi_X$ is a tuple $(\Phi,X)$ such that
\begin{itemize}
    \item $\Phi$ is a poic-complex and 
    \item $X\colon \Phi\to \underline{(N_X)_\bbR}$ is a morphism of poic-complexes, where $N_X$ is a free abelian group of finite rank.
\end{itemize}
\end{defi}

Clearly, if $\Phi_X$ is a linear poic-complex and $\Upsilon$ is a poic-subcomplex of $\Phi$, then $\Upsilon_{X|_\Upsilon}=(\Upsilon,X|_\Upsilon)$ is also a linear poic-complex. In this case, we say that $\Upsilon_{X|_\Upsilon}$ is a \emph{linear poic-subcomplex} of $\Phi_X$.

\begin{const}
Suppose $\Phi_X$ and $\Psi_Y$ are linear poic-complexes. Both morphisms $X\colon \Phi\to \underline{(N_X)_\bbR}$ and $Y\colon \Psi\to \underline{(N_Y)_\bbR}$ induce a morphism
\begin{equation*}
X\times Y\colon \Phi\times\Psi \to \underline{(N_X)_\bbR}\times \underline{(N_Y)_\bbR}=\underline{(N_X\oplus N_Y)_\bbR}.     
\end{equation*}
In particular, $\left(\Phi\times\Psi\right)_{X\times Y}$ is a linear poic-complex, where $N_{X\times Y} := N_X\oplus N_Y$.
\end{const}

\begin{defi}
A \emph{morphism of linear poic-complexes} $\phi\colon\Phi_X\to \Psi_Y$ consists of the following data:
\begin{itemize}
    \item A morphism of poic-complexes $\phi\colon \Phi\to\Psi$.
    \item An integral linear map $\phi_\tint \colon N_X\to N_Y$.
\end{itemize}
These are subject to the condition of yielding a commutative square. The composition of morphisms of linear poic-complexes is defined in the obvious way. Linear poic-complexes and their morphisms give rise to a category, which we denote by $\LinPOICCmplxs$.
\end{defi}

\begin{nota}
    If $\Phi$ is a poic-complex and $s$ is a cone of $\Phi$ with $d(s)=k-1$, where $k>0$ is an integer. We denote the set $[\Phi(s,)](k)$ by $\Star^1_\Phi(s)$. If $s\to t$ is an object of $[\Phi(s,)](k)$, then we will denote its isomorphism class in $\Star^1_\Phi(s)$ by $[s\to t]$.
\end{nota}

\begin{const}[Relative linear maps]\label{const: rel lin maps}

Suppose $\Phi_X$ is a linear poic-complex and consider a morphism $s\to t$ of $\Phi$, with $\Phi(s)\leq \Phi(t)$ a codimension $1$ face. 
The quotient $N^t/N^s$ is a rank $1$ lattice and the morphism $X$ induces a linear map
\begin{equation} 
X_{s\to t}\colon  N^t/N^s\to N_X/X_s(N^{s}).\label{eq: relative linear maps}
\end{equation}
\end{const}

\begin{lem}
   Let $\Phi_X$ and $s\to t$ be as in Construction \ref{const: rel lin maps}. If $t^\prime$ is an object of $\Phi$ isomorphic to $t$, then clearly there is a unique morphism $s\to t^\prime$. Under the maps $X_{s\to t}$ and $X_{t^\prime/s}$, the images of the generators of $N^t/N^s$ and $N^{t^\prime}/N^s$ lying correspondingly in the projections of $N^t\cap \Phi(t)$ and $N^{t^\prime}\cap\Phi(t^\prime)$ coincide.
\end{lem}
\begin{proof}
    Let $t\to t^\prime$ denote the isomorphism. The unique morphism $s\to t^\prime$ is then given by the composition $s\to t\to t^\prime$.
    This gives rise to the commutative diagram of abelian groups
    \begin{equation*}
    \begin{tikzcd}
        N^t\arrow{rr}{\Phi(t\to t^\prime)}[swap]{\sim} \arrow[hookleftarrow]{dr}[swap]{\Phi(s\to t)}&&N^{t^\prime}\\
        &N^s\arrow[hookrightarrow]{ur}[swap]{\Phi(s\to t^{\prime})}&
    \end{tikzcd},
    \end{equation*}
    and therefore the isomorphism $\Phi(t\to t^\prime)$ induces an isomorphism $N^t/N^s\to N^{t^\prime}/N^s$. Observe that under $\Phi(t\to t^\prime)$ the underlying cone of $\Phi(t)\subset \left(N^t\right)_\bbR$ is mapped bijectively onto the underlying cone of $\Phi(t^\prime)\subset \left(N^{t^\prime}\right)_\bbR$. In particular, the generator of $N^t/N^s$ pointing towards $\Phi(t)$ is mapped to the generator of $N^{t^\prime}/N^s$ pointing towards $\Phi(t^\prime)$. The following commutative diagram 
    \begin{equation*}
    \begin{tikzcd}
        N^t/N^s\arrow{dr}[swap]{X_{s\to t}} \arrow{rr}{\sim}&&N^{t^\prime}/N^s\arrow{dl}{X_{t^\prime/s}}\\
        &N/X_s(N^s)&
    \end{tikzcd}
    \end{equation*}
    shows that the images of these generators coincide in $N_X/X_s(N^s)$.
\end{proof}

\begin{defi}
Following the notation of the previous construction and the result of the previous lemma, for an isomorphism class $[s\to t]\in \Star^1_\Phi(s)$ the \emph{normal vector $u^X_{[s\to t]}$} is defined as the image of the generator of $N^t/N^s$ contained in (the projection of) $N^t\cap \Phi(t)$ under the linear map $X_{s\to t}$, where $s\to t$ is any object representing this isomorphism class. 
\end{defi}

\begin{defi}\label{defi: balanced weight}
Suppose $\Phi_X$ is a linear poic-complex and let $s$ be a cone of $\Phi$ with $d(s)=k-1$. A $k$-dimensional weight $\omega\in W_k(\Phi)$ is \emph{$X$-balanced at $s$}, if 
\begin{equation*}
    \sum_{[s\to t]\in\Star^1_\Phi(s)} \omega([t]) u_{[s\to t]}^X =0.
\end{equation*}
It is said to be \emph{$X$-balanced}, if it is $X$-balanced at every object of $\Phi$ of dimension $k-1$.
\end{defi}

\begin{lem}\label{lem: balancing respects isomorphisms}
   Let $\Phi_X$ be a linear poic-complex, with $s^\prime\to s$ an isomorphism of $\Phi$ and $d(s)=k-1$. A weight $\omega\in W_k(\Phi)$ is $X$-balanced at $s$ if and only if $\omega$ is $X$-balanced at $s^\prime$. 
\end{lem}
\begin{proof}
    The isomorphism $s^\prime\to s$ gives rise to an isomorphism of categories $\Phi(s^\prime,)\to\Phi(s,)$, and in particular to a bijection $\Star^1_\Phi(s^\prime)\to \Star^1_\Phi(s)$. More interestingly, it also gives rise to the commutative diagram
    \begin{equation*}
        \begin{tikzcd}
            N^{s^\prime} \arrow{rr}{\Phi(s^\prime\to s)}[swap]{\sim}\arrow{dr}[swap]{X_{s^\prime}}&& N^{s}\arrow{dl}{X_s}\\
            &N_X&
        \end{tikzcd}.
    \end{equation*}
    This diagram shows that $N/X_s(N^s)= N/X_{s^\prime}(N^{s^\prime})$. Furthermore, if $s\to t$ is a morphism of $\Phi$ with $d(t)=k$, then $\Phi(s\to s^\prime)$ gives an isomorphism between $N^t/N^{s^\prime}$ and $N^t/N^{s}$. In addition, this isomorphism commutes with the maps induced by $X_t$:
    \begin{align*}
       N^t/N^{s^\prime}&\to  N/X_{s^\prime}(N^{s^\prime})=N^t/N^{s},\\
        N^t/N^{s}&\to N/X_s(N^s)= N/X_{s^\prime}(N^{s^\prime}).
    \end{align*}
    Hence, $u^X_{[s\to t]} = u^X_{[s^\prime\to t]}$, and it is then immediate that a weight $\omega\in W_k(\Phi)$ is $X$-balanced at $s$ if and only if it is $X$-balanced at $s^\prime$.
\end{proof}
\begin{defi}
    For a linear poic-complex $\Phi_X$, the set of \emph{$k$-dimensional Minkowski weights} is the subset $M_k(\Phi_X)$ of $W_k(\Phi)$ consisting of the $X$-balanced weights $\omega\in W_k(\Phi)$. 
\end{defi}

\begin{remark}
    Let $k\geq 1$ be an integer. From Lemma \ref{lem: balancing respects isomorphisms}, it follows that a weight is $X$-balanced if and only if it is $X$-balanced at a set of representatives of $[\Phi](k-1)$.
\end{remark}

\begin{lem}\label{lem: ext by zero minkowski weights}
    Suppose $\Phi_X$ is a linear poic-complex. The set $M_k(\Phi_X)$ is an abelian subgroup of $W_k(\Phi)$. Furthermore, if $\Upsilon$ is a poic-subcomplex of $\Phi$, with $I\colon \Upsilon\to \Phi$ the natural inclusion, then the map $I_!$ of Lemma \ref{lem: extension by zero weights} restricts to an injective linear map
    \begin{equation*}
        I_!\colon M_k(\Upsilon_{X|_\Upsilon})\to M_k(\Phi_X),
    \end{equation*}
    for every non-negative integer $k$.
\end{lem}

\begin{proof}
    The first statement is immediate. For the second, let $k\geq0$ be an integer. It suffices to show that the image of a Minkowski weight of $\Upsilon_{X|_\Upsilon}$ is a Minkowski weight of $\Phi_X$, because then linearity and injectivity follow from Lemma \ref{lem: extension by zero weights}. Let $\omega\in M_k(\Upsilon_{X|_\Upsilon})$ and consider a cone $s$ of $\Phi$ with $d(s)=k-1$.  Observe that if $s$ is not a cone of $\Upsilon$, then there is nothing to check as $I_!\omega$ would be trivial in $\Star^1_\Phi(s)$. So, suppose that $s$ is also a cone of $\Upsilon$. Clearly, if $[s\to t]\in \Star^1_\Upsilon(s)$, then $u^{X|_\Upsilon}_{[s\to t]}=u^X_{[s\to t]}$. As $I_!$ is extension by zero, it follows that
    \begin{align*}
        \sum_{[s\to t]\in\Star^1_\Phi(s)} I_!\omega([t]) u^X_{[s\to t]} &=\sum_{[s\to t]\in\Star^1_\Upsilon(s)}\omega([t]) u^X_{[s\to t]}.
    \end{align*}
    This sum vanishes, because $\omega\in M_k(\Upsilon_{X|_\Upsilon})$, and hence $I_!\omega\in M_k(\Phi_X)$.
\end{proof}

\begin{const}[Cross product]\label{const: cross products}
    Suppose $\Phi_X$ and $\Psi_Y$ are linear poic-complexes. Observe that for any $k\geq0$, the following decomposition holds
    \begin{equation*}
        [\Phi\times \Psi] (k)= \bigcup_{i+j=k}[\Phi](i)\times [\Psi](j).
    \end{equation*}
    Let $i$ and $j$ be such that $k=i+j$. Any $\omega\in M_i(\Phi_X)$ and $\eta\in M_j(\Psi_Y)$ give rise to a map 
    \begin{equation*}
        \omega\times\eta\colon  [\Phi](i)\times[\Psi](j)\to \bbZ, (s,t)\mapsto \omega(s)\cdot\eta(t),
    \end{equation*}
    which after extending by zero gives a map
    \begin{equation*}
        \omega\times\eta\colon  [\Phi\times\Psi ](k)\to \bbZ.
    \end{equation*}
\end{const}
\begin{lem}\label{lem: cross products}
    Following the notation of Construction \ref{const: cross products}: If $\omega\in M_i(\Phi_X)$ and $\eta\in M_j(\Psi_Y)$, then $\omega\times\eta \in M_{i+j}((\Phi\times\Psi)_{X\times Y})$. In addition, the cross product is a bilinear map
    \begin{equation*}
    M_i(\Phi_X)\otimes_\bbZ M_j(\Psi_Y) \to M_{k}((\Phi\times\Psi)_{X\times Y}), \omega\otimes \eta\mapsto \omega\times\eta.
    \end{equation*}
\end{lem}
\begin{proof}
    First, notice that for any cone $s\times t$ of $\Phi\times\Psi$, we have that $N^{s\times t} = N^s\oplus N^t$ and the linear map 
    \begin{equation*}
    (X\times Y)_{s\times t}\colon N^{s\times t}\to N_X\oplus N_Y=N_{X\times Y} 
    \end{equation*}
    is the direct sum of $X_s$ and $Y_t$. Now, any cone of $\Phi\times\Psi$ of dimension $(i+j-1)$ is of the form $s\times t$, where:
    \begin{itemize}
        \item $s$ is a cone of $\Phi$. 
        \item $t$ is a cone of $\Psi$.
        \item Their dimensions add up to $i+j-1$.
    \end{itemize} 
    The definition of $\omega\times\eta$ implies that the only balancing to check is at cones of $\Phi\times \Psi$ that are of the form:
    \begin{enumerate}
        \item $p\times t$, where $p$ is a cone of $\Phi$ of dimension $i-1$ and $t$ is a cone of $\Psi$ of dimension $j$.
        \item $s\times q$, where $s$ is a cone of $\Phi$ of dimension $i$ and $q$ is a cone of $\Psi$ of dimension $j-1$.
    \end{enumerate} 
    Both situations are analogous, so only the first one will be explicitly carried out. Observe that
    \begin{equation*}
        \{ [p\times t\to s\times t] \colon  [p\to s]\in \Star^1_\Phi(p)\}\subset \Star^1_{\Phi\times\Psi}(p\times t)
    \end{equation*}
    is the only subset of $\Star^1_{\Phi\times\Psi}(p\times t)$ where $\omega\times \eta$ does not vanish. So for $[p\times t\to s\times t]\in\Star^1_{\Phi\times\Psi}(p\times t)$ the relative linear map
    \begin{equation*}
        (X\times Y)_{ p\times t\to s\times t}\colon N^{s\times t}/N^{p\times t}\to (N_{X\times Y})/(X\times Y)_{p\times t}(N^{p\times t})
    \end{equation*}
    is simply
    \begin{equation*}
        \begin{pmatrix}
        X_{p \to s}\\
        0
    \end{pmatrix}\colon  N^s/N^p \to N_X/X_p(N^p)\oplus N_Y/Y_t(N^t)=(N_{X\times Y})/(X\times Y)_{p\times t}(N^{p\times t}).
    \end{equation*}
    The normal vector $u_{[p\times t\to s\times t]}^{X\times Y} \in N_{X\times Y}/(X\times Y)_{p\times t}(N^{p\times t})$ is consequently just $\begin{pmatrix}
        u_{[p\to s]}^X \\
        0 
    \end{pmatrix}$, and 
    \begin{align*}
        \sum_{[p\times t\to s\times t]\in \Star^1_{\Phi\times\Psi}(p\times t)} (\omega\times\eta)([p\times t\to s\times t]) u_{[p\times t\to s\times t]}^{X\times Y} &= \sum_{[p\to s]\in\Star^1_\Phi(p)} \eta([t])\omega([s])u_{[p\times t\to s\times t]}^{X\times Y}\\
        &=\eta([t])\begin{pmatrix}
           \displaystyle \sum_{[p\to s]\in\Star^1_\Phi(p)} \omega([p\to s])u_{[p\to s]}^{X}\\
            0
        \end{pmatrix}\\
        &= \begin{pmatrix}
            0 \\
            0
        \end{pmatrix}.
    \end{align*}
    This shows that $\omega\times\eta\in M_k(\Phi\times\Psi_{X\times Y})$, whereas the bilinearity directly follows from the definition.
\end{proof}

\subsection{Moduli space of rational tropical curves.} \label{ssec: mod spaces of rtc}
We expand on example \ref{ex: m0n is a poic-complex}. Let $X$ be a set with $\#A\geq 3$. We explain, as in \cite{GathmannKerberMarkwigTFMSTC}, how the topological space $\cM_{0,A}^{\trop}$ can be embedded as an $(\#A-3)$-dimensional fan through the distance map. A point in this space consists of an $X$-marked metric tree $\Gamma$.
\begin{nota}
Given a tree $T$ of $\bbG_{0,A}$ and a $\delta\in \sigma_T$, we let $|(T,\delta)|$ denote the metric tree defined by $T$ with the lengths prescribed by $\delta$ (where length $0$ means that the corresponding edge is contracted). 
\end{nota}
Let $\binom{A}{2}$ denote the set of $2$ element subsets of $A$ and consider the vector space $\bbR^{\binom{A}{2}}$, together with the linear map
\begin{equation*}
    M_A\colon \bbR^A\to \bbR^{\binom{A}{2}},(x_i)_{i\in A}\mapsto (x_i+x_j)_{\{i,j\}\subset A}.
\end{equation*}
We denote the cokernel of $M_A$ by $Q_A$. It is shown in Theorem 4.2 of \cite{SpeyerSturmfels} that the distance map 
\begin{equation}
    \dist\colon\cM^{\trop}_{0,A}\to Q_A, \Gamma\mapsto (\dist_\Gamma(\ell_i(\Gamma)),\ell_j(\Gamma))_{\{i,j\}\subset A}\label{eq: distance maps}
\end{equation}
embeds the space $\cM^{\trop}_{0,A}$ as a simplicial fan of pure dimension $(\#A-3)$, where the $k$-dimesional cones are given by the combinatorial types. These combinatorial types are the same (by definition) as the isomorphism classes of our category $\bbG_{0,A}$. Furthermore, this fan is rational with respect to the lattice $N_{\dist(A)}$ generated by the vectors $v_I$, where $I\subset A$ with $1<\#I<\#A-1$, given by the image of the metric tree with a unique bounded edge of length $1$ having the marked legs with labels in $I$ on one side of the edge and the marked legs with labels in $A\backslash I$ on the other. In our terminology, the distance map \eqref{eq: distance maps} defines a morphism of poic-complexes 
\begin{equation}
    \dist_A\colon \Mtrop_{0,A}\to \underline{\left(N_{\dist_A}\right)_\bbR} \label{eq: distance morphism of poic-complexes}
\end{equation} given by the trivial functor and the natural transformation $\eta_{\dist_A}$ defined at a tree $T$ of $\bbG_{0,A}$ by
\begin{equation*}
    \eta_{\dist_{A},T}\colon \sigma_T\to \left(N_{\dist_A}\right)_\bbR, \delta\mapsto \left(\dist_{|(T,\delta)|}(\ell_i(|(T,\delta)|),\ell_j(|(T,\delta)|)\right)_{\{i,j\}\subset A}.
\end{equation*}
\begin{defi}
The morphism \eqref{eq: distance morphism of poic-complexes} defines a linear poic-complex structure on $\Mtrop_{0,A}$. We call this linear poic-complex \emph{the linear poic-complex of $A$-marked trees}. Since we will always consider $\Mtrop_{0,A}$ as a linear poic-complex in this way, we will avoid cumbersome subscript notation, and simply denote this linear poic-complex by $\Mtrop_{0,A}$.
\end{defi}
\begin{example}
    Suppose $\Sigma$ is a rational fan of the real vector space $V$ with respect to an underlying lattice $N$. As has been previously explained, the fan $\Sigma$ gives rise to a poic-complex that we denote in the same way. The natural inclusions $\sigma\subset N_\bbR$ at cones of the fan give morphisms of poics $(\sigma,N\cap\sigma)\to (N_\bbR,N)$, which together with the constant functor $\Sigma \to \underline{N_\bbR}$ define a morphism of poic-complexes $\Sigma\to \underline{N_\bbR}$. In particular, the poic-complex $\Sigma$ defined by the fan is naturally a linear poic-complex. The Minkowski weights of this linear poic-complex correspond to the Minkowski weights of the fan as defined in \cite{FultonSturmfelsITTV}.
\end{example}

\subsection{Subdivisions and tropical cycles.} Subdivisions are a special kind of morphisms of poic-complexes used to define tropical cycles on a linear poic-complex. For subsequent constructions the following notions are useful. Let $N$ be a free abelian group of finite rank. 
\begin{itemize}
    \item A closed convex polyhedral cone $\sigma$ of $N_\bbR$ is called \emph{simplicial} if $\dim \sigma$ coincides with the number of $1$-dimensional faces.
    \item A partially open polyhedral cone of $N_\bbR$ is called \emph{simplicial} if its closure is simplicial.
    \item A poic-complex is \emph{simplicial} if the underlying partially open polyhedral cone of every cone is simplicial.
\end{itemize}
\begin{defi}\label{defi: subdivision}
Let $\Phi$ be a poic-complex. A \emph{subdivision} $S\colon \Upsilon\to\Phi$ is a morphism of poic-complexes $S\colon \Upsilon\to\Phi$, such that: 
\begin{itemize}
    \item The underlying functor $S\colon C_\Upsilon\to C_\Phi$ is essentially surjective, and the underlying integral linear maps of the natural transformation are injective. Furthermore, if $t$ is a cone of $\Upsilon$ such that $d(t) = d(S(t))$, then the underlying integral linear map of $S_t\colon \Upsilon(t)\to \Phi(S(t))$ is an isomorphism.
    \item For every cone $s$ of $\Phi$, the relative interior of $\Phi(s)$ is the disjoint union of the images of the relative interiors of the cones $p$ of $\Upsilon$ with $S(p)\cong s$, and the images of these relative interiors characterize these cones of $\Upsilon$ up to isomorphism (i. e. if they coincide, then they are isomorphic in $C_\Upsilon$).
    \item For every cone $t$ of $\Upsilon$ and every $s\to S(t)$, there exists $q\to t$ in $\Upsilon$ with $S(q\to t) \cong s\to S(t)$.
\end{itemize}
\end{defi}
Let $k\geq 0$ be an integer. A subdivision $S\colon \Upsilon\to\Phi$ gives rise to the map
    \begin{equation} 
    S^*\colon W_k(\Phi)\to W_k(\Upsilon), \omega\mapsto S^*\omega,\label{eq: subdivision map}\end{equation}
    where
    \begin{equation*}
    S^*\omega\colon  [\Upsilon](k)\to\bbZ, [p]\mapsto \begin{cases}
        \omega([S(p)]),& \textnormal{ if }[S(p)]\in [\Phi](k),\\
        0,& \text{otherwise}.
    \end{cases}
    \end{equation*}
\begin{lem}
    The map \eqref{eq: subdivision map} is injective and linear.
\end{lem}
\begin{proof}
    Linearity is clear from the definition, so we focus on injectivity. Suppose $\omega\in W_k(\Phi)$ is such that $S^*\omega = 0$ and let $[s]\in [\Phi](k)$. Because of essential surjectivity, there exists an object $p$ of $\Upsilon$ such that $[S(p)]= [s]$. Observe that $d(p)\leq k$, because the underlying linear map of $S_p$ is injective. The second condition of Definition \ref{defi: subdivision} shows that if $d(p)<k$, then there exists a morphism $p\to q$ of $\Upsilon$ such that $d(q)=k$, and $[S(q)] = [s]$. Now, $S^*\omega = 0$ implies that $\omega([s])=\omega([S(q)]) = 0$. Since $[s]\in[\Phi](k)$ was arbitrary, it follows that $\omega= 0$ and hence the map is injective.
\end{proof}

\begin{lem}
Suppose $\Phi_X$ is a linear poic-complex, and let $S\colon \Upsilon\to\Phi$ be a subdivision. Then $S^*\colon W_k(\Phi)\to W_k(\Upsilon)$ restricts to an injective linear map 
\begin{equation*}
S^*\colon M_k(\Phi_X)\to M_k(\Upsilon_{X\circ S}).
\end{equation*}
\end{lem} 
\begin{proof}
Since $S^*\colon W_k(\Phi)\to W_k(\Upsilon)$ is injective and linear, it is only necessary to show that if $\omega\in M_k(\Phi_X)$, then $S^*\omega\in M_k(\Upsilon_{X\circ S})$. Consider a cone $p$ of $\Upsilon$ with $d(p)=k-1$. Then $d(S(p))\geq k-1$, and we have the following two situations: Either $d(S(p))=k-1$, or $d(S(p))\geq k$.
\begin{itemize}
\item Suppose first that $d(S(p))=k-1$. Now, if $p\to q$ represents a class of $\Star^1_\Upsilon(p)$, then necessarily $d(S(q))=k$ and hence $S(p\to q)$ represents a class of $\Star^1_\Upsilon(S(p))$. Therefore, the functor $S$ gives a map
\begin{equation}
    S\colon \Star^1_{\Upsilon}(p)\to\Star^1_\Phi(S(p)), [p\to q]\mapsto [S(p)\to S(q)].\label{eq: map between stars is bijection}
\end{equation}
We claim that this map is a bijection. Surjectivity follows from the third condition, while injectivity is more delicate but follows from the second condition. Namely, suppose that $p\to q$ and $p\to q^\prime$ are in $\Upsilon$ and represent classes of $\Star^1_\Upsilon(p)$, with $[S(p)\to S(q)]=[S(p)\to S(q^\prime)]$. Then there exists a cone $s$ of $\Phi$ with $d(s)=k$ and $S(q)\cong s\cong S(q^\prime)$. Now, the images of their relative interiors in $\Phi(s)^0$ are either disjoint or the same. If they were disjoint, then the closures (relative to $\Phi(s)$) of these relative interiors would consist of two $k$-dimensional cones that intersect at two proper faces: a $(k-1)$-dimensional one lying in the boundary of the closure of $\Phi(s)^0$, and the other lying in the interior. Therefore, they are necessarily the same, and it follows from the second condition, that they must also be isomorphic in $\Upsilon$. Hence, \eqref{eq: map between stars is bijection} is a bijection.\\
From the latter, it is clear that if $[p\to q]\in\Star^1_{\Upsilon}(p)$, then $\omega([S(p)\to S(q)]) = S^*\omega([p\to q])$. In addition, the subdivision morphism induces an isomorphism between the corresponding lattices and the associated normal vectors coincide. Hence, balancing around $p$ follows from balancing around $S(p)$.
\item Suppose now that $d(S(p))> k-1$. Then either the dimension increases by $1$ or by more. If $d(S(p)) = k$, then $\Upsilon (p)$ is a $(k-1)$-dimensional cone subdividing a $k$-dimensional cone. The second condition shows that $\Star^1_{\Upsilon}(p)$ consists of two classes, and the third condition implies that these have opposite normal vectors. In this case, $S^*\omega$ is a constant function on $\Star^1_{\Upsilon}(p)$ and the balancing condition holds around such $p$. If $d(S(\tau))>k$ then, by construction, $S^*\omega$ is trivially zero on $\Star^1_{\Upsilon}(p)$ and balancing around $p$ trivially holds.
\end{itemize}
\end{proof}

\begin{defi}
Let $\Phi$ be a poic-complex. The \emph{category of subdivisions of $\Phi$} is the category $\Subd(\Phi)$ whose class of objects consists of subdivisions $S\colon \Upsilon\to \Phi$ and morphisms are commuting subdivision morphisms. In other words, a morphism $Q\colon S^\prime\to S$, where $S^\prime\colon \Upsilon^\prime\to\Phi$ and $S\colon \Upsilon\to \Phi$ are subdivisions, is a subdivision $Q\colon \Upsilon^\prime\to\Upsilon$ such that $S\circ Q=S^\prime$.
\end{defi}
\begin{restatable}[]{lem}{sbdvs}\label{lem: subdivisions are essentially small for simplicial}
   Any poic-complex $\Phi$ has a subdivision $S\colon \Phi^\prime\to \Phi$ such that $\Subd(\Phi^\prime)$ is essentially small.
\end{restatable}
\begin{proof}
    We postpone the proof of this lemma to \ref{appendix: proofs of lemmas} in the appendix.
\end{proof}
\begin{lem}
    Suppose $\Phi_X$ is a linear poic-complex, and let $k\geq0$ be an integer. Minkowski weights give rise to a functor
    \begin{equation} M_k(\bullet)\colon  \Subd(\Phi)^{\op} \to \textnormal{Ab},S\colon \Upsilon\to\Phi \mapsto M_k(\Upsilon_{X\circ S}),\label{eq: minkowski weights functor}\end{equation}
    whose colimit is representable.
\end{lem}
\begin{proof}
    It follows from Lemma \ref{lem: subdivisions are essentially small for simplicial} that there is a subdivision $S\colon\Phi^\prime\to\Phi$ such that $\Subd(\Phi^\prime)$ is an essentially small category. The inclusion functor $\Subd(\Phi^\prime)\to\Subd(\Phi)$ is final, since any subdivision of $\Phi$ can be further subdivided to a subdivision of $\Phi^\prime$. Therefore, this colimit can be computed at $\Subd(\Phi^\prime)$ and at an equivalent small category thereof. Since small colimits of abelian groups are representable, it follows that the colimit of $M_k(\bullet)$ is representable.
\end{proof}

\begin{defi}
For a linear poic-complex $\Phi_X$, the \emph{group of tropical $k$-cycles} $Z_k(\Phi_X)$ is defined as the colimit of \eqref{eq: minkowski weights functor}. 
\end{defi}

\begin{lem}
    Suppose $\Upsilon$ is a poic-subcomplex of $\Phi$, and $\Phi_X$ is a linear poic-complex. Let $I\colon\Upsilon\to\Phi$ denote the natural inclusion. Then extension by zero gives rise to an injective map
    \begin{equation*}
        I_!\colon  Z_k(\Upsilon_{X|_\Upsilon})\to Z_k(\Phi_X).
    \end{equation*}
\end{lem}
\begin{proof}
    This follows directly from Lemma \ref{lem: ext by zero minkowski weights} and the definition of tropical cycles.
\end{proof}
\begin{nota}
    We say that a linear poic-complex $\Phi_X$ is pure of dimension $n$ if the poic-complex $\Phi$ is pure of dimension $n$.
\end{nota}
\begin{defi}
A linear poic-complex $\Phi_X$ pure of dimension $n$ is called \emph{irreducible} if $Z_n(\Phi_X)$ is free of rank $1$.
\end{defi}

\begin{example}\label{ex: open cone is irred}
    Suppose $\sigma$ is a poic. Any linear poic-complex structure on $\underline{\sigma^0}$ is irreducible. Indeed, at any subdivision the star of a codimension-$1$ cone is of size two and the corresponding normal vectors are related by a change of sign. Analogously, any complete rational fan is irreducible. Let $N$ be a free abelian group of finite rank and let $\Phi$ be a complete rational fan in $N\otimes_\bbZ\bbR$. Then $\Phi$ is a subdivision of the poic $\underline{N_\bbR}$. Since $\underline{N_\bbR}$ is equal to the poic given by its relative interior, the linear poic-complex $\underline{N_\bbR}$ is irreducible. Therefore, $\Phi$ must also be irreducible.
\end{example}

\begin{lem}\label{lem: top dimensional cycles}
    For a linear poic-complex $\Phi_X$ the natural inclusion $M_n(\Phi_X)\hookrightarrow Z_n(\Phi_X)$ is an isomorphism.
\end{lem}
\begin{proof}
    We show that for any subdivision $S\colon \Phi^\prime\to \Phi$ the induced map $S^*\colon M_n(\Phi_X)\to M_n(\Phi^\prime_{X\circ S})$ is surjective. A weight $\omega^\prime\in M_n(\Phi^\prime_{X\circ S})$ lies in the image of $S^*$ if and only if for every class $[s]\in[\Phi](n)$ the weight $\omega^\prime$ is constant on the $[p]\in [\Phi^\prime](n)$ with $[S(p)] = [s]$. Now, for an arbitrary cone $s$ of $\Phi$ the weight $\omega^\prime$ restricted to the $[p]\in[\Phi^\prime](n)$ with $[S(p)]=[s]$ gives a top-dimensional weight of $\underline{\Phi(s)^0}_X$, and hence Example \ref{ex: open cone is irred} shows shows that it is constant among these classes.
\end{proof}

\begin{cor}
    If $\Phi_X$ is an irreducible linear poic-complex pure of dimension $n$, then $M_n(\Phi_X)$ is free of rank $1$.
\end{cor}
\begin{example}\label{ex: trop moduli space of rational curves is irreducible}
    Suppose $A$ is a finite set with $\#A\geq 3$. Consider the linear poic-complex $\Mtrop_{0,A}$ and the rational fan $\cM^{\trop}_{0,A}$ as in subsection \ref{ssec: mod spaces of rtc}. It is clear that the Minkowski weights of the linear poic-complex $\Mtrop_{0,A}$ are equal to the Minkowski weights of the fan $\cM^{\trop}_{0,A}$. Since $\cM^{\trop}_{0,A}$ has a subdisivion that is a matroidal fan (see \cite{ArdilaKlivans}), it follows from Lemma \ref{lem: top dimensional cycles} and Proposition 5.2 of \cite{AdiprasitoHuhKatzHTCG} that $\Mtrop_{0,A}$ is irreducible.
\end{example}
\begin{prop}\label{prop: product of irreducibles is irreducible}
    Suppose $\Phi_X$ and $\Psi_Y$ are linear poic-complexes, where $\Phi$ is pure of dimension $n$ and $\Psi$ is pure of dimension $m$. If both are irreducible then $(\Phi\times\Psi)_{X\times Y}$ is also irreducible and, moreover, the cross product induces an isomorphism
    \begin{equation*}
        Z_n(\Phi_X)\otimes_\bbZ Z_m(\Psi_Y)\to Z_{n+m}((\Phi\times\Psi)_{X\times Y}).
    \end{equation*}
    \end{prop}
\begin{proof}
    It is sufficient to show that the cross product induces an isomorphism, and without loss of generality, we can assume that both $\Phi$ and $\Psi$ are simplicial. In particular, $\Phi\times\Psi$ is also simplicial. Following Lemma \ref{lem: top dimensional cycles}, it will suffice to show that the cross product induces an isomorphism
    \begin{equation*}
        M_n(\Phi_X)\otimes_\bbZ M_m(\Psi_Y) \to M_{n+m}(\Phi\times\Psi_{X\times Y}).    
    \end{equation*}
    From Lemma \ref{lem: top dimensional cycles}, the irreducibility of $\Phi_X$ and $\Psi_Y$ implies that both $M_n(\Phi_X)$ and $M_m(\Psi_Y)$ are free of rank $1$. Consider respective generators $\omega\in M_n(\Phi_X)$ and $\eta\in M_m(\Psi_Y)$.
    The map is clearly injective. Hence, it is necessary to show that any weight $\theta\in M_{n+m}((\Phi\times\Psi)_{X\times Y})$ is an integral multiple of $\omega\times\eta$. Since any cone of $\Phi$, resp. $\Psi$, is at most $n$-dimensional, resp. $m$-dimensional, any class of $[\Phi\times\Psi](n+m)$ must be of the form $[s\times t]$ with $[s]\in [\Phi](n)$ and $[t]\in [\Psi](m)$.
    
    Consider an arbitrary $\theta\in M_{n+m}((\Phi\times\Psi)_{X\times Y})$ and let $[t]\in[\Psi](m)$. An argument analogous to the proof to Lemma \ref{lem: cross products} shows that
    \begin{align*}
    \omega_{\theta,t} \colon  [\Phi](n)\to \bbZ, [s]\mapsto \theta([s\times t])
    \end{align*}
    is an $n$-dimensional Minkowski weight on $\Phi_X$. From irreducibility, it follows that $\omega_{\theta,t} = \lambda_t \cdot\omega$, for some integer $\lambda_t$. Again, as $\theta\in M_{n+m}((\Phi\times\Psi)_{X\times Y})$, a similar argument to that of the proof of Lemma \ref{lem: cross products} shows that
    \begin{equation*}
        \lambda\colon [\Psi](m)\to \bbZ, [t]\mapsto \lambda_t
    \end{equation*}
    is an $m$-dimensional Minkowski weight on $\Psi_Y$. As above, irreducibility implies that there is an integer $\ell\in\bbZ$, such that $\lambda_t = \ell \cdot \eta([t])$. Thus, the above shows that $\theta = \ell \cdot \omega\times \eta$.
\end{proof}

\subsection{Pushforwards.}\label{ssec: pushforwards}
Consider two poic-complexes $\Phi$ and $\Psi$, together with a morphism $P\colon \Phi\to\Psi$. Sadly, it is not always possible to pushforward weights on $\Phi$ to weights on $\Psi$, since $\Psi$ may be too coarse. However, if $\Psi$ is subdivided enough, this can be achieved. For this description we introduce the notion of a $P$-fine subdivision of $\Psi$, which makes rigorous what is meant by $\Psi$ being subdivided enough.
\begin{defi}
    Suppose $P\colon \Phi\to\Psi$ is a poic-morphism. A subdivision $S\colon \Psi^\prime\to \Psi$ is called \emph{$P$-fine (or fine for the morphism $P$)}, if for every cone $s$ of $\Phi$ for which the linear part of $P_s\colon \Phi(s)\to\Psi(P(s))$ is injective, there exists a cone $t$ of $\Psi^\prime$ such that: 
   \begin{itemize}
       \item There is an isomorphism $f\colon P(s) \xrightarrow{\sim} S(t)$ of $\Psi$.
       \item The images of ${\Psi^\prime}(t)^0$ and $\Phi(s)^0$ under the linear parts of $S_t$ and $\Psi(f)\circ P_s$ coincide.
   \end{itemize}    
   In this case, it is clear that $d(t)=d(s)$ and such a $t$ is unique up to isomorphism. Following the notation, we write $(P_S)_*([s])= [t]$, for the classes of $s$ and $t$ as in the above definition.
\end{defi}

\begin{const}\label{const: pushforward general weights}
    Suppose $P\colon \Phi\to\Psi$ is a poic-morphism, $S\colon \Psi^\prime\to\Psi$ is a $P$-fine subdivision, and let $k$ be a non-negative integer. The \emph{pushforward of} $\omega\in W_k(\Phi)$ along $P$ is the weight
    \begin{equation*}
    (P_S)_*\omega\colon  [\Psi^\prime](k) \to \bbZ,[t]\mapsto \sum_{\substack{[p]\in [\Phi](k),\\
    (P_S)_*([p]) = [t]}}\omega([p])[S_t N^{t}\colon  P_p N^p].
    \end{equation*}
\end{const}

\begin{lem}
    In the notation of Construction \ref{const: pushforward general weights}, the pushforward gives a linear map:
    \begin{equation*}
        (P_S)_*\colon  W_k(\Phi)\to W_k(\Psi^\prime).
    \end{equation*}
    
\end{lem}
\begin{proof}
    This is clear from the construction.
\end{proof}

We now study the pushforward for Minkowski weights. For this, the notion of a weakly proper morphism has to be introduced, because a general morphism does not automatically reflect the face relations of the target.
\begin{defi}\label{defi: proper poic-complexes}
   We say that a morphism $P\colon \Phi\to\Psi$ is \emph{weakly proper} if it satisfies the following lifting property: For every cone $s$ of $\Phi$ and every face $\tau$ in the closure of $P_s(\Phi(s))$ in $\Psi(P(s))$, there exists $q\to s$ of $\Phi$ with the same codimension which under the morphism $P$ gives this face $\tau$. We say that $P$ is \emph{proper}, if for every subdivision $S\colon \Phi^\prime\to \Phi$, the composition $P\circ S$ is weakly proper. In other words, the lifting property holds for every sub-poic of $\Phi(s)$.
\end{defi}

\begin{example}
    Suppose $\sigma$ is an arbitrary poic. The inclusion $\underline{\sigma^0}\hookrightarrow\underline{\sigma}$ is a weakly peroper morphisms. This inclusion is proper if and only if $\sigma^0 = \sigma$.
\end{example}

\begin{example}
Consider $\bbR^2$ with the standard lattice $\bbZ^2$ and let $\pr_2\colon\bbR^2\to\bbR$ denote the integral linear map given by projection onto the second coordinate. We have the (linear) poic-complex $\underline{H_{\pr_2>0}}$, and the map $\pr_2$ induces a morphism of (linear) poic-complexes $\underline{H_{\pr_2>0}}\to\underline{\bbR}$. This morphism is (vacuosly) weakly proper but it is definitely not proper. Indeed, any $1$-dimensional subcone of $\underline{H_{\pr_2>0}}$ different from the coordinate axis shows that this morphism cannot be proper.  
\end{example}

\begin{lem}\label{lem: pushforward}
   Suppose $\Psi_Y$ is a linear poic-complex and $P\colon \Phi\to\Psi$ is a weakly proper morphism of poic-complexes. If $S:\Psi^\prime\to\Psi$ is a $P$-fine subdivision, then the pushforward map $(P_S)_*\colon  W_k(\Phi)\to W_k(\Psi^\prime)$ restricts to a linear map
    \begin{equation*}
        (P_S)_*\colon M_k(\Phi_{Y\circ P})\to M_k(\Psi^\prime_{Y\circ S}),
    \end{equation*}
    for any integer $k\geq0$.
\end{lem}
\begin{proof}
    The proof of this lemma follows that of Proposition 2.25 of \cite{GathmannKerberMarkwigTFMSTC}. Let $\omega\in M_k(\Phi_{Y\circ P})$ and consider a cone $q$ of $\Psi^\prime$ of dimension $k-1$. If there is a cone $s$ of $\Phi$ and $q\to t$ of $\Psi^\prime$, such that $(P_S)_*([s])=[t]$ and $[q\to t]\in \Star^1_{\Psi^\prime}(q)$, then weakly properness implies that there exists a morphism $p\to s$ of $\Phi$ with $[p]\in [\Phi](k-1)$, $[P(p)]=[S(q)]$, and further $(P_S)_*([p])=[q]$. As a consequence of this, it is only necessary to check the balancing around classes $[q]\in[\Psi^\prime](k-1)$ of the form $(P_S)_*([s])$ for some $[s]\in[\Phi](k-1)$.\\
    Consider $q\to t$ in $\Psi^\prime$ representing an object $[q\to t]\in\Star^1_{\Psi^\prime}(q)$, and let $p\to s$ in $\Phi$ be a morphism with $(P_S)_*([p])=[q]$, $(P_S)_*([s])=[t]$, and representing a class $[p\to s]\in\Star^1_\Phi([p])$. We get the following commutative diagram with exact rows and columns:
    \begin{equation*}
        \begin{tikzcd}
            N^p\arrow[hookrightarrow]{r}\arrow[hookrightarrow]{d}{P_p}&N^s\arrow[twoheadrightarrow]{r}\arrow[hookrightarrow]{d}{P_s}&N^s/N^p\arrow{d}\\
            S_q(N^q)\arrow[twoheadrightarrow]{d}\arrow[hookrightarrow]{r}&S_t(N^t)\arrow[twoheadrightarrow]{d}\arrow[twoheadrightarrow]{r}&S_t(N^t)/S_q(N^q)\arrow[twoheadrightarrow]{d}\\
            S_q(N^q)/P_p(N^p)\arrow[hookrightarrow]{r}&S_t(N^t)/P_s(N^s)\arrow[twoheadrightarrow]{r}&S_t(N^t)/(S_q(N^q)+P_s(N^s)).
        \end{tikzcd}
    \end{equation*}
    The bottom row shows that
    \begin{equation*} [S_t(N^t)\colon P_s(N^s)] = [S_q(N^q)\colon P_p(N^p)]\cdot [S_t(N^t)\colon S_q(N^q)+P_s(N^s)].\end{equation*}
    After composition with $Y$, the previous diagram shows that $P$ induces a map $N_Y/(Y\circ P)_p(N^p) \to N_Y/(Y\circ S)_q(N^q)$, and the normal vector $u_{[p\to s]}^{Y\circ P}$ is mapped to $[S_t(N^t)\colon S_q(N^q)+P_s(N^s)] u_{[q\to t]}^{Y\circ S}$.\\
    Since $\omega\in M_k(\Phi_{Y\circ P})$, necessarily
    \begin{equation*}
        \sum_{[p\to s]\in \Star^1_\Phi(p)}\omega([s])u_{[p\to s]}^{Y\circ P} = 0,
    \end{equation*}
    in the quotient $N_Y/(Y\circ P)_p (N^p)$. Now, it follows from the previous discussion that after multiplying by a factor of $[S_q(N^q)\colon P_p(N^p)]$ the morphism $P$ gives rise to the following equation in the quotient $N_Y/(Y\circ S)_q (N^q)$:
    \begin{equation*}
        \sum_{[p\to s]\in \Star^1_\Phi(p)}\omega([s])[S_{(P_S)_*(s)}(N^{(P_S)_*(s)})\colon  P_s(N^s)]u_{[q\to (P_S)_*(s)]}^{Y\circ S} = 0.
    \end{equation*}
    To finalize, we sum all these equations with respect to all the $[p]\in [\Phi](k-1)$ such that $(P_S)_*([p]) = [t]$, and obtain:
    \begin{multline*}
        0=\sum_{\substack{[p]\in [\Phi](k-1),\\ (P_S)_*([p])=[q]}}\sum_{[p\to s]\in \Star^1_\Phi(p)}\omega([s])[S_{(P_S)_*(s)}(N^{(P_S)_*(s)})\colon  P_s(N^s)]u_{[q\to (P_S)_*(s)]}^{Y\circ S}=\\\sum_{[q\to t]\in\Star^1_{\Psi^\prime}(q)}\sum_{\substack{[s]\in \Phi(k),\\
    (P_S)_*([s]) = [t]}}\omega([s])[S_t(N^t)\colon P_s(N^s)] u_{[q\to t]}^{Y\circ S}=\sum_{[q\to t]\in\Star^1_{\Psi^\prime}(q)}(P_S)_*\omega([t])u_{[q\to t]}^{Y\circ S}.
    \end{multline*}
    This shows balancing for $(P_S)_*\omega$ at the cone $q$ of $\Psi^\prime$. Since the latter was arbitrary, it follows that $(P_S)_*\omega\in M_{k}(\Psi^\prime_{Y\circ S})$.
\end{proof}
The notion of weakly properness (resp. properness) and the above constructions can easily be extended to morphisms of linear poic-complexes. Namely, a morphism $\phi\colon \Phi_X\to \Psi_Y$ is called \emph{weakly proper} (resp. \emph{proper}) if the morphism of poic-complexes $\phi\colon \Phi\to\Psi$ is weakly proper (resp. proper). In this situation, the map $\phi_\tint$ induces a linear map 
\begin{equation*}
    (\phi_\tint)_*\colon  M_k(\Phi_X)\to M_k(\Phi_{Y\circ \phi}),
\end{equation*}
and if $S\colon \Psi^\prime\to \Psi$ is $\phi$-fine, then $\phi$ induces the pushforward map
\begin{equation*}
    ({\phi}_S)_*\colon M_k(\Phi_{Y\circ \phi})\to M_k(\Psi^\prime_{Y\circ S}).
\end{equation*}
Collectively, $\phi$ and $S$ give rise to the map
\begin{equation}
    (\phi_S)_*\colon M_k(\Phi_X)\to M_k(\Psi^\prime_{Y\circ S})\label{eq: pforward minkowski weights P-fine subdiv}
    \end{equation}
defined as the composition $({\phi}_S)_*\circ (\phi_\tint)_*$. In order to extend these constructions to tropical cycles, the following lemma, whose proof we postpone to \ref{appendix: proofs of lemmas} in the appendix, is necessary.
\begin{restatable}[]{lem}{pfnsbdvs}\label{lem: P-fine subdivisions are final}
    If $P\colon \Phi\to\Psi$ is a proper morphism, then there exists a $P$-fine subdivision $\Psi^\prime$ of $\Psi$.
\end{restatable}
\begin{proof}
    The proof of this lemma is postponed to Section \ref{appendix: proofs of lemmas} in the appendix.
\end{proof}

Suppose $\phi\colon \Phi_X\to \Psi_Y$ is a proper morphism between poic-complexes. By definition, if $S\colon \Phi^\prime \to\Phi$ is a subdivision, then $\phi\circ S\colon \Phi^\prime\to\Phi$ is weakly proper. Now, Lemma \ref{lem: P-fine subdivisions are final}  shows that there exists a $\phi\circ S$-fine subdivision $Q\colon \Psi^\prime\to\Psi$ and Lemma \ref{lem: pushforward} shows that we obtain the pushforward map
\begin{equation*}
    (\phi_S)_*\colon M_k(\Phi^\prime_{X\circ S}) \to M_k(\Psi^\prime_{Y\circ Q}).
\end{equation*}
Moreover, taking multiple subdivisions gives rise to commutative diagrams, and composition with the standard map $M_k(\Psi^\prime_{Y\circ Q})\to Z_k(\Psi_Y)$ shows then that these maps give rise to a linear map
\begin{equation}
    \phi_*\colon Z_k(\Phi_X)\to Z_k(\Psi_Y). \label{eq: pushforward map}
\end{equation}
We call this map \eqref{eq: pushforward map} the \emph{pushforward map of $\phi$}. If the morphism $\phi\colon\Phi_X\to\Psi_Y$ is just weakly proper, then we just obtain the \emph{weak pushforward map of $\phi$} 
\begin{equation}
    \phi_*\colon M_k(\Phi_X)\to Z_k(\Psi_Y). \label{eq: weak pushforward map}
\end{equation}

\subsection{Relation to tropical intersection theory.} We will briefly explain how the usual tropical intersection theory of polyhedral complexes in a tropical vector space (see \cite{AllermannRau} and \cite{MikhalkinRau}) can be obtained through our framework, as well as the tropical intersection theory of weakly embedded cone-complexes (see \cite{GrossITTTE}). As before, let $N$ be a free abelian group of finite rank and consider $N$ as a full rank lattice of the vector space $N_\bbR$. The pair $(N_\bbR,N)$ is then a tropical vector space. 
\begin{const}\label{const: polyhedral cycles}
    Suppose $\cX$ is a rational polyhedral complex of this tropical vector space, such that the collection of recession cones of $\cX$ is a rational fan in $N_\bbR$. Let $N_\cX= N\oplus \bbZ$ and $(N_\cX)_\bbR = N_\bbR\oplus \bbR$, with last coordinate denoted by $z$. For a cell $\xi\in \cX$, we consider the partially open integral cone $\sigma_\xi$ of $(N_\cX)_\bbR$ given by adjoining the origin to the intersection of: the cone generated by $\xi\times{1}\subset M_\bbR$ and $\rc(\xi)\times\{0\}\subset M_\bbR$, and the half space $H_z^o$. It is clear that the collection $\Phi_\cX=\{\sigma_\xi\colon \xi\in\cX\}$ is a partially open rational fan (see \cite{GathmannOchseMSCTV}) in $(N_\cX)_\bbR$, and we regard $\Phi(\cX)$ as a linear poic-complex in the standard way. 
\end{const}
\begin{restatable}{lem}{classicaltint}\label{lem: classical cycles and our cycles}
    The tropical cycles of the linear poic-complex $\Phi_\cX$ correspond with the classical tropical cycles of $\cX$.
\end{restatable}
\begin{proof}
    The proof makes use of some notation and constructions from the Appendix, and hence is postponed to the end of \ref{appendix: poic-complexes coming from partially open fans}.
\end{proof}
Recall that classical tropical cycles are stable under subdivisions, and it is always possible to subdivide a polyhedral complex into one whose recession cones form a fan (see \cite{GilSombra}). Therefore, following the above construction, we can compute the classical tropical cycles of any polyhedral complex $\cX$ in our framework.

The rehashing of the case of weakly embedded cone complexes is as follows. We briefly borrow notation and terminology of \cite{GrossITTTE} and \cite{AbramovichCaporasoPayneTTMSC}. Out of an integral cone $(\sigma,M)$ (abusing notation slightly) we produce a poic $\sigma$, where $N^\sigma$ is simply the intersection $\Lin(\sigma)\cap (M^\vee)$. It is clear that a morphism of integral cones induces directly a poic-morphism (the condition on pullbacks immediately implies that the function is integral linear with respect to the lattices). From a cone complex $\Sigma$ we produce a poic-complex by taking $\Sigma$ as a poset (with respect to inclusions), and the tautological associations of cones. We differentiate the two by denoting the poic-complex by $\mathrm{p}\left(\Sigma\right)$. The realization of $\mathrm{p}\left(\Sigma\right)$ is homeomorphic to $|\Sigma|$. Now, if $\Sigma$ is a weakly embedded cone complex, then the map $\phi^\Sigma$ defines a linear poic-complex structure on $\mathrm{p}\left(\Sigma\right)$. Naturally, the Minkowski weights of a weakly embedded cone complex $\Sigma$ (Section 3.1 of \cite{GrossITTTE}) correspond to our Minkowski weights on $\mathrm{p}\left(\Sigma\right)$. Such comparison actually extends to the notion of tropical cycles, which are just the corresponding colimit with respect to proper subdivisions of cone complexes (Section 2.1 of loc. cit.). We record this as a lemma and  postpone its proof to \ref{appendix: poic-complexes coming from partially open fans} in the appendix. 

\begin{restatable}{lem}{moderntint}\label{lem: modern cycles and our cycles}
    The tropical cycles of a weakly embedded cone complex $\Sigma$ correspond with the tropical cycles of the linear poic-complex $\mathrm{p}\left(\Sigma\right)$.
\end{restatable}

\section{Poic-fibrations and their cycles}
This section begins with the introduction of poic-spaces, which are closely related to poic-complexes but come with some non-trivial isomorphisms. A natural family of examples of these spaces are the moduli spaces of tropical curves of positive genus. We seek to describe tropical cycles in poic-spaces, and for this, we shift our attention to poic-fibrations and describe cycles therein. A poic-fibration is a morphism between a poic-complex and a poic-space with some nice local properties and a lifting property concerning face relations and isomorphisms. More interestingly, when the source is a linear poic-complex, we are able to introduce the tropical cycles of the poic-fibration as a subgroup of cycles of the source. We follow an analogous path as before so that in order to define them, we introduce subdivisions and weights in a subdivision (that exhibit compatibility with the poic-fibration). We also discuss the spanning tree fibrations, which are our poic-fibrations of interest to describe cycles of the moduli spaces of tropical curves. Regarding these poic-fibrations and these spaces, we describe the ``forgetful'' and ``clutching'' morphisms. Finally, this section is closed with a description of the relationship between our constructions and those of \cite{CavalieriGrossMarkwig} concerning these moduli spaces.

\subsection{Poic-spaces.} Our leading examples for a poic-spaces come from $\cM_{g,n}^\trop$. These are spaces obtained by glueing finitely many cells given as quotients of cones by actions of finite groups.
\begin{defi}
    A \emph{poic-space} $\cX$ is a functor $\cX\colon C_\cX\to \POIC$, where $C_\cX$ is an essentially finite category with finite hom-sets, subject to the following conditions:
    \begin{enumerate}
        \item The functor $\cX$ maps morphisms of $C_\cX$ to face-embeddings.
        \item If $x$ is an object of $C_\cX$ and $\tau$ is a face of $\cX(x)$, then there exists a unique up to isomorphism $w\to x$ in $C_\cX$ with $\cX(w\to x) \cong \tau\leq \cX(x)$.
    \end{enumerate}
    If $\cY$ is an additional poic-space, a \emph{morphism of poic-spaces} $F\colon \cX\to\cY$ consists of the following data:
    \begin{itemize}
        \item A functor $F\colon C_\cX\to C_\cY$. 
        \item A natural transformation $\eta_F\colon \cX\implies \cY\circ F$.
    \end{itemize}
    These are subject to the following condition: for an object $x$ of $C_\cX$, the cone $\eta_{F,x}\left(\cX(x)\right)$ does not lie in a proper face of $\Psi(F(x))$. Composition of morphisms is defined analogously to that of poic-complexes.
\end{defi}

Both conditions on poic-spaces are analogous to those of poic-complexes. We point out that the main difference between these two notions lies in the groups of automorphisms of objects in the category. Namely, cones of poic-complexes do not carry non-trivial automorphisms.

\begin{nota}
    We follow the same notational conventions as with poic-complexes. Namely, if $\cX$ is a poic-space, then:
    \begin{itemize}
        \item We will refer to the objects of $C_\cX$ as the cones of $\cX$.
        \item We will denote the set of isomorphism classes of $C_\cX$ by $[\cX]$. This set is partitioned by the dimensions of its cones, and, for an integer $k\geq0$, $[\cX](k)$ denotes the set of isomorphism classes of $k$-dimensional cones.
        \item We say that $\cX$ is pure of dimension $n$, if every cone $x$ of $\cX$ appears as the source of a morphism to an $n$-dimensional cone of $\cX$.
        \item For an integer $k\geq0$, the \emph{group $W_k(\cX)$ of $k$-dimensional weights on $\cX$} is the abelian group of functions $[\cX](k)\to\bbZ$.
    \end{itemize}
\end{nota}

Analogously to $\POICCmplxs$, poic-spaces and their morphisms form a category. We denote this category by $\POICSpcs$. We remark that $\POICCmplxs$ is a full subcategory of $\POICSpcs$. As in the case \eqref{eqdefi: realization of poiccomplexes} of poic-complexes, the realization functor \ref{eqdefi: realization of poics} can be extended naturally to poic-spaces, and hence the \emph{realization functor} is obtained:
\begin{equation}
    |\bullet|\colon \POICSpcs\to\Top,\cX\mapsto \colim |\cX|.
\end{equation}
\subsection{Fibrations and their Minkowski weights.}
\begin{defi}
    Suppose $\Phi$ is a poic-complex and $\cX$ is a poic-space. A morphism of poic-spaces $\pi\colon \Phi\to \cX$ is a \emph{poic-fibration}, if:
\begin{enumerate}
    \item $\pi$ is essentially surjective.
    \item For any cone $p$ of $\Phi$ the map $\eta_{\pi,p}\colon  \Phi(p)\to \cX\left( \pi(p)\right)$ induces an isomorphism between their relative interiors.
    \item For any object $p$ of $\Phi$ and any morphism $f\colon \pi(p)\to x$ in $\cX$, there exists a morphism $h\colon p\to q$ in $\Phi$ and an isomorphism $g\colon  \pi(q)\to x$ such that $f=g\circ \pi(h)$, where $h$ is unique up to isomorphism under $p$, and $g$ is unique up to isomorphism over $x$.\footnote{This is the same as saying that $\pi^{\op}\colon C_{\Phi}^{\op}\to C_{\cX}^{\op}$ is a \emph{Street} fibration. We strain ourselves from this notation and simply state the corresponding property with the aim of emphasizing its combinatorial content in our framework.}
\end{enumerate}
A morphism of poic-fibrations is simply a morphism of pairs yielding a commutative square. If $\Phi_X$ is additionally a linear poic-complex, then we say that $\pi\colon\Phi_X\to\cX$ is a linear poic-fibration.
\end{defi}
\begin{example}
    The reader is exhorted to look in subsection \ref{ssec: spanning tree cover.} for a family of examples. Nonetheless, we carry out explicitly the case of $g=1$ and $n=2$. In the spirit of illustration and simplicity, we reduce multiple considerations to finite categories. Let $C_\Phi$ denote the poset given by $4$ cones $\rho$, $\sigma_1$, $\sigma_2$, and $\sigma_2^*$, where $\rho$ is isomorphic to $\bbR^1_{>0}$ and the rest are all isomorphic to the subcone of $\bbR^2_{\geq0}$ given by removing one ray. We let $\leq$ denote the partial order on $C_\Phi$ and assume that in $C_\Phi$ the order relation $\rho\leq \sigma_1,\sigma_2,\sigma_2^*$ holds, and the latter three are incomporable. The tautological association makes $C_\Phi$ into a poic-complex $\Phi$. Now, we let $C_\cX$ denote the finite category with $3$ objects $\{\fish,\tailcirc,\eyetwolashes\}$, where there is a unique map $\fish\to\tailcirc$, a unique non-trivial automorphism of $\eyetwolashes$, and two maps $\fish\to\eyetwolashes$. We let $\cX$ denote the functor defined 
    \begin{itemize}
        \item at objects by $\cX(\fish) = \rho$, $\cX(\tailcirc) = \sigma_1$, and $\cX(\eyetwolashes) = \sigma^\prime$, where $\sigma^\prime$ is the subcone of $\bbR^2_{\geq0}$ defined by $\bbR^2_{\geq0}\backslash\{(0,0)\}$,
        \item at morphism by $\cX(\fish\to\tailcirc)$ the facet inclusion, $\cX(\eyetwolashes)\to\cX(\eyetwolashes)$ the linear map switching both coordinates, and the maps $\cX(\fish\to\eyetwolashes)$ the two facet inclusions. 
    \end{itemize}  
    In this case, we get a poic-fibration $\pi\colon\Phi\to\cX$ by setting $\pi(\sigma_1) = \tailcirc$, $\pi(\rho)=\fish$, and $\pi(\sigma_2)=\pi(\sigma_2^*) = \eyetwolashes$, where natural transformations is given by the identity maps, and the natural inclusions. We depict the realization of $\Phi$, the realia $\cX$ and the poic-fibration in Figure \ref{fig: example of fibration}.

    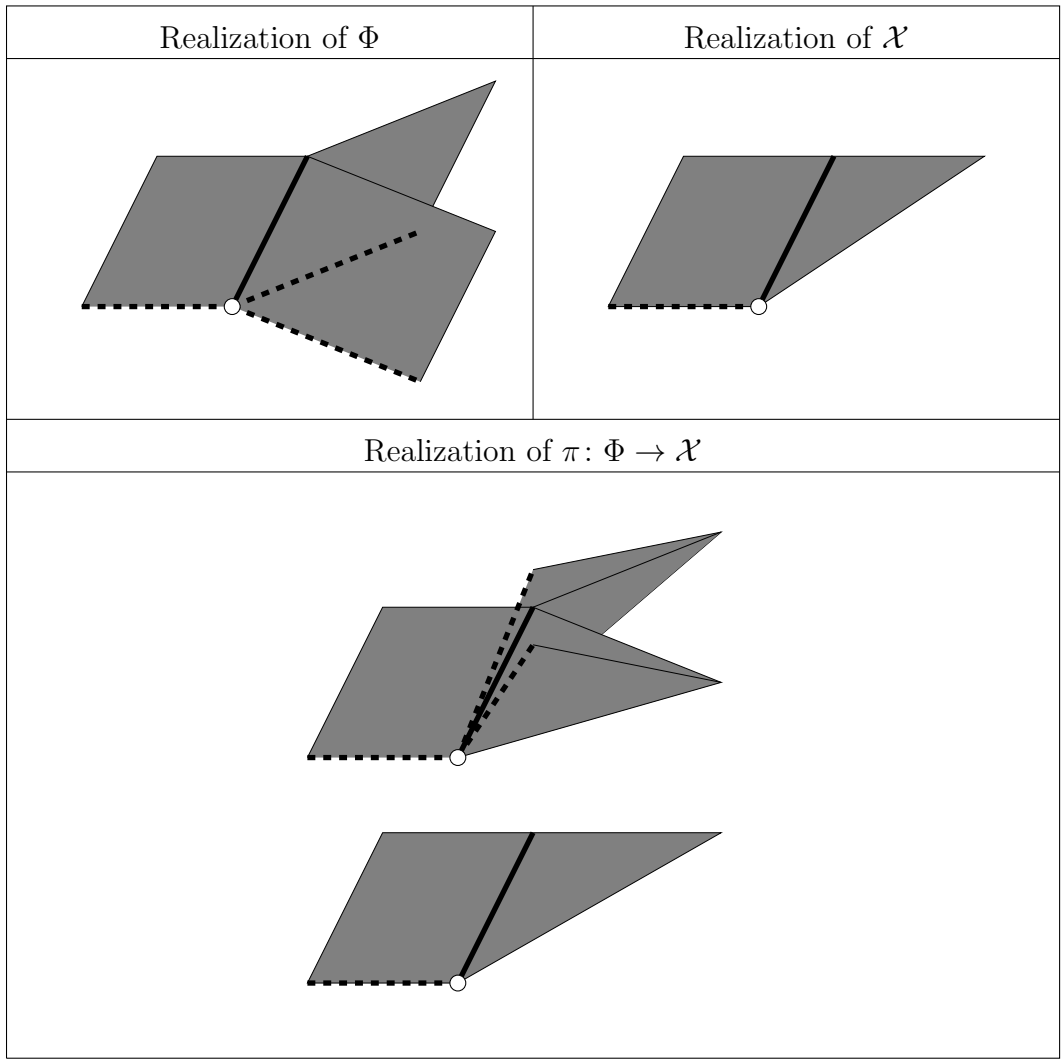
\begin{figure}[!ht]
        \centering
        \begin{tikzpicture}
            \draw[fill=gray,opacity=0.3,line width=0pt] (-3,0)--(-2,2)--(0.5,3)--(-0.5,1);
            \draw[fill=gray,opacity=0.3,line width=0pt] (-3,0)--(-2,2)--(0.5,1)--(-0.5,-1);
            \draw[fill=gray,opacity=0.3,line width=0pt] (-3,0)--(-2,2)--(-4,2)--(-5,0);
            \draw[line width = 2pt] (-3,0)--(-2,2);
            \draw[dashed, line width = 2pt] (-5,0)--(-3,0);
            \draw[dashed, line width = 2pt] (-3,0)--(-0.5,1);
            \draw[dashed, line width = 2pt] (-3,0)--(-0.5,-1);
            \draw[fill=white] (-3,0) circle (3pt);

            \draw[fill=gray,opacity=0.3,line width=0pt] (4,0)--(2,0)--(3,2)--(5,2);
            \draw[fill=gray,opacity=0.3,line width=0pt] (4,0)--(7,2)--(5,2);
            \draw[line width = 2pt] (4,0)--(5,2);
            \draw[dashed, line width = 2pt] (2,0)--(4,0);
            \draw[fill=white] (4,0) circle (3pt);

            \draw[fill=gray,opacity=0.3,line width=0pt] (0,-6)--(3.5,-3)--(1,-3.5);
            \draw[fill=gray,opacity=0.3,line width=0pt] (0,-6)--(1,-4)--(3.5,-3);
            \draw[fill=gray,opacity=0.3,line width=0pt] (0,-6)--(1,-4)--(3.5,-5);
            \draw[fill=gray,opacity=0.3,line width=0pt] (0,-6)--(3.5,-5)--(1,-4.5);
            \draw[fill=gray,opacity=0.3,line width=0pt] (0,-6)--(1,-4)--(-1,-4)--(-2,-6);
            \draw[line width = 2pt] (0,-6)--(1,-4);
            \draw[dashed, line width = 2pt] (-2,-6)--(0,-6);
            \draw[dashed, line width = 2pt] (0,-6)--(1,-3.5);
            \draw[dashed, line width = 2pt] (0,-6)--(1,-4.5);
            \draw[fill=white] (0,-6) circle (3pt);
            
            \draw[fill=gray,opacity=0.3,line width=0pt] (0,-9)--(-2,-9)--(-1,-7)--(1,-7);
            \draw[fill=gray,opacity=0.3,line width=0pt] (0,-9)--(3.5,-7)--(1,-7);
            \draw[line width = 2pt] (0,-9)--(1,-7);
            \draw[dashed, line width = 2pt] (-2,-9)--(0,-9);
            \draw[fill=white] (0,-9) circle (3pt);

            \draw (-6,-10)--(8,-10)--(8,4)--(-6,4)--(-6,-10);
            \draw (-6,3.3)--(1,3.3) node[midway, above] {Realization of $\Phi$};
            \draw (1,3.3)--(8,3.3) node[midway, above] {Realization of $\cX$};
            \draw (-6,-1.5)--(8,-1.5);
            \draw (-6,-2.2)--(8,-2.2) node [midway, above] {Realization of $\pi\colon\Phi\to\cX$};
            \draw (1,-1.5)--(1,4);
            
        \end{tikzpicture}
        \caption{Example of a poic-fibration}
        \label{fig: example of fibration}
    \end{figure}
\end{example}
\begin{example}\label{ex: working example}
    The following will be our working example throughout the section. As a poic-space $\cX$ we consider $\bbR^3_{>0}$ with set of automorphisms: 
    \begin{equation*}
        \left\{
        \begin{pmatrix}
            1&0&0\\
            0&1&0\\
            0&0&1
        \end{pmatrix},
        \begin{pmatrix}
            0&0&1\\
            1&0&0\\
            0&1&0
        \end{pmatrix},
        \begin{pmatrix}
            0&1&0\\
            0&0&1\\
            1&0&0
        \end{pmatrix}
        \right\}.
    \end{equation*}
    As poic-complex $\Phi$ we just take two disjoint copies of $\bbR^3_{>0}$, and for the poic-fibration we just consider the identity maps. It can immediately be checked that this, trivially, defines a poic-fibration. We depict both $\Phi$ and $\cX$ in Figure \ref{fig: working example} by looking at the corresponding slices given by $x+y+z=1$.
    
    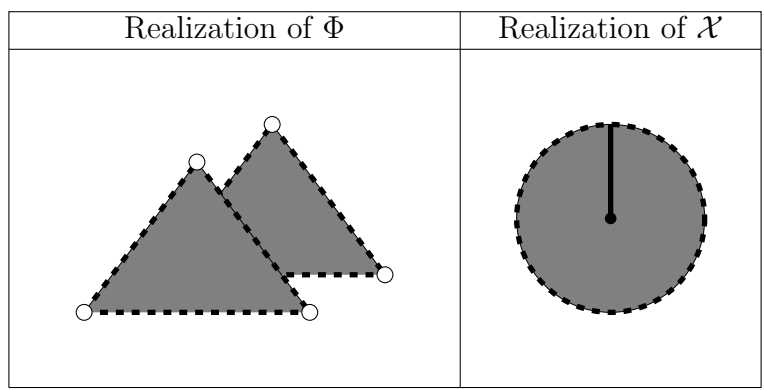
\begin{figure}[!ht]
        \centering
        \begin{tikzpicture}
            \draw[fill=gray,opacity=0.3,line width=0pt] (-3,0.5)--(-1.5,2.5)--(0,0.5); 
            \draw[dashed, line width = 2pt] (-3,0.5)--(-1.5,2.5);
            \draw[dashed, line width = 2pt] (-1.5,2.5)--(0,0.5);
            \draw[dashed, line width = 2pt] (-3,0.5)--(0,0.5);
            \draw[fill = white] (-3,0.5) circle (3pt);
            \draw[fill = white] (0,0.5) circle (3pt);
            \draw[fill = white] (-1.5,2.5) circle (3pt);
            
            \draw[fill=gray,opacity=0.3,line width=0pt] (-4,0)--(-2.5,2)--(-1,0); 
            \draw[dashed, line width = 2pt] (-4,0)--(-2.5,2);
            \draw[dashed, line width = 2pt] (-2.5,2)--(-1,0);
            \draw[dashed, line width = 2pt] (-4,0)--(-1,0);
            \draw[fill = white] (-4,0) circle (3pt);
            \draw[fill = white] (-1,0) circle (3pt);
            \draw[fill = white] (-2.5,2) circle (3pt);

            \draw[fill=gray,opacity=0.3,line width=0pt] (3,1.25) circle (1.25);
            \draw[line width = 2pt] (3,1.25)--(3,2.5);
            \draw[dashed, line width = 2pt] (4.25,1.25) arc(360:0:1.25);
            \draw[fill = black] (3,1.25) circle (2pt);

            \draw (-5,-1)--(5,-1)--(5,4)--(-5,4)--(-5,-1);
            \draw (1,-1)--(1,4);
            \draw (-5,3.5)--(1,3.5) node [midway, above] {Realization of $\Phi$};
            \draw (1,3.5)--(5,3.5) node [midway, above] {Realization of $\cX$};
        \end{tikzpicture}
        \caption{Depiction of poic-spaces of Example \ref{ex: working example}.}
        \label{fig: working example}
    \end{figure}
\end{example}
Suppose now that $\pi\colon \Phi_X\to \cX$ is a linear poic-fibration. The first and second condition imply that $\pi$ induces a surjective map of sets
\begin{equation*}
    [\pi]\colon [\Phi](k)\to [\cX](k).
\end{equation*} 
This map, in turn, induces an injective linear map $[\pi]^*\colon W_k(\cX)\to W_k(\Phi)$.
\begin{defi}
    If $\pi\colon\Phi_X\to \cX$ is a linear poic-fibration. A \emph{$k$-dimensional $\pi$-equivariant Minkowski weight $\omega$} is a weight $\omega\in M_k(\Phi_X)$ such that $\omega\in [\pi]^*\left(W_k(\cX)\right)$. The set of $\pi$-equivariant Minkowski weights is denoted by $M_k(\cX_{\pi,X})$, and clearly is an abelian subgroup of $M_k(\Phi_X)$.
\end{defi}

\subsection{Subdivisions and tropical cycles of a poic-fibration.} To define tropical cycles for a poic-fibration, we proceed as in the case of poic-complexes. Namely, we introduce subdivisions of the source. Naturally, an arbitrary subdivision will not work as it will not reflect any compatibility with the poic-fibration (more precisely, the isomorphisms of the target). For this we introduce the notion of subdivisions compatible with the poic-fibration. This forces us to begin with some constructions and notation.\\
Suppose that $\Phi$ is a poic-complex, $S\colon \Phi^\prime\to\Phi$ is a subdivision, and let $s$ be a cone of $\Phi$. 
\begin{const}
    Let $t$ be a cone of $\Phi^\prime$ such that $S(t)\cong s$. The induced linear map
    \begin{equation*}
    S_t\colon {\Phi^\prime}(t)\to \Phi(s)
    \end{equation*}
    is injective, and $S_t({\Phi^\prime}(t))$ does not lie in a proper face of $\Phi(s)$. Therefore, the intersection $S^0(t)\colon =S_t\left({\Phi^\prime}(t)\right)\cap \Phi(s)^0$ is a (non-empty!) partially open subcone of ${\Phi}(s)^0$. Let $S^{-1}(s)$ denote the following set of partially open subcones of $\Phi(s)^0$
    \begin{equation*}
        S^{-1}(s) :=\{ S^0(t) : t \textnormal{ is a cone of }\Phi^\prime \textnormal{ with } S(t) \cong s\}.
    \end{equation*}
    
\end{const}
\begin{lem}
    Following the above notation, the set $S^{-1}(s)$ is closed under face relations, and hence the tautological association gives rise to a poic-complex (which we denote in the same way). This poic-complex is a subdivision of $\Phi(s)^0$.
\end{lem}
\begin{proof}
    We remark that a face of $S^0(t)$ must come from a face of ${\Phi^\prime}(t)$, and hence comes from a morphism $t^\prime\to t$ of $\Phi^\prime$. Since $S$ is a subdivision, necessarily $S(t^\prime)\cong s$, and therefore $S^0(t^\prime)\in S^{-1}(s)$. In particular, this set is closed under face relations and must also be a subdivision of $\Phi(s)^0$, because $S$ is a subdivision of $\Phi$.
\end{proof}
\begin{defi}
    Let $\pi\colon \Phi\to\cX$ be a poic-fibration. A subdivision $S\colon \Phi^\prime\to\Phi$ is called \emph{$\pi$-compatible} if: For any cones $p$ and $q$ of $\Phi$ together with an isomorphism $f\colon \pi(p) \to \pi(q)$ of $\cX$, the subdivisions $\cX(f)\left(\pi_p \left(S^{-1}(p)\right)\right)$ and $\pi_q\left(S^{-1}(q)\right)$ of $\cX(\pi(q))^0$ coincide.
\end{defi}
\begin{example}
    Following example \ref{ex: working example}, we depict in Figure \ref{fig: working example subdivisions} several $\pi$-stable subdivisions and not $\pi$-stable subdivisions. The new ray introduced in each cone, depicted by the black dot, is the ray given by $x=y=z>0$ at the corresponding cone.
    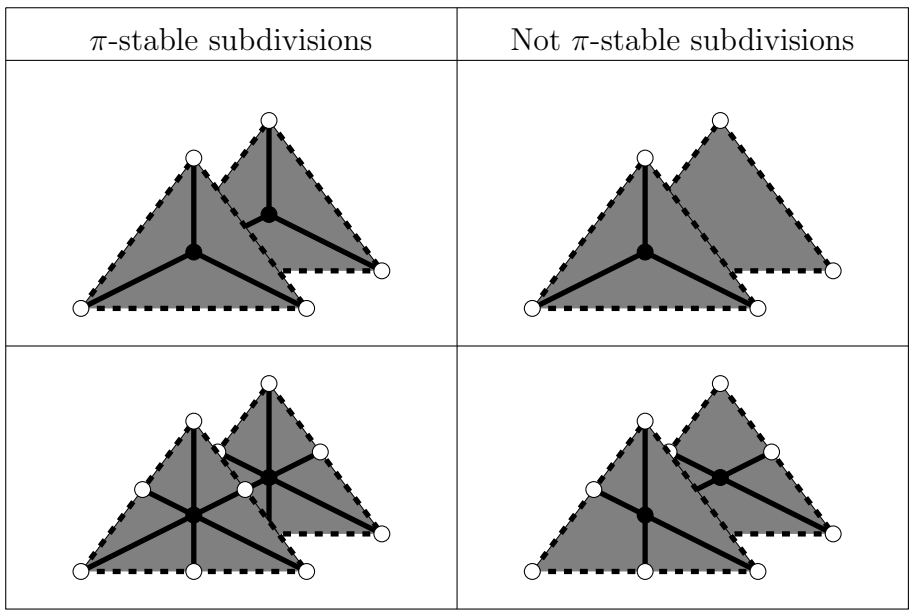
\begin{figure}[!ht]
        \centering
        \begin{tikzpicture}
            \draw[fill=gray,opacity=0.3,line width=0pt] (-3,7.5)--(-1.5,9.5)--(0,7.5); 
            \draw[dashed, line width = 2pt] (-3,7.5)--(-1.5,9.5);
            \draw[dashed, line width = 2pt] (-1.5,9.5)--(0,7.5);
            \draw[dashed, line width = 2pt] (-3,7.5)--(0,7.5);
            \draw[line width = 2pt](-3,7.5)--(-1.5,8.25);
            \draw[line width = 2pt](-1.5,9.5)--(-1.5,8.25);
            \draw[line width = 2pt](0,7.5)--(-1.5,8.25);
            \draw[fill = black] (-1.5,8.25) circle (3pt);
            \draw[fill = white] (-3,7.5) circle (3pt);
            \draw[fill = white] (0,7.5) circle (3pt);
            \draw[fill = white] (-1.5,9.5) circle (3pt);
            
            \draw[fill=gray,opacity=0.3,line width=0pt] (-4,7)--(-2.5,9)--(-1,7); 
            \draw[dashed, line width = 2pt] (-4,7)--(-2.5,9);
            \draw[dashed, line width = 2pt] (-2.5,9)--(-1,7);
            \draw[dashed, line width = 2pt] (-4,7)--(-1,7);
            \draw[line width = 2pt](-4,7)--(-2.5,7.75);
            \draw[line width = 2pt](-2.5,9)--(-2.5,7.75);
            \draw[line width = 2pt](-1,7)--(-2.5,7.75);
            \draw[fill = black] (-2.5,7.75) circle (3pt);
            \draw[fill = white] (-4,7) circle (3pt);
            \draw[fill = white] (-1,7) circle (3pt);
            \draw[fill = white] (-2.5,9) circle (3pt);

            \draw[fill=gray,opacity=0.3,line width=0pt] (-3,4)--(-1.5,6)--(0,4); 
            \draw[dashed, line width = 2pt] (-3,4)--(-1.5,6);
            \draw[dashed, line width = 2pt] (-1.5,6)--(0,4);
            \draw[dashed, line width = 2pt] (-3,4)--(0,4);
            \draw[line width = 2pt](-3,4)--(-0.818,5.09);
            \draw[line width = 2pt](-1.5,6)--(-1.5,4);
            \draw[line width = 2pt](0,4)--(-2.181,5.09);
            \draw[fill = black] (-1.5,4.75) circle (3pt);
            \draw[fill = white] (-0.818,5.09) circle (3pt);
            \draw[fill = white] (-1.5,4) circle (3pt);
            \draw[fill = white] (-2.181,5.09) circle (3pt);
            \draw[fill = white] (-3,4) circle (3pt);
            \draw[fill = white] (0,4) circle (3pt);
            \draw[fill = white] (-1.5,6) circle (3pt);
            
            \draw[fill=gray,opacity=0.3,line width=0pt] (-4,3.5)--(-2.5,5.5)--(-1,3.5); 
            \draw[dashed, line width = 2pt] (-4,3.5)--(-2.5,5.5);
            \draw[dashed, line width = 2pt] (-2.5,5.5)--(-1,3.5);
            \draw[dashed, line width = 2pt] (-4,3.5)--(-1,3.5);
            \draw[line width = 2pt](-4,3.5)--(-1.818,4.59);
            \draw[line width = 2pt](-2.5,5.5)--(-2.5,3.5);
            \draw[line width = 2pt](-1,3.5)--(-3.181,4.59);
            \draw[fill = black] (-2.5,4.25) circle (3pt);
            \draw[fill = white] (-1.818,4.59) circle (3pt);
            \draw[fill = white] (-2.5,3.5) circle (3pt);
            \draw[fill = white] (-3.181,4.59) circle (3pt);
            \draw[fill = white] (-4,3.5) circle (3pt);
            \draw[fill = white] (-1,3.5) circle (3pt);
            \draw[fill = white] (-2.5,5.5) circle (3pt);

            \draw (-5,11)--(-5,3)--(7,3)--(7,11)--(-5,11);
            \draw (-5,10.3)--(1,10.3) node [midway,above]{$\pi$-stable subdivisions};
            \draw (1,10.3)--(7,10.3) node [midway,above]{Not $\pi$-stable subdivisions};
            \draw (-5,6.5)--(7,6.5);
            \draw (1,11)--(1,3);

            \draw[fill=gray,opacity=0.3,line width=0pt] (3,7.5)--(4.5,9.5)--(6,7.5); 
            \draw[dashed, line width = 2pt] (3,7.5)--(4.5,9.5);
            \draw[dashed, line width = 2pt] (4.5,9.5)--(6,7.5);
            \draw[dashed, line width = 2pt] (3,7.5)--(6,7.5);
            \draw[fill = white] (3,7.5) circle (3pt);
            \draw[fill = white] (6,7.5) circle (3pt);
            \draw[fill = white] (4.5,9.5) circle (3pt);
            
            \draw[fill=gray,opacity=0.3,line width=0pt] (2,7)--(3.5,9)--(5,7); 
            \draw[dashed, line width = 2pt] (2,7)--(3.5,9);
            \draw[dashed, line width = 2pt] (3.5,9)--(5,7);
            \draw[dashed, line width = 2pt] (2,7)--(5,7);
            \draw[line width = 2pt] (2,7)--(3.5,7.75);
            \draw[line width = 2pt] (5,7)--(3.5,7.75);
            \draw[line width = 2pt] (3.5,9)--(3.5,7.75);
            \draw[fill= black] (3.5,7.75) circle (3pt);
            \draw[fill = white] (2,7) circle (3pt);
            \draw[fill = white] (5,7) circle (3pt);
            \draw[fill = white] (3.5,9) circle (3pt);

            \draw[fill=gray,opacity=0.3,line width=0pt] (3,4)--(4.5,6)--(6,4); 
            \draw[dashed, line width = 2pt] (3,4)--(4.5,6);
            \draw[dashed, line width = 2pt] (4.5,6)--(6,4);
            \draw[dashed, line width = 2pt] (3,4)--(6,4);
            \draw[line width = 2pt] (3,4)--(5.181,5.09);
            \draw[line width = 2pt] (6,4)--(3.818,5.09);
            \draw[fill= black] (4.5,4.75) circle (3pt);
            \draw[fill = white] (3.818,5.09) circle (3pt);
            \draw[fill = white] (5.181,5.09) circle (3pt);
            \draw[fill = white] (3,4) circle (3pt);
            \draw[fill = white] (6,4) circle (3pt);
            \draw[fill = white] (4.5,6) circle (3pt);
            
            \draw[fill=gray,opacity=0.3,line width=0pt] (2,3.5)--(3.5,5.5)--(5,3.5); 
            \draw[dashed, line width = 2pt] (2,3.5)--(3.5,5.5);
            \draw[dashed, line width = 2pt] (3.5,5.5)--(5,3.5);
            \draw[dashed, line width = 2pt] (2,3.5)--(5,3.5);
            \draw[line width = 2pt] (3.5,5.5)--(3.5,3.5);
            \draw[line width = 2pt] (5,3.5)--(2.818,4.59);
            \draw[fill= black] (3.5,4.25) circle (3pt);
            \draw[fill = white] (2.818,4.59) circle (3pt);
            \draw[fill = white] (3.5,3.5) circle (3pt);
            \draw[fill = white] (2,3.5) circle (3pt);
            \draw[fill = white] (5,3.5) circle (3pt);
            \draw[fill = white] (3.5,5.5) circle (3pt);
        \end{tikzpicture}
        \caption{Examples and non-exmaples of $\pi$-stable subdivisions.}
        \label{fig: working example subdivisions}
    \end{figure}

\end{example}
\begin{restatable}{lem}{pfnpicmptblsbdvsns}\label{lem: refining to pi-compatible subdivisions}
    Suppose $\pi\colon\Phi\to\cX$ is a poic-fibration, where $\Phi$ is pure of dimension $n$. If $S\colon \Phi^\prime\to\Phi$ is a subdivision, then there exists a subdivision $S^{\prime}\colon \Phi^{\prime\prime}\to\Phi^\prime$ such that $S\circ S^\prime\colon \Phi^{\prime\prime}\to\Phi$ is $\pi$-compatible. In other words, $\pi$-compatible subdivison are final as subdivisions of $\Phi$.
\end{restatable}
\begin{proof}
    The proof of the lemma is postponed to Section \ref{appendix: proofs of lemmas} in the appendix.
\end{proof}
\begin{nota}\label{nota: notacion linear pf}
    In what follows, we will assume that $\pi\colon \Phi_X\to\cX$ is a linear poic-fibration, and $S\colon \Phi^\prime\to\Phi$ is a $\pi$-compatible subdivision.
\end{nota}
\begin{const}\label{const: bijection}
    If $p$ and $q$ are cones of $\Phi$, then an isomorphism $f\colon\pi(p)\to \pi(q)$ of $\cX$ induces a bijective map $[S^{-1}(p)]\to [S^{-1}(q)]$ that preserves the respective dimensions. Observe that we can regard the sets $[S^{-1}(p)]$ and $[S^{-1}(q)]$ as subsets of $[\Phi^\prime]$, because $S$ is a subdivision map. Therefore, by trivially extending the  previous bijection, we obtain a permutation of $[\Phi^\prime]$ which we denote by $b_f$. By construction $b_f$ preserves the dimension of the isomorphism classes of the cones.
\end{const}
\begin{defi}
    Following Notation \ref{nota: notacion linear pf}, let $p$ and $q$ be cones of $\Phi$ with $f\colon \pi(p)\to \pi(q)$ an isomorphism of $\cX$. A weight $\omega\in M_k(\Phi^\prime_{X\circ S})$, where $k\geq 0$ is an integer, is called \emph{$f$-stable} if $\omega= \omega\circ w_f$ (here $w_f$ is the bijection of Construction \ref{const: bijection}). The weight $\omega$ is called a \emph{$\pi$-equivariant Minkowski weight} if it is $f$-stable for every triple $(p,q,f)$ as above. We denote the set of $\pi$-equivariant Minkowski weights of $S\colon \Phi^\prime\to\Phi$ by $M_k(\cX_{\pi,X\circ S})$.
\end{defi}
\begin{example}
    In Figure \ref{fig: working example weights} we use the poic-complex $\Phi$ of Example \ref{ex: working example} and the $\pi$-stable subdivisions thereof from Figure \ref{fig: working example subdivisions} to depict some $2$-dimensional $\pi$-equivariant Minkowski weights and some non-$\pi$-equivariant Minkowski weights. In these depictions of $\Phi$ we separate the disjoint cones to make the weight assignment most explicit.

    \begin{figure}
        \centering
        \begin{tikzpicture}
            \draw[fill=gray,opacity=0.3,line width=0pt] (-7,0)--(-5.5,2)--(-4,0);
            \draw[dashed, line width=2pt] (-7,0)--(-5.5,2);
            \draw[dashed, line width=2pt] (-4,0)--(-5.5,2);
            \draw[dashed, line width=2pt] (-7,0)--(-4,0);
            \draw[color = red,line width=2pt](-7,0)--(-5.5,0.75)node [midway, above] {$1$};
            \draw[line width=2pt](-5.5,0.75)--(-4.818,1.09);
            \draw[color = red,line width=2pt](-5.5,2)--(-5.5,0.75)node [pos= 0.6,right] {$1$};
            \draw[line width=2pt](-5.5,0.75)--(-5.5,0);
            \draw[color = red,line width=2pt](-4,0)--(-5.5,0.75)node [pos= 0.7,below] {$1$};
            \draw[line width=2pt](-5.5,0.75)--(-6.181,1.09);
            \draw[fill=black] (-5.5,0.75) circle (3pt);
            \draw[fill=white] (-6.181,1.09) circle (3pt);
            \draw[fill=white] (-5.5,0) circle (3pt);
            \draw[fill=white] (-4.818,1.09) circle (3pt);
            \draw[fill=white] (-7,0) circle (3pt);
            \draw[fill=white] (-5.5,2) circle (3pt);
            \draw[fill=white] (-4,0) circle (3pt);
            
            \draw[fill=gray,opacity=0.3,line width=0pt] (-3,0)--(-1.5,2)--(0,0);
            \draw[dashed, line width=2pt] (-3,0)--(-1.5,2);
            \draw[dashed, line width=2pt] (0,0)--(-1.5,2);
            \draw[dashed, line width=2pt] (-3,0)--(0,0);
            \draw[color = red,line width=2pt](-3,0)--(-1.5,0.75)node [midway, above] {$1$};
            \draw[line width=2pt](-1.5,0.75)--(-0.818,1.09);
            \draw[color = red,line width=2pt](-1.5,2)--(-1.5,0.75)node [pos = 0.6, right] {$1$};
            \draw[line width=2pt](-1.5,0.75)--(-1.5,0);
            \draw[color = red,line width=2pt](-0,0)--(-1.5,0.75)node [pos=0.7, below] {$1$};
            \draw[line width=2pt](-1.5,0.75)--(-2.181,1.09);
            \draw[fill=black] (-1.5,0.75) circle (3pt);
            \draw[fill=white] (-2.181,1.09) circle (3pt);
            \draw[fill=white] (-1.5,0) circle (3pt);
            \draw[fill=white] (-0.818,1.09) circle (3pt);
            \draw[fill=white] (-3,0) circle (3pt);
            \draw[fill=white] (-1.5,2) circle (3pt);
            \draw[fill=white] (0,0) circle (3pt);

            \draw[fill=gray,opacity=0.3,line width=0pt] (2,0)--(3.5,2)--(5,0);
            \draw[dashed, line width=2pt] (2,0)--(3.5,2);
            \draw[dashed, line width=2pt] (5,0)--(3.5,2);
            \draw[dashed, line width=2pt] (2,0)--(5,0);
            \draw[color = red,line width=2pt](2,0)--(3.5,0.75)node [midway, above] {$1$};
            \draw[color = red,line width=2pt](3.5,0.75)--(4.181,1.09) node [pos=0.8,below]{$1$};
            \draw[line width=2pt](3.5,2)--(3.5,0.75);
            \draw[line width=2pt](3.5,0.75)--(3.5,0);
            \draw[color = red,line width=2pt](5,0)--(3.5,0.75)node [pos =0.7,below]{$1$};
            \draw[color = red,line width=2pt](3.5,0.75)--(2.818,1.09)node [pos=0.5, above] {$1$};
            \draw[fill=black] (3.5,0.75) circle (3pt);
            \draw[fill=white] (2.818,1.09) circle (3pt);
            \draw[fill=white] (3.5,0) circle (3pt);
            \draw[fill=white] (4.181,1.09) circle (3pt);
            \draw[fill=white] (2,0) circle (3pt);
            \draw[fill=white] (3.5,2) circle (3pt);
            \draw[fill=white] (5,0) circle (3pt);
            
            \draw[fill=gray,opacity=0.3,line width=0pt] (6,0)--(7.5,2)--(9,0);
            \draw[dashed, line width=2pt] (6,0)--(7.5,2);
            \draw[dashed, line width=2pt] (9,0)--(7.5,2);
            \draw[dashed, line width=2pt] (6,0)--(9,0);
            \draw[color = red,line width=2pt](6,0)--(7.5,0.75)node [midway, above] {$1$};
            \draw[color = red,line width=2pt](7.5,0.75)--(8.181,1.09)node [pos=0.8,below]{$1$};
            \draw[line width=2pt](7.5,2)--(7.5,0.75);
            \draw[line width=2pt](7.5,0.75)--(7.5,0);
            \draw[color = red,line width=2pt](9,0)--(7.5,0.75)node [pos =0.7,below]{$1$};
            \draw[color = red,line width=2pt](7.5,0.75)--(6.818,1.09)node [pos =0.5,above]{$1$};
            \draw[fill=black] (7.5,0.75) circle (3pt);
            \draw[fill=white] (6.818,1.09) circle (3pt);
            \draw[fill=white] (7.5,0) circle (3pt);
            \draw[fill=white] (8.181,1.09) circle (3pt);
            \draw[fill=white] (6,0) circle (3pt);
            \draw[fill=white] (7.5,2) circle (3pt);
            \draw[fill=white] (9,0) circle (3pt);


        \draw[fill=gray,opacity=0.3,line width=0pt] (-7,-4)--(-5.5,-2)--(-4,-4);
            \draw[dashed, line width=2pt] (-7,-4)--(-5.5,-2);
            \draw[dashed, line width=2pt] (-4,-4)--(-5.5,-2);
            \draw[dashed, line width=2pt] (-7,-4)--(-4,-4);
            \draw[color = red,line width=2pt](-7,-4)--(-5.5,-3.25)node [midway, above] {$1$};
            \draw[color = blue,line width=2pt](-5.5,-3.25)--(-4.818,-2.91)node [pos=0.8, below] {$2$};
            \draw[color = red,line width=2pt](-5.5,-2)--(-5.5,-3.25)node [pos=0.6, right] {$1$};
            \draw[color = blue,line width=2pt](-5.5,-3.25)--(-5.5,-4)node [pos=0.6, left] {$2$};
            \draw[color = red,line width=2pt](-4,-4)--(-5.5,-3.25)node [pos=0.7, below] {$1$};
            \draw[color = blue,line width=2pt](-5.5,-3.25)--(-6.181,-2.91)node [pos=0.5, above] {$2$};
            \draw[fill=black] (-5.5,-3.25) circle (3pt);
            \draw[fill=white] (-6.181,-2.91) circle (3pt);
            \draw[fill=white] (-5.5,-4) circle (3pt);
            \draw[fill=white] (-4.818,-2.91) circle (3pt);
            \draw[fill=white] (-7,-4) circle (3pt);
            \draw[fill=white] (-5.5,-2) circle (3pt);
            \draw[fill=white] (-4,-4) circle (3pt);
            
            \draw[fill=gray,opacity=0.3,line width=0pt] (-3,-4)--(-1.5,-2)--(0,-4);
            \draw[dashed, line width=2pt] (-3,-4)--(-1.5,-2);
            \draw[dashed, line width=2pt] (0,-4)--(-1.5,-2);
            \draw[dashed, line width=2pt] (-3,-4)--(0,-4);
            \draw[color = red,line width=2pt](-3,-4)--(-1.5,-3.25)node [midway, above] {$1$};
            \draw[color = blue,line width=2pt](-1.5,-3.25)--(-0.818,-2.91)node [pos=0.8, below] {$2$};
            \draw[color = red,line width=2pt](-1.5,-2)--(-1.5,-3.25)node [pos=0.6, right] {$1$};
            \draw[color = blue,line width=2pt](-1.5,-3.25)--(-1.5,-4)node [pos=0.6, left] {$2$};
            \draw[color = red,line width=2pt](-0,-4)--(-1.5,-3.25)node [pos=0.7, below] {$1$};
            \draw[color = blue,line width=2pt](-1.5,-3.25)--(-2.181,-2.91)node [pos=0.5, above] {$2$};
            \draw[fill=black] (-1.5,-3.25) circle (3pt);
            \draw[fill=white] (-2.181,-2.91) circle (3pt);
            \draw[fill=white] (-1.5,-4) circle (3pt);
            \draw[fill=white] (-0.818,-2.91) circle (3pt);
            \draw[fill=white] (-3,-4) circle (3pt);
            \draw[fill=white] (-1.5,-2) circle (3pt);
            \draw[fill=white] (0,-4) circle (3pt);

            \draw[fill=gray,opacity=0.3,line width=0pt] (2,-4)--(3.5,-2)--(5,-4);
            \draw[dashed, line width=2pt] (2,-4)--(3.5,-2);
            \draw[dashed, line width=2pt] (5,-4)--(3.5,-2);
            \draw[dashed, line width=2pt] (2,-4)--(5,-4);
            \draw[color = red,line width=2pt](2,-4)--(3.5,-3.25)node [midway, above] {$1$};
            \draw[line width=2pt](3.5,-3.25)--(4.181,-2.91);
            \draw[color = red,line width=2pt](3.5,-2)--(3.5,-3.25)node [pos=0.6,right] {$1$};
            \draw[line width=2pt](3.5,-3.25)--(3.5,-4);
            \draw[color = red,line width=2pt](5,-4)--(3.5,-3.25)node [pos=0.7, below] {$1$};
            \draw[line width=2pt](3.5,-3.25)--(2.818,-2.91);
            \draw[fill=black] (3.5,-3.25) circle (3pt);
            \draw[fill=white] (2.818,-2.91) circle (3pt);
            \draw[fill=white] (3.5,-4) circle (3pt);
            \draw[fill=white] (4.181,-2.91) circle (3pt);
            \draw[fill=white] (2,-4) circle (3pt);
            \draw[fill=white] (3.5,-2) circle (3pt);
            \draw[fill=white] (5,-4) circle (3pt);
            
            \draw[fill=gray,opacity=0.3,line width=0pt] (6,-4)--(7.5,-2)--(9,-4);
            \draw[dashed, line width=2pt] (6,-4)--(7.5,-2);
            \draw[dashed, line width=2pt] (9,-4)--(7.5,-2);
            \draw[dashed, line width=2pt] (6,-4)--(9,-4);
            \draw[line width=2pt](6,-4)--(7.5,-3.25);
            \draw[color = blue,line width=2pt](7.5,-3.25)--(8.181,-2.91)node [pos=0.8, below] {$1$};
            \draw[line width=2pt](7.5,-2)--(7.5,-3.25);
            \draw[color = blue,line width=2pt](7.5,-3.25)--(7.5,-4)node [pos=0.6, left] {$1$};
            \draw[line width=2pt](9,-4)--(7.5,-3.25);
            \draw[color = blue,line width=2pt](7.5,-3.25)--(6.818,-2.91)node [pos=0.5, above] {$1$};
            \draw[fill=black] (7.5,-3.25) circle (3pt);
            \draw[fill=white] (6.818,-2.91) circle (3pt);
            \draw[fill=white] (7.5,-4) circle (3pt);
            \draw[fill=white] (8.181,-2.91) circle (3pt);
            \draw[fill=white] (6,-4) circle (3pt);
            \draw[fill=white] (7.5,-2) circle (3pt);
            \draw[fill=white] (9,-4) circle (3pt);


            \draw(-7.5,3.2)--(9.5,3.2)--(9.5,-4.5)--(-7.5,-4.5)--(-7.5,3.2);
            \draw(-7.5,2.5)--(1,2.5) node [midway, above] {$\pi$-equivariant Minkowski weights};
            \draw(1,2.5)--(9.5,2.5) node [midway, above] {non-$\pi$-equivariant Minkowski weights};
            \draw(-7.5,-1)--(9.5,-1);
            \draw(1,-4.5)--(1,3.2);
        \end{tikzpicture}
        \caption{Various examples and non-examples of $\pi$-equivariant Minkowski weights.}
        \label{fig: working example weights}
    \end{figure}
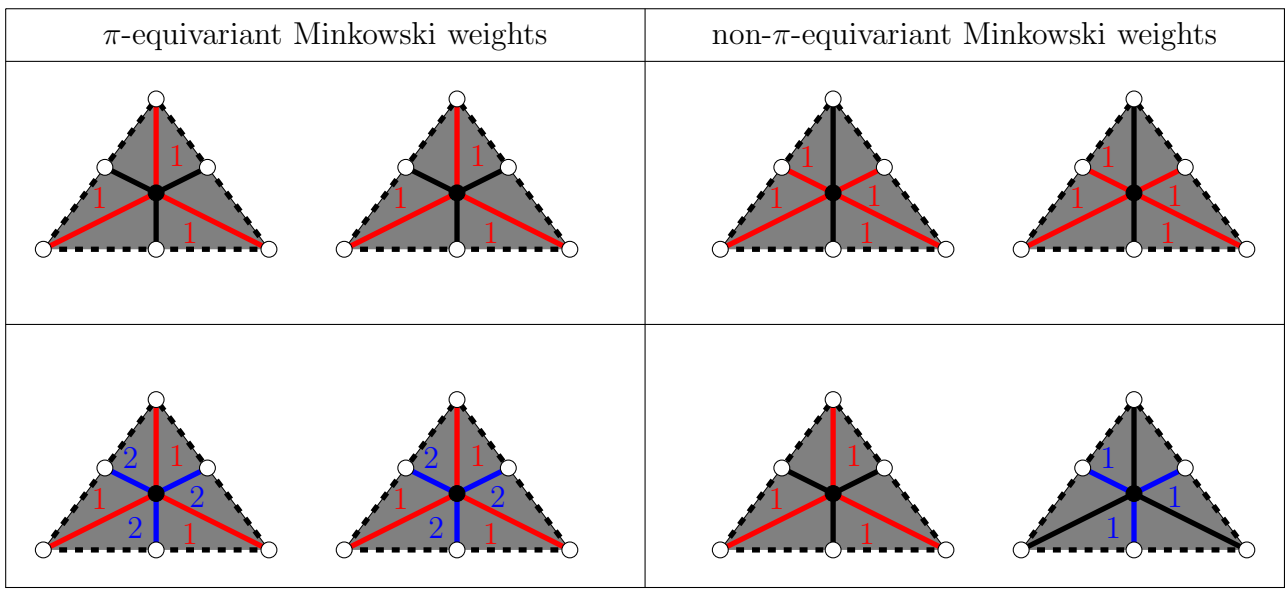
    
\end{example}
\begin{lem}
    Following our previous notation. Being a $\pi$-equivariant Minkowski weight of $S\colon \Phi^\prime\to\Phi$ is equivalent to being invariant under a finite number of permutations of $[\Phi^\prime](k)$. In particular, the set $M_k(\cX_{\pi,X\circ S})$ is an abelian subgroup of $M_k(\Phi^{\prime}_{X\circ S})$.
\end{lem}
\begin{proof}
    The second statement follows from the first. To show the first we proceed with the following observation: If $s$ and $t$ are objects of $\cX$, then the set of isomorphisms $s\to t$ is a torsor under the automorphisms of $t$. The poic-fibration $\pi\colon \Phi\to\cX$ induces a surjective map (that preserves the dimensions of the cones)
    \begin{equation*}
        \pi_*\colon [\Phi]\to [\cX],
    \end{equation*}
    so the fibers of $\pi_*$ partition $[\Phi]$. For each $[x]\in [\cX]$ let $s_{[x]}$ denote a cone of $\Phi$ with $\pi(s_{[x]})$ representing the class $[x] \in [\cX]$. Consider a full set of representatives $R$ of $[\Phi]$ that includes the previously chosen $s_{[x]}$. For each $r\in R$ fix an isomorphism $f_r\colon \pi(r)\to \pi(s_{[\pi(r)]})$ of $\cX$. Then a weight $\omega\in M_k(\Phi^\prime_{X\circ S})$ is a $\pi$-equivariant Minkowski weight if and only if 
    \begin{itemize}
        \item $\omega$ is $f_r$-stable, for every $r\in R$,
        \item for every $[x]\in [\cX]$ the weight $\omega$ is $f$-stable, for every $f\in \Aut_\cX(\pi(s_{[x]}))$.
    \end{itemize}
    From this the result follows.
\end{proof}

\begin{lem}\label{lem: compatibility of subdivisions}
Following Notation \ref{nota: notacion linear pf}, let $S^\prime\colon \Phi^{\prime\prime}\to \Phi^\prime$ be a subdivision such that the composition $\Phi^{\prime\prime}\to\Phi$ is $\pi$-compatible. The map $(S^\prime)^*\colon M_k(\Phi^\prime_{X\circ S})\hookrightarrow M_k(\Phi^{\prime\prime}_{X\circ S\circ S^\prime})$ maps $\pi$-equivariant Minkowski weights to $\pi$-equivariant Minkowski weights.

\end{lem}
\begin{proof}
    Consider a $\pi$-equivariant Minkowski weight $\omega\in M_k(\cX_{\pi,X\circ S})$. We want to show that $(S^\prime)^*\omega\in M_k(\cX_{\pi,X\circ S^\prime\circ S})$. Let $p$ and $q$ be objects of $\Phi$ with $f\colon \pi(p)\to \pi(q)$ an isomorphism of $\cX$. Observe first that $(S^\prime\circ S)^{-1}(p)$ is a subdivision of $S^{-1}(p)$ (resp. $(S^\prime\circ S)^{-1}(q)$ is a subdivision of $S^{-1}(q)$), and furthermore we have the following commutative diagram
    \begin{equation*}
        \begin{tikzcd}
            \pi_p((S^\prime\circ S)^{-1}(p))\arrow{d}[swap]{S^\prime}\arrow{r}{\cX(f)}&\pi_q((S^\prime\circ S)^{-1}(q))\arrow{d}{S^\prime}\\
            \pi_p(S^{-1}(p))\arrow{r}[swap]{\cX(f)}&\pi_q(S^{-1}(q))
        \end{tikzcd}.
    \end{equation*}
    From this it follows that $(S^\prime)^*(\omega\circ w_f) = ((S^\prime)^*\omega)\circ w_f$, and hence $(S^\prime)^*\omega$ is also $\pi$-compatible.
\end{proof}
\begin{const}
    Let $\pi\textnormal{-}\Subd(\Phi)$ denote the full subcategory of $\Subd(\Phi)$ given by $\pi$-compatible subdivisions of $\Phi$. Then the $\pi$-equivariant Minkowski weights give rise to the functor
    \begin{equation*} M_k(\cX_{\pi,\bullet})\colon  \pi\textnormal{-}\Subd(\Phi)^{\op}\to \text{Ab}.\end{equation*}
    As in the case of tropical cycles, the colimit of this functor is representable.
\end{const}
\begin{defi}
The \emph{group of tropical $k$-cycles of the poic-fibration}  $\pi\colon \Phi_X\to \cX$ is defined as  
\begin{equation*}
Z_k(\cX_{\pi,X}) \colon = \varinjlim\limits M_k(\cX_{\pi,\bullet}).
\end{equation*}
\end{defi}

\begin{lem}
    Suppose $\pi\colon\Phi_X\to\cX$ is a linear poic-fibration. The natural inclusions $M_k(\cX_{\pi,X\circ S})\subset M_k(\Phi^\prime_{X\circ S})$, where $S\colon \Phi^\prime\to\Phi$ is a $\pi$-compatible subdivision, give rise to an injective linear map
        \begin{equation*}
            \pi_k^*\colon Z_k(\cX_{\pi,X})\to Z_k(\Phi_X).
        \end{equation*}
\end{lem}
\begin{proof}
   We remark that any subdivision of $\Phi$ can be refined into a $\pi$-compatible subdivision, so that $Z_k(\Phi_X)$ can be computed as the direct limit over the $\pi$-compatible subdivisions. Hence, Lemma \ref{lem: compatibility of subdivisions} shows that there is a map
    \begin{equation}
        \pi^*\colon Z_k(\cX_{\pi,X})\to Z_k(\Phi_X), \label{eq: tropical pi cycles are subgroup of tropical cycles}
    \end{equation}
    and since the maps
    \begin{equation*}
        M_k(\cX_{\pi,X\circ S})\hookrightarrow M_k(\Phi^\prime_{X\circ S}) 
    \end{equation*}
    are always injective, it follows that \eqref{eq: tropical pi cycles are subgroup of tropical cycles} is also injective.
\end{proof}

The notion of irreducibility from linear poic-complexes can be naturally extended to linear poic-fibrations.
\begin{defi}
    A linear poic-fibration $\pi\colon \Phi_X\to \cX$, with $\Phi$ pure of dimension $n$, is said to be \emph{irreducible} if $Z_n(\cX_{\pi,X})$ is free of rank $1$.
\end{defi}
\begin{remark}
    If $\Phi_X$ is pure of dimension $n$ and irreducible, then a linear poic-fibration $\pi\colon \Phi_X\to \cX$ is either irreducible or has trivial $n$-cycles.
\end{remark}

\subsection{Pushforwards for poic-fibrations.} Similar to subsection \ref{ssec: pushforwards}, we want to pushforward weights of linear poic-fibrations through morphisms. As is expected, in order to do so we need to impose additional conditions on the morphism. Recall that a morphism of poic-fibrations was defined simply as a morphism of pairs yielding a commutative square. However, up to this point, we have not dealt with any morphism of poic-fibrations. Hence, in the interest of readability and clarity, we first dedicate some space to fix notations in this situation.

Consider first two poic-fibrations $\pi\colon\Phi\to\cX$ and $\rho\colon\Psi\to\cY$. For simplicity of notation, we say thus that a morphism of poic-fibrations $\mathfrak{f}\colon \pi\to\rho$ consists of: 
\begin{itemize}
    \item A morphism of poic-complexes $\mathfrak{f}^\pcomplexes\colon \Phi\to\Psi$, where the superscript c stands for complexes (as in poic-complexes).
    \item A morphism of poic-spaces $\mathfrak{f}^\pspaces\colon\cX\to \cY$, where the superscript s stands for spaces (as in poic-spaces).
    \item These are subject to the condition $\rho\circ \mathfrak{f}^\pcomplexes = \mathfrak{f}^\pspaces\circ \pi$.
\end{itemize}
If, in addition, $\pi\colon \Phi_X\to\cX$ and $\rho\colon\Psi_Y\to\cY$ are linear poic-fibrations, then we require $\mathfrak{f}^\pcomplexes$ to also be a morphism linear poic-complexes $\mathfrak{f}\colon \Phi_X\to\Psi_Y$. 

\begin{defi}\label{defi: proper morphism poic-fibrations}
    Suppose $\pi\colon\Phi\to\cX$ and $\rho\colon\Psi\to\cY$ are poic-fibrations. A morphism of poic-fibrations $\mathfrak{f}\colon \pi\to\rho$ is called \emph{weakly proper} (resp. \emph{proper}) if:
    \begin{enumerate}
        \item $\mathfrak{f}^\pcomplexes$ is a weakly proper (resp. proper) morphism of poic-complexes.
        \item $\mathfrak{f}^\pspaces$ satisfies the following lifting property: If $f\colon s \to t$ is an isomorphism of $\cY$, then for every cone $s^\prime$ of $\cX$ with $\mathfrak{f}^\pspaces(s^\prime)=s$ there exists a unique isomorphism $f^\prime\colon s^\prime\to t^\prime$ of $\cX$ with $\mathfrak{f}^\pspaces(f^\prime)=f$ (in particular $\mathfrak{f}^\pspaces(t^\prime)=t$).
    \end{enumerate}
\end{defi}

The following proposition is quite technical, but is the basis for the definition of the pushforward of tropical cycles of a linear poic-fibration through a proper morphism. We briefly explain the underlying situation. Suppose $\pi\colon\Phi_X\to\cX$ and $\rho\colon\Psi_Y\to \cY$ are linear poic-fibrations, $\mathfrak{f}\colon\pi\to\rho$ is a morphism of linear poic-fibrations, and $k\geq0$ is an integer. Our general goal is to pushforward $\pi$-equivariant tropical cycles to $\rho$-equivariant tropical cycles. Of course, this depends on the properness of $\mathfrak{f}^{\pcomplexes}$, so we focus on a very special case. Consider $\omega\in M_k(\cX_{\pi,X\circ S})$, where $S\colon\Phi^\prime\to\Phi$ is a $\pi$-compatible subdivision such that $S\circ \mathfrak{f}^\pcomplexes$ is weakly proper. This subdivision gives the weakly proper morphism of linear poic-complexes $\mathfrak{f}\circ S\colon \Phi^\prime_{X\circ S}\to\Psi_Y$, and for the pushforward of $\omega$ through this morphism we need to consider a $(S\circ \mathfrak{f}^\pcomplexes)$-fine subdivision of $\Psi$, say $T\colon \Psi^\prime\to \Psi$. In this case, from \eqref{eq: pforward minkowski weights P-fine subdiv} we will have $((S\circ \mathfrak{f}^\pcomplexes)_T)_*\omega \in M_k(\Psi^\prime_{Y\circ T})$. However, since the subdivision $T\colon\Psi^\prime\to\Psi$ is not necessarily compatible with the poic-fibration, we cannot say (and it is not even defined) that $((S\circ \mathfrak{f}^\pcomplexes)_T)_*\omega$ is $\rho$-equivariant. For this reason, we have to further consider a subdivision $T^\prime\colon\Psi^{\prime\prime}\to\Psi^\prime$ such that $\Psi^{\prime\prime}\to\Psi$ is $\rho$-compatible, and show that $(T^\prime)^*\left(((S\circ \mathfrak{f}^\pcomplexes)_T)_*\omega\right)$ is $\rho$-equivariant.

\begin{prop}\label{prop: pushforward minkowkski weights of fibration}
    Suppose that $\pi\colon\Phi_X\to\cX$ and $\rho\colon\Psi_Y\to\cY$ are linear poic-fibrations, $\mathfrak{f}\colon\pi\to\rho$ is a morphism of poic-fibrations, $S\colon \Phi^\prime\to\Phi$ is a $\pi$-stable subdivision, and $T\colon\Psi^\prime\to\Psi$ is a $(S\circ \mathfrak{f}^{\pcomplexes})$-fine subdivision. Suppose that $\mathfrak{f}^\pcomplexes\circ S$ is weakly proper and $T^\prime\colon \Psi^{\prime\prime}\to\Psi^\prime$ is such that $T^\prime\circ T$ is $\rho$-stable. For any integer $k\geq0$, if $\omega\in M_k(\cX_{\pi, X\circ S})$, then $(T^\prime)^*((S\circ \mathfrak{f}^{\pcomplexes})_T)_*\omega$ is $\rho$-equivariant.
\end{prop}
\begin{proof}
    it is clear that $((S\circ \mathfrak{f}^{\pcomplexes})_T)_*\omega\in M_k(\Psi^\prime_{Y\circ T})$, and therefore $(T^\prime)^*((S\circ \mathfrak{f}^{\pcomplexes})_T)_*\omega\in M_k(\Psi^\prime_{Y\circ T\circ T^\prime})$. By definition of the map $(T^\prime)^*$ we can assume for the sake of simplicity, and without loss of generality, that the subdivision $T\colon \Psi^\prime\to \Psi$ is itself $\rho$-compatible (the new equations to check either correspond to previous ones or are just constant weights for opposite direction vectors). Therefore, we must only check that $((S\circ \mathfrak{f}^{\pcomplexes})_T)_*\omega$ is $\rho$-equivariant.\\ 
    Suppose $p$ and $q$ are objects of $\Psi$ and $g\colon \rho(p)\to \rho(q)$ is an isomorphism of $\cY$, we must check then that $((S\circ \mathfrak{f}^{\pcomplexes})_T)_*\omega\circ b_g = ((S\circ \mathfrak{f}^{\pcomplexes})_T)_*\omega$. As before, we regard $[T^{-1}(p)]$ and $[T^{-1}(q)]$ as subsets of $[\Psi^\prime]$, and let $p^\prime$ and $q^\prime$ be $k$-dimensional cones of $\Psi^\prime$ such that $[p^\prime]\in [T^{-1}(p)]$ and $ [q^\prime]\in [T^{-1}(q)]$ with $b_g([p^\prime]) = [q^\prime]$. Consider the following subsets of $[\Phi](k)$:
    \begin{align*}
        (S\circ \mathfrak{f}^{\pcomplexes})^{-1}([p^\prime]):=\{[s] : ((S\circ \mathfrak{f}^{\pcomplexes})_T)_*[s] = [p^\prime]\},\\
        (S\circ \mathfrak{f}^{\pcomplexes})^{-1}([q^\prime]):=\{[s] : ((S\circ \mathfrak{f}^{\pcomplexes})_T)_*[s] = [q^\prime]\}.
    \end{align*}
    Both of these sets are partitioned by the fibers of $(\mathfrak{f}^{\pspaces})^{-1}([\rho(T(p^\prime))])=(\mathfrak{f}^{\pspaces})^{-1}([\rho(T(p^\prime))])$. More precisely, let $y$ be an object of $\cY$ with $[y]=[\rho(p)]=[\rho(q)]$, then following the notation,
    \begin{equation*}
        [y]=[\rho(T(p^\prime))] = [\rho(T(q^\prime))],
    \end{equation*}
    and we obtain the partitions:
    \begin{align}
        (S\circ \mathfrak{f}^{\pcomplexes})^{-1}([p^\prime])=\bigsqcup_{[x]\in(\mathfrak{f}^{\pspaces})^{-1}([y])} \{[s] : ((S\circ \mathfrak{f}^{\pcomplexes})_T)_*[s] = [p^\prime] \text{ and }\pi([s])=[x]\},\label{eq: part 1}\\
        (S\circ \mathfrak{f}^{\pcomplexes})^{-1}([q^\prime])=\bigsqcup_{[x]\in(\mathfrak{f}^{\pspaces})^{-1}([y])} \{[s] : ((S\circ \mathfrak{f}^{\pcomplexes})_T)_*[s] = [q^\prime] \text{ and }\pi([s])=[x]\}.\label{eq: part 2}
    \end{align}
    The condition on lifting isomorphisms of $\mathfrak{f}^\pspaces$ implies that $b_f$ gives rise to a bijection
    \begin{equation*}
        \widehat{b_f}\colon (S\circ \mathfrak{f}^{\pcomplexes})^{-1}([p^\prime])\to(S\circ \mathfrak{f}^{\pcomplexes})^{-1}([q^\prime]),
    \end{equation*}
    such that:
    \begin{itemize}
        \item $\widehat{b_f}$ preserves the partitions induced by $(\mathfrak{f}^{\pspaces})^{-1}([y])$.
        \item If $[s]\in (S\circ \mathfrak{f}^{\pcomplexes})^{-1}([p^\prime])$, then:
        \begin{itemize}
            \item We have equality between the weights $\omega([s]) = \omega(\widehat{b_f}([s]))$, because $\omega$ is a $\pi$-equivariant Minkowski weight. 
            \item We have equality between the indices $\left[T_{p^\prime}N^{p^\prime} : \mathfrak{f}^{\pcomplexes}_sN^s\right] = \left[T_{q^\prime}N^{q^\prime} : \mathfrak{f}^{\pcomplexes}_{\widehat{b}_f(s)}N^{\widehat{b}_f(s)}\right]$, because these are related by integral linear isomorphisms.
        \end{itemize}
    \end{itemize}
    It follows then that $((S\circ \mathfrak{f}^{\pcomplexes})_T)_*\omega\in M_k(\cY_{\rho, Y\circ T})$, and hence the proposition.
\end{proof}

In general, suppose that $\pi\colon \Phi_X\to \cX$ and $\rho\colon \Psi_Y\to\cY$ are poic-fibrations, and $\mathfrak{f}\colon \pi\to\rho$ is a proper morphism of linear poic-fibrations. Observe that if $S\colon \Phi^\prime\to\Phi$ is a $\pi$-compatible subdivision, then Lemma \ref{lem: P-fine subdivisions are final} shows that there is a $(S\circ \mathfrak{f}^\pcomplexes)$-fine subdivision and Lemma \ref{lem: refining to pi-compatible subdivisions} shows that there is $T^\prime\colon \Psi^{\prime\prime}\to \Psi^\prime$ such that $T\circ T^\prime$ is $\rho$-compatible. Therefore, starting from a $\pi$-compatible subdivision, it is always possible to place ourselves in the situation of Proposition \ref{prop: pushforward minkowkski weights of fibration}. Now, since $\mathfrak{f}^\pcomplexes\colon \Phi_X\to \Psi_Y$ is proper, it follows that the map $\mathfrak{f}$ induces a linear map
\begin{equation*}
    ((S\circ \mathfrak{f}^\pcomplexes)_T)_*\colon M_k(\cX_{\pi,X\circ S})\to M_k(\cX_{\pi,Y\circ \mathfrak{f}^\pcomplexes\circ S}), 
\end{equation*}
which together with $(T^\prime)^*$ and the natural inclusion $M_k(\cY_{\rho, Y\circ T\circ T^\prime})\subset Z_k(\cY_{\rho,Y})$ gives rise to a linear map
\begin{equation*}
    (\mathfrak{f}_S)_*\colon M_k(\cX_{\pi,X\circ S})\to Z_k(\cY_{\rho,Y}).
\end{equation*}
Just as in the previous situations, taking subsequent subdivisions will give rise to commutative maps, so that these maps give rise to a linear map
\begin{equation}
    \mathfrak{f}_*\colon Z_k(\cX_{\pi,X})\to Z_k(\cY_{\rho,Y}).\label{eq: pushforward of fibrations}
\end{equation}
We call \eqref{eq: pushforward of fibrations} the \emph{pushforward map of $\mathfrak{f}$}. This whole situation also gives rise to the following commutative diagram
\begin{equation*}
    \begin{tikzcd}
         Z_k(\cX_{\pi,X})\arrow[hookrightarrow]{d}\arrow{rr}{\mathfrak{f}_*}&&Z_k(\cY_{\rho,Y})\arrow[hookrightarrow]{d}\\
        Z_k(\Phi_X) \arrow{rr}[swap]{\mathfrak{f}^\pcomplexes_*}&& Z_k(\Psi_Y).
    \end{tikzcd}
\end{equation*}
Separately, if $\mathfrak{f}$ is just weakly proper in dimension $k$, then from Proposition \ref{prop: pushforward minkowkski weights of fibration} we obtain a linear map
\begin{equation}
    \mathfrak{f}_*\colon M_k(\cX_{\pi,X})\to Z_k(\cY_{\rho,Y}), \label{eq: weak pushforward tropcyc}
\end{equation}
which we call the \emph{weak pushforward of $\mathfrak{f}$}. In this case, we obtain the analogous commutative diagram
\begin{equation*}
    \begin{tikzcd}
         M_k(\cX_{\pi,X})\arrow[hookrightarrow]{d}\arrow{rr}{\mathfrak{f}_*}&&Z_k(\cY_{\rho,Y})\arrow[hookrightarrow]{d}\\
        M_k(\Phi_X) \arrow{rr}[swap]{\mathfrak{f}^\pcomplexes_*}&& Z_k(\Psi_Y).
    \end{tikzcd}
\end{equation*}
\subsection{The Spanning Tree cover.} \label{ssec: spanning tree cover.} Suppose $A$ is a finite set and $g\geq 0$ with $2g+\#A-2>0$. We will describe a poic-fibration over the poic-space $\Mtrop_{g,A}$. For this we let $\mathbbm{g}$ denote the set $\{1,\dots,g,1^*,\dots,g^*\}$, where $i^*$ is just a symbol for $1\leq i\leq g$. We seek to produce a genus-$g$ $A$-marked discrete graph out of a $(A\sqcup \mathbbm{g})$-marked tree by joining the $i$- and $i^*$-legs into a new edge.

\begin{const}
    Consider an object $T$ of $\bbG_{0,A\sqcup \mathbbm{g}}$. The graph $T$ has its own involution map $\iota_T$, which among its fixed points contains the legs $L(T)$ of $T$ (here for the sake of simplicity, we regard $L(T)\subset F(T)$). Consider the following map
    \begin{equation*} \iota_{\st_{g,A}(T)}\colon  F(T)\to F(T), h\mapsto\begin{cases}
        \iota_T(h),& h\not\in L(T),\\
        \ell_i,& h=\ell_i \,(i\in A),\\
        \ell_{i^*},& h=\ell_{i} \,(1\leq i\leq g),\\
        \ell_{i},& h=\ell_{i^*} \,(1\leq i\leq g).
    \end{cases}\end{equation*}
    This is an involution of $F(T)$, which together with the root map $r_T$ and the marking on the $A$-legs gives rise to a connected $A$-marked graph $\st_{g,A}(T)$ of genus-$g$ (here the $\st$ stands for \emph{spanning tree}). 
\end{const}

\begin{prop}\label{prop: st is ess surj}
    The above construction gives rise to an essentially surjective functor
    \begin{equation*}
    {\st_{g,A}}\colon  \bbG_{0,A\sqcup \mathbbm{g}}\to \bbG_{g,A}.
    \end{equation*}
\end{prop}
\begin{proof}
    Contraction of an edge of a tree $T$ becomes contraction of an edge of the graph $\st_{g,A}(T)$, therefore the functoriality readily follows. To show essential surjectivity, let $G$ be any connected genus-$g$ $A$-marked graph, and let $T$ denote a spanning tree of $G$. For $1\leq i \leq g$ let $e_1,\dots,e_g$ denote edges of $G$ that do not lie in $T$. Let $T_G$ denote the tree obtained from $T$ by grafting at the incident vertices of $e_i$ two marked legs $\ell_{i}$ and $\ell_{i^*}$ in any order. The tree $T_G$ is connected, has $(A\sqcup \mathbbm{g})$-marked legs, and by construction satisfies $\st_{g,A}(T_G)\cong G$.
\end{proof}
\begin{example}
    In the case of $g=2$ and $n=0$ the functor ${\st_{2,0}}$ can be represented vertically between isomorphism classes, excluding the morphisms, as depicted in Figure \ref{fig:fig 3}.
    \begin{figure}[htbp]
        \begin{tikzpicture}
            \draw[fill =black](0,6) circle (3pt);
            \draw[dashed, line width = 1pt](0,6)--(1,7);
            \node[] at (1.1,7.1) {$1$};
            \draw[dashed, line width = 1pt](0,6)--(1,5);
            \node[] at (1.1,4.9) {$3$};
            \draw[dashed, line width = 1pt](0,6)--(-1,7);
            \node[] at (-1.1,7.1) {$2$};
            \draw[dashed, line width = 1pt](0,6)--(-1,5);
            \node[] at (-1.1,4.9) {$4$};

            \draw[fill =black](-5.5,8) circle (3pt);
            \draw[fill =black](-4.5,8) circle (3pt);
            \draw[] (-5.5,8)--(-4.5,8);
            \draw[dashed, line width = 1pt](-4.5,8)--(-3.5,9);
            \node[] at (-3.4,9.1) {$1$};
            \draw[dashed, line width = 1pt](-4.5,8)--(-3.5,7);
            \node[] at (-3.4,6.9) {$2$};
            \draw[dashed, line width = 1pt](-5.5,8)--(-6.5,9);
            \node[] at (-6.6,9.1) {$3$};
            \draw[dashed, line width = 1pt](-5.5,8)--(-6.5,7);
            \node[] at (-6.6,6.9) {$4$};

            \draw[fill =black](-5.5,4) circle (3pt);
            \draw[fill =black](-4.5,4) circle (3pt);
            \draw[] (-5.5,4)--(-4.5,4);
            \draw[dashed, line width = 1pt](-4.5,4)--(-3.5,5);
            \node[] at (-3.4,5.1) {$1$};
            \draw[dashed, line width = 1pt](-4.5,4)--(-3.5,3);
            \node[] at (-3.4,2.9) {$4$};
            \draw[dashed, line width = 1pt](-5.5,4)--(-6.5,5);
            \node[] at (-6.6,5.1) {$2$};
            \draw[dashed, line width = 1pt](-5.5,4)--(-6.5,3);
            \node[] at (-6.6,2.9) {$3$};

            \draw[fill =black](5.5,6) circle (3pt);
            \draw[fill =black](4.5,6) circle (3pt);
            \draw[dashed, line width = 1pt](5.5,6)--(6.5,7);
            \node[] at (6.6,7.1) {$1$};
            \draw[dashed, line width = 1pt](5.5,6)--(6.5,5);
            \node[] at (6.6,4.9) {$3$};
            \draw[dashed, line width = 1pt](4.5,6)--(3.5,7);
            \node[] at (3.4,7.1) {$2$};
            \draw[dashed, line width = 1pt](4.5,6)--(3.5,5);
            \node[] at (3.4,4.9) {$4$};
            \draw[] (4.5,6)--(5.5,6);

            \draw[fill=none](-5,-1.5) circle (1);
            \draw[fill=black](-5,-0.5) circle (3pt);
            \draw[fill=black](-5,-2.5) circle (3pt);
            \draw[](-5,-2.5)--(-5,-0.5);

            \draw[fill=none](-1,-1.5) circle (1);
            \draw[fill=none](1,-1.5) circle (1);
            \draw[fill=black](0,-1.5) circle (3pt);

            \draw[fill=black] (5.25,-1.5) circle (3pt);
            \draw[fill=black] (4.75,-1.5) circle (3pt);
            \draw[fill=none] (4,-1.5) circle (0.75);
            \draw[fill=none] (6,-1.5) circle (0.75);
            \draw[](5.25,-1.5)--(4.75,-1.5);

            \draw (-7,9.5)--(7,9.5) node [midway,below] {Isomorphism classes of $\bbG_{0,4}$.};
            \draw (-7,9.5)--(-7,2);
            \draw (-3,9.5)--(-3,2);
            \draw (-7,6)--(-3,6);
            \draw (3,9.5)--(3,2);
            \draw (7,9.5)--(7,2);
            \draw (-7,2)--(7,2);
            
            \draw (-7,0)--(7,0);
            \draw (-7,0)--(-7,-3.5);
            \draw (-3,0)--(-3,-3.5);
            \draw (3,0)--(3,-3.5);
            \draw (7,0)--(7,-3.5);
            \draw (-7,-3.5)--(7,-3.5) node [midway, above] {Isomorphism classes of $\bbG_{2,0}$.};

            \draw[-to, line width = 2pt] (-5,1.5)--(-5,0.5);
            \draw[-to, line width = 2pt] (0,1.5)--(0,0.5);
            \draw[-to, line width = 2pt] (5,1.5)--(5,0.5);
            
        \end{tikzpicture}
        \caption{The functor $\st_{2,0}$ at work between the corresponding isomorphism classes.}
        \label{fig:fig 3}
    \end{figure}
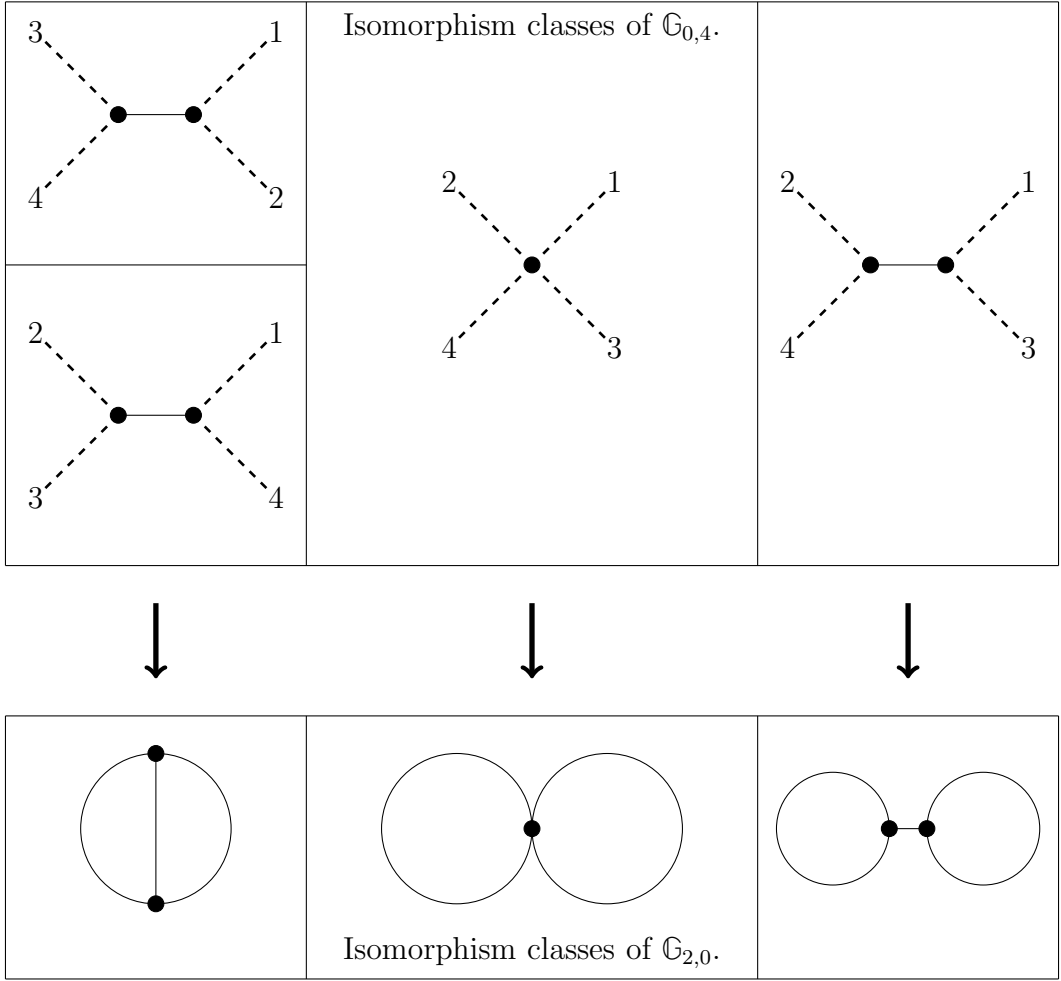
\end{example}

\begin{const}
    Let $T$ be any graph of $\bbG_{0,A\sqcup\mathbbm{g}}$. We denote by $\ST_T$ the poic $\sigma_T\times \bbR_{>0}^g$, and observe that this defines a functor 
    \begin{equation}
        \ST_{g,A}\colon \bbG_{0,A\sqcup\mathbbm{g}}^\op\to \POIC,
    \end{equation}
    which is, in particular, a poic-complex. Since colimits commute with finite products, it is clear that $\colim \left|\ST_{g,A}\right|  \cong \cM_{0,A\sqcup \mathbbm{g}}^\trop\times\bbR_{>0}^g$. 
    In addition, the $\dist_{A\sqcup\mathbbm{g}}$ morphism and the natural projection give a morphism of poic-complexes 
    \begin{equation}
        D_{g,A}\colon \ST_{g,A}\to \underline{\left(N_{\dist_{A\sqcup\mathbbm{g}}}\right)_\bbR}.\label{eq: distance morphism general}
    \end{equation}
\end{const}

\begin{defi}
    The morphism of poic-complexes \eqref{eq: distance morphism general} makes $\ST_{g,A}$ into a linear poic-complex. We call this the \emph{linear poic-complex of spanning trees of genus-$g$ $A$-marked graphs} and will just denote it by $\ST_{g,A}$, since it will always be considered as a linear poic-complex in this way.
\end{defi}

\begin{const}
    Consider an arbitrary discrete graph $T$ of $\bbG_{0,A\sqcup\mathbbm{g}}$. For $1\leq i\leq g$ let $e_i$ denote the edge of $\st_{g,A}(T)$ given by the orbit $\{\ell_{i},\ell_{i^*}\}$. Then, the discrete graph $\st_{g,A}(T)$ has the set of edges
    \begin{equation*}
        E(\st_{g,A}(T)) = E(T) \cup \{e_1,\dots,e_g\},
    \end{equation*} 
    and the cone $\sigma_T$ can therefore be identified with a face of $\sigma_{\st_{g,A}(T)}$. Furthermore, the poic $\bbR_{>0}^g$ corresponds to the relative interior of a face of $\sigma_{\st_{g,A}(T)}$ by identifying the $i$th-coordinate with the $e_i$-coordinate (corresponding to the edge $e_i$). Let
    \begin{equation}
        \eta_{\st_{g,A},T }\colon \ST_T\to \sigma_{\st_{g,A}(T)} \label{eq: poic morphisms from spanning tree cover}
    \end{equation}
    denote the poic-morphism given by the product of the previosuly described morphisms.
\end{const}
\begin{prop}\label{prop: spanning tree fibration}
    The poic-morphisms \eqref{eq: poic morphisms from spanning tree cover} give rise to a natural transformation
    \begin{equation*}
    \eta_{\st_{g,A}}\colon  \ST_{g,A} \to \Mtrop_{g,A}\circ \st_{g,A},
    \end{equation*}
    which together with $\st_{g,A}$ gives a linear poic-fibration $\st_{g,A}\colon \Mtrop_{0,A\sqcup\mathbbm{g}}\to\Mtrop_{g,A}$.
\end{prop}
\begin{proof}
    The essential surjectivity of the functor has already been stablished in Proposition \ref{prop: st is ess surj}. For a tree $T$ of $\bbG_{0,A\sqcup \mathbbm{g}}$, the map $\eta_{\st_{g,A},T}$ identifies the relative interiors of the poics $\ST_T$ and $\sigma_{\st_{g,A},T}$. Therefore, the third condition is also fulfilled. It remains to show the second condition of a poic-fibration. By definition of the structure of the morphisms of $\bbG_{g,A}$, it is sufficient to show the lifting condition in the following set-up: Let $T$ be an object of $\bbG_{0,A\sqcup\mathbbm{g}}$, $G$ an object of $\bbG_{g,A}$, and consider a morphism $\ST_{g,A}(T)\to G$ given by: the contraction of an edge $e\in E(G)$ and a bijective map of graphs $f\colon \st_{g,A}(T)\to G/e$. Let $V\in V(T)$ denote the vertex of $T$ such that $f(V) = V_e$. Consider the object $T^\prime$ of $\bbG_{0,A\sqcup\mathbbm{g}}$ obtained by resolving $V$ in the same way as $V_e\in V(G/e)$ to $G$ and let $e^\prime$ denote the resolving edge. So that there is a natural bijective map of graphs $g\colon T\to T^\prime/e^\prime$, and the map $f$ can be extended to a bijective map of graphs
    \begin{equation*}
        f^\prime\colon   \st_{g,A}(T^\prime)\to G,
    \end{equation*}
    by additionally specifying $e^\prime\to e$. Then $e^\prime\in E(G)$ and the bijective map of graphs $g\colon T\to T^\prime$ define a morphism $T\to T^\prime/e^\prime$ that satisfies the required condition.
\end{proof}

\begin{defi}
    We call the linear poic-fibration $\st_{g,A}\colon \ST_{g,A}\to\cM^\trop_{g,A}$ of Proposition \ref{prop: spanning tree fibration} the \emph{spanning tree fibration of genus-$g$ $A$-marked graphs}.
\end{defi}
\begin{prop}
    The poic-complex $\ST_{g,A}$ is pure of rank $3g+\#A-3$ and the poic-fibration $\st_{g,A}$ is irreducible. 
\end{prop}
\begin{proof}
    Observe that $\#\left(A\sqcup\mathbbm{g}\right) = 2g+\#A$, hence the linear poic-complex $\Mtrop_{0,A\sqcup\mathbbm{g}}$ is pure of dimension $2g+\#A-3$, and $\ST_{g,A}$ is pure of dimension $2g+\#A-3+g = 3g+\#A-3$. Irreducibility of the poic-fibration follows from the analogous fact for $\Mtrop_{0,A\sqcup \mathbbm{g}}$, and the fact that the constant weight $1$ is $\st_{g,A}$-equivariant.
\end{proof}
\subsection{Forgetting the marking.} Let $A$ denote a finite set and let $g\geq0$ be an integer with $2g+\#A-2>1$. Suppose $A$ is non-empty and let $a\in A$ be an element. This subsection deals with the construction of a weakly proper morphism of linear poic-fibrations $\st_{g,A}\to\st_{g,A\backslash a}$, given by forgetting the $a$-marked leg. We carry out this construction meticulously first for the poic-spaces $\Mtrop_{g,A}$ and $\Mtrop_{g,A\backslash a}$, then for the linear poic-complexes $\Mtrop_{0,A}$ and $\Mtrop_{0,A\backslash a}$, and finally for the linear poic-fibrations.
\begin{const}
    Let $G$ be an object of $\bbG_{g,A}$. We will describe the construction of an object $\ft_{a}G$ of $\bbG_{g,A\backslash a}$ given by forgetting the $a$-leg of $G$. Let $V_a\in V(G)$ denote the vertex incident to $\ell_a(G)$.
    \begin{enumerate}
        \item If $\val (V_a) >3$, then set $F\left(\ft_{a}G\right) := F(G)\backslash\ell_a(G)$. The root, involution, and marking maps of $G$ restrict to respective root, involution and marking maps of $F(\ft_{a}G)$, and hence define an $A\backslash a$-marked graph $\ft_{a}G$.
        \item If $\val (V_a) =3$ with $V_a$ incident to two edges, then $r_G^{-1}(V_a) = \ell_a\cup\{f_1,f_2,V_a\}$. In this case, we set $F(\ft_{a}G) = F(G)\backslash r_G^{-1}(V_a)$ and $r_{\ft_{a}G} = r_G$. To define $\iota_{\ft_{a}G}$ it is only necessary to specify where to map $\iota_G(f_1)$ and $\iota_G(f_2)$, for which we simply put
        \begin{align*}
            &\iota_{\ft_{a}G}\left(\iota_G(f_1)\right) = \iota_G(f_2),
            &\iota_{\ft_{a}G}\left(\iota_G(f_2)\right) = \iota_G(f_1).
        \end{align*}
        In this case, the triple $(F(\ft_{a}G),r_{\ft_{a}G},\iota_{\ft_{a}G})$ defines an $A\backslash a$-marked graph $\ft_{a}G$.
        \item If $\val V_a = 3$ with $V_a$ incident to an additional leg, then $r_G^{-1}(V_a) = \ell_a\cup \{l,f,V_a\}$ with $\{l\}\in L(G)$. In this situation, set $F(\ft_aG) =F(G)\backslash \left(\ell_a\cup \{f,\iota_G(f), V_a\}\right)$ and $\iota_{\ft_aG} = \iota_G$. To define $r_{\ft_aG}$ it is sufficient to specify where to map $l$, for which we set $r_{\ft_aG}(l) = r_G(\iota_G(f))$. Therefore, the triple $(F(\ft_a(G)),r_{\ft_a(G)},\iota_{\ft_aG})$ defines an $A\backslash a$-marked graph $\ft_aG$.
    \end{enumerate}
    Let $H$ denote an additional object of $\bbG_{g,A}$, and let $g\colon G\to H$ denote a morphism thereof. We define a morphism $\ft_{a} g\colon\ft_{a}G\to \ft_{a}H$ as follows. Let $W_a\in V(H)$ denote the vertex $\partial\ell_a(H)$. Observe first that we must only study the behavior of $g$ around $g^{-1}(W_a)$, and as such we can assume without loss of generality that the inverse image of any other vertex is a single vertex. Namely, any edge contraction goes to $W_a$ and nowhere else. We proceed as above and separate into the same three cases:
    \begin{enumerate}
        \item Suppose that $\val V_a>3$. Then $\val W_a \geq\val V_a>3$, and in this case $g$ directly induces $\ft_ag\colon\ft_aG\to\ft_aH$.
        \item Suppose that $\val V_a=3$ with $V_a$ incident to two edges, so that, as before, $r_G^{-1}(V_a) = \ell_a\cup\{f_1,f_2,V_a\}$. If $r^{-1}_G(V_a)\subset g^{-1}(W_a)$, then $g$ directly induces $\ft_{a}g\colon\ft_{a}G\to \ft_{a}H$. We make explicit how to proceed when $r^{-1}_G(V_a)\not\subset g^{-1}(W_a)$:
        \begin{itemize}
            \item If $f_1\not\in g^{-1}(W_a)$ but $f_2\in g^{-1}(W_a)$ (or equivalently $f_2\not\in g^{-1}(W_a)$ but $f_1\in g^{-1}(W_a)$), then $\iota_G(f_1)\not\in (W_a)$ and $\iota_G(f_2)\in g^{-1}(W_a)$ and we define $\ft_{a} g$ as follows. We only change $g$ at the elements $\iota_G(f_1),\iota_G(f_2)\in F(\ft_{a}G)$ and set
                \begin{align*}
                    &\ft_{a}g ( \iota_G(f_1) ) = \iota_H(g(f_1)), &\ft_{a}g(\iota_G(f_2)) = \iota_H(f_1).
                \end{align*}
            \item If both $f_1,f_2\not\in g^{-1}(W_a)$, then $\iota_G(f_1),\iota_G(f_2)\not\in g^{-1}(W_a)$. We set 
                \begin{align*}
                    &\ft_{a}g(\iota_G(f_i)) = \iota_H(g(f_i)),&\textnormal{for }i=1,2,
                \end{align*}
        and let $\ft_{a}g$ agree with $g$ everywhere else. 
    \end{itemize}
        \item Suppose that $\val V_a = 3$ with $V_a$ incident to an additional leg. As before, $r_G^{-1}(V_a) = \ell_a\cup\{l,f,V_a\}$, with $\{l\}\in L(G)$. Then $g$ directly induces $\ft_{a}g\colon \ft_aG\to \ft_aH$.
    \end{enumerate}
    We now introduce a poic-morphism $\eta_{\ft_a,G}\colon\sigma_G\to \sigma_{\ft_aG}$. There are the following possibilities for $E(\ft_{a}G)$:
    \begin{enumerate}
        \item If $\val(V_a)>3$, then $E(G)=E(\ft_{a}G)$. Here we set $\eta_{\ft_{a},G}= \Id_{\bbR_{\geq0}^{E(G)}}$.
        \item If $\val(V_a) =3$ with $V_a$ incident to two edges, then $r_G^{-1}(V_a) = \ell_a\cup\{f_1,f_2,V_a\}$, and $e=\{\iota_G(f_1),\iota_G(f_2)\}$ defines an edge of $E(\ft_{a}G)$. For $i=1,2$, let $e_i\in E(G)$ denote the edge given by $\{f_i,\iota_G(f_i)\}$. In this instance
            \begin{equation*}
                E(\ft_{a}G) = E(G)\backslash\{e_1,e_2\} \cup \{e\},
            \end{equation*}
        and we let $\eta_{\ft_{a},G}\colon\sigma_G\to\sigma_{\ft_{a}G}$ be the map given by 
            \begin{equation*}
                \eta_{\ft_{a},G}(\delta):=\begin{cases}
                \delta(h) = \delta(h), & \textnormal{ if }h\in E(\ft_{a}G)\backslash\{e\},\\
                \delta(h)=\delta(e_1)+\delta(e_2), &\textnormal{ if } h=e.\\
                \end{cases}
            \end{equation*}
        \item If $\val V_a = 3$ with $V_a$ incident to an additional leg, then $r_G^{-1}(V_a) = \ell_a\cup \{l,f,V_a\}$ with $\{l\}\in L(G)$. The set $\{f,\iota_G(f)\}$ defines an edge $e^\prime\in E(G)$ and $E(\ft_aG) = E(G)\backslash e^\prime$. In this case, we let $\eta_{\ft_a,G}\colon\sigma_G\to\sigma_{\ft_aG}$ denote the face-embedding given by $E(\ft_aG)\subset E(G)$. 
    \end{enumerate}
\end{const}
\begin{lem}\label{lem: forgetting the a-mark}
    The previous construction gives rise to a functor $\ft_a\colon \bbG_{g,A}^\op\to\bbG_{g,A\backslash a}^\op$ and a natural transformation $\eta_{\ft_a}$, that define a morphism of poic-spaces
    \begin{equation}
        {\ft_{a}}\colon \Mtrop_{g,A}\to\Mtrop_{g,A\backslash a}. \label{eq: forgetting the a-marking functor}
    \end{equation}
    Furthermore, if $b\in A\backslash a$ is an additional element, then ${\ft_{a}}\circ {\ft_{b}}={\ft_{b}}\circ {\ft_{a}}$. If $g=0$, then $\ft_a$ is a proper morphism of poic-complexes. If $g>0$, then $\ft_a$ satisfies the lifting property of Definition \ref{defi: proper morphism poic-fibrations}.
\end{lem}
\begin{proof}
    The functoriality of $\ft_{a}$, and the commutativity ${\ft_{a}}\circ {\ft_{b}}={\ft_{b}}\circ {\ft_{a}}$, follow from the construction, and the reader is entrusted with the corresponding details. We show that $\eta_{\ft_a}$ is a natural transformation. Let $g\colon G\to H$ be a morphism of $\bbG_{g,A}$ (that is $f\in\Hom_{\bbG_{g,A}^\op}(H,G)$), we want to show that the following diagram is commutative:
    \begin{equation}
        \begin{tikzcd}
            \sigma_{H}\arrow{d}[swap]{\eta_{\ft_{a},H}}\arrow{rr}{\sigma_{g,A}(g)}&&\sigma_G\arrow{d}{\eta_{\ft_{a},G}}\\
            \sigma_{\ft_a H}\arrow{rr}[swap]{\sigma_{g,A\backslash a}(\ft_{a} g)}&&\sigma_{\ft_{a} G}
        \end{tikzcd}\label{eq: forgetting a-mark nat transformation}
    \end{equation}
    Let $V_a\in V(G)$ be the incident vertex to $\ell_a$. The only interesting case is when $\val V_a=3$ and $V_a$ is incident to two edges. As before, we have that $r_G^{-1}(V_a) = \ell_a\cup\{f_1,f_2,V_a\}$ and the set $e=\{\iota_G(f_1),\iota_G(f_2)\}\subset F(\ft_aG)$ is an edge. Any morphism of $\bbG_{g,A}$ factors as a composition of multiple single edge contractions and automorphisms, so it suffices to study these particular cases.
    \begin{itemize}
        \item Suppose $g$ is an automorphism of $G$. Since it preserves the $A$-marking, it follows that $g\left(\{f_1,f_2\}\right) = \{f_1,f_2\}$. Commutativity with the involution $\iota_G$ implies that $\ft_a(g)$ must fix the new edge of $\ft_aG$, and therefore \eqref{eq: forgetting a-mark nat transformation} is commutative.
        \item Suppose $g$ is a single edge contraction $G\to G/h$, where $h\in E(G)$. If $V_a\not\in\partial h$, then $e$ is also an edge of $\ft_a G/h$. In this case, both maps $\sigma_{g,A}(g)$ and $\sigma_{g,A\backslash a}(\ft_ag)$ preserve this $e$-coordinate, and thus the commutativity of \eqref{eq: forgetting a-mark nat transformation}. If $V_a\in \partial h$, we can assume without loss of generality that $h=\{f_1,\iota_Gf_1\}$. In this case, $e^\prime =\{f_2,\iota_G f_2\}$ is an edge of both $\ft_aG/h$ and $G$. Additionally, $\ft_{a}g\colon \ft_a G\to \ft_aG/h$ is an isomorphism of $\bbG_{g,A\backslash a}$, where the edge $e$ is mapped to the edge $e^\prime$. The poic-morphism $\sigma_{g,A}(g)$ is the face-embedding where the $h$-coordinate is trivial, the map $\eta_{\ft_{a},H}$ is the identity, and the map $\sigma_G\xrightarrow{\eta_{\ft_{a},G}}\sigma_{\ft_{a},G}$ is given as follows at a $\delta\in \sigma_{G}$: 
        \begin{equation*}
            \left(\eta_{\ft_{a},G}(\delta)\right)_t:=\begin{cases}
                \delta(t) , & t\neq e^\prime,\\
                \delta(e^\prime)+\delta(h)& t=e^\prime,
            \end{cases}
        \end{equation*}
        where $t\in E(\ft_aG/h)$ is an edge. From the latter, the commutativity of \eqref{eq: forgetting a-mark nat transformation} follows.
    \end{itemize}
    Now, if $g=0$, then these are linear maps between closed convex polyhedral cones. Therefore, the morphism $\ft_a\colon \Mtrop_{0,A}\to\Mtrop_{0,A\backslash a}$ is proper in the sense of Definition \ref{defi: proper poic-complexes}. For $g>0$, the lifting property of $\ft_{a}$ readily follows, because an isomorphism between the target will define an isomorphism from the source (flags of the target are a subset of those of the source).
\end{proof}

\begin{const}\label{const: forgetting marking spanning tree}
   Let $T$ be an object of $\bbG_{0,\mathbbm{g}\sqcup A}$ and consider the morphism of poics 
   \begin{equation*}
   \eta_{\mathfrak{ft}_a^\pcomplexes,T}\colon \ST_{g,A}(T)\to\ST_{g,A\backslash a}(\ft_aT) \label{eq: morphism of poics}    
   \end{equation*}
   given by the linear map
    \begin{equation*}
        \begin{pmatrix}
        \eta_{\ft_a,T}&0\\
        M(T,a)&\Id_{g\times g}
        \end{pmatrix}\colon \sigma_T\times\bbR_{>0}^g\to \sigma_{\ft_aT}\times\bbR_{>0}^g,
    \end{equation*}
    where $M(T,a)$ is the $g\times\#E(T)$-matrix such that for $1\leq i\leq g$ and $e\in E(T)$ its $(i,e)$-entry is
    \begin{equation*}
        M(T,a)_{i,e} = \begin{cases}1,& \textnormal{ if }\val r_T\left(\ell_a(T)\right)=3 \textnormal{ and }\partial\ell_a(T)=\partial\ell_{i}(\ft_aT)\subset \partial e,\\
        1,& \textnormal{ if }\val r_T\left(\ell_a(T)\right)=3 \textnormal{ and }\partial\ell_a(T)=\partial\ell_{i^*}(\ft_aT)\subset \partial e,\\
        0,& \textnormal{ else}.
        \end{cases}
    \end{equation*}
\end{const}

\begin{const}\label{const: forgetting marking spanning tree2}
    We follow the notation of subsection \ref{ssec: mod spaces of rtc} and the notation established throughout this section. Let $X$ denote the finite set $A\sqcup \mathbbm{g}$. Let $p^\prime_a\colon \bbR^X\to\bbR^{X\backslash a}$ and $p_a\colon \bbR^{\binom{X}{2}}\to\bbR^{\binom{X\backslash a}{2}}$, denote the natural projection maps respectively given by forgetting the $a$-coordinate and forgetting any coordinate containing $a$. These fit into the commutative diagram of integral linear maps
    \begin{equation*}
        \begin{tikzcd}
            \bbR^X\arrow{d}[swap]{p^\prime_a}\arrow{rr}{M_X}&&\bbR^{\binom{X}{2}}\arrow{d}{p_a}\\
            \bbR^{X\backslash a}\arrow{rr}[swap]{M_{X\backslash a}}&&\bbR^{\binom{X\backslash a}{2}}.
        \end{tikzcd}
    \end{equation*}
    Hence, $p_a$ gives rise to a linear integral map $p_a\colon Q_X\to Q_{X\backslash a}$, and in particular to a morphism of poic-complexes $p_a\colon \underline{N_{\dist_X}}_\bbR\to \underline{N_{\dist_{X\backslash a}}}_\bbR$. Furthermore, $p_a\circ \dist_X = \dist_{X\backslash a}\circ \ft_a$, so that by letting $(\ft_{a})_\tint := p_a$, we obtain a morphism of linear poic-complexes $\ft_a\colon\Mtrop_{0,X}\to\Mtrop_{0,X\backslash a}$.
\end{const}

\begin{prop}
    Following the notations of Constructions \ref{const: forgetting marking spanning tree} and \ref{const: forgetting marking spanning tree2}, let $\mathfrak{ft}^\pcomplexes$ denote the morphism of linear poic-complexes given by the morphisms \eqref{eq: morphism of poics} and $(\mathfrak{ft}^\pcomplexes_a)_{\tint}:=(\ft_a)_{\tint}$ (following Construction \ref{const: forgetting marking spanning tree2}), and let $\mathfrak{ft}^\pspaces$ denote the morphism of poic-spaces $\ft_a\colon \Mtrop_{g,A}\to \Mtrop_{g,A\backslash a}$. These define a weakly proper morphism of linear poic-fibrations 
    \begin{equation}
        \mathfrak{ft}_{a}\colon \st_{g,A}\to\st_{g,A\backslash a}.\label{eq: forgetting the a-mark fibrations.}
    \end{equation}
\end{prop}
\begin{proof}
    The equality $\mathfrak{ft}^\pspaces\circ\st_{g,A} = \st_{g,A\backslash a}\circ \mathfrak{ft}^\pcomplexes$ is a routine check left to the reader. This shows that $\mathfrak{ft}_a$ actually defines a morphism of poic-fibrations, and to show its weakly properness, we must show that it satisfies two properties:
    \begin{enumerate}
        \item The first property follows directly from the fact that an edge contraction after forgetting the marking actually comes from an edge contraction of the original graph (perhaps best understood through illustration), so that facets after forgetting a marked leg arise from facets before forgetting the marked leg.
        \item The second property has already been shown in Lemma \ref{lem: forgetting the a-mark},
    \end{enumerate}
\end{proof}

\begin{example}\label{ex: weak properness of forgetting the marking}
Suppose $A=\{a,b,c\}$, $B=\{b,c\}$, and let $g=1$. We consider the morphism $\ft_a\colon \st_{1,A}\to \st_{1,B}$. Consider the graphs $G$ of $\bbG_{g,A}$ and $\ft_aG$ of $\bbG_{g,B}$ depicted in Figure \ref{fig: wp forgetting the marking}, and consider the trees $T$ and $\ft_aT$, lying in $\st_{g,A}^{-1}(G)$ and $\st_{g,B}^{-1}(\ft_aG)$ correspondingly, depicted in Figure \ref{fig: wp forgetting the marking trees}. Then the morphism of poics $\ST_{1,A}(T)\to \ST_{1,B}(\ft_aT)$ is given explicitly as the morphism of poics:
\begin{equation*}
    F\colon\bbR^2_{\geq 0}\times \bbR_{>0}\to \bbR_{\geq 0}\times\bbR_{>0}, (x,y,\ell)\mapsto (x,y+\ell).
\end{equation*}
Consider the subcone $\sigma$ of $\bbR^2_{\geq0}\times\bbR_{>0}$ given by
\begin{equation*}
    \sigma:=\{(x,y,\ell) \in \bbR^2_{\geq0}\times\bbR_{>0} : x=\ell\}.
\end{equation*}
Then $F(\sigma)$ is the subcone of $\bbR_{\geq0}\times\bbR_{>0}$ given by
\begin{equation*}
    F(\sigma) = \{ (s,t)\in \bbR_{\geq0}\times\bbR_{>0} : t-s\geq 0, s>0\}.
\end{equation*}
However, the closure of $F(\sigma)$ relative to $\bbR_{\geq0}\times\bbR_{>0}$ contains the facet given by $s= 0$, which does not arise as a face of $\sigma$.
\begin{figure}
    \centering
    \begin{tikzpicture}
    \draw(-6,-3)--(6,-3)--(6,3.8)--(-6,3.8)--(-6,-3);
    \draw(-6,3)--(0,3) node [midway, above] {Graph $G$ of $\bbG_{1,A}$};
    \draw (0,3)--(6,3) node [midway, above] {Graph $\ft_aG$ of $\bbG_{1,B}$};
    \draw(0,3.8)--(0,-3);
    \draw[dashed] (-5,0)--(-4,0) node [pos = -0.2] {$b$};
    \draw[line width=2pt](-3,-1)--(-4,0) node [midway, below left] {$e_2$};
    \draw[line width = 2pt](-4,0)--(-3,1) node [midway, above left]{$e_1$};
    \draw[line width = 2pt] (-3,1) arc [start angle=90,end angle= -90, x radius = 1, y radius= 1]node[ midway, right] {$e_3$};
    \draw[dashed] (-4,-2)--(-3,-1) node [pos=-0.2] {$a$};
    \draw[dashed] (-4,2)--(-3,1) node [pos = -0.2] {$c$};
    \draw[fill=black] (-3,1) circle (3pt);
    \draw[fill=black] (-4,0) circle (3pt);
    \draw[fill=black] (-3,-1) circle (3pt);

    \draw[dashed] (1,0)--(2,0) node [pos = -0.2] {$b$};
    \draw[line width=2pt](2,0)--(3,1) node [midway, above left] {$e_1^\prime$};
    \draw[line width = 2pt] (3,1) arc [start angle=90,end angle= -180, x radius = 1, y radius= 1] node [midway, right] {$e^\prime_2$};
    \draw[dashed] (2,2)--(3,1) node [pos = -0.2] {$c$};
    \draw[fill=black] (3,1) circle (3pt);
    \draw[fill=black] (2,0) circle (3pt);
    \end{tikzpicture}
    \caption{Graphs of Example \ref{ex: weak properness of forgetting the marking}}
    \label{fig: wp forgetting the marking}
\end{figure}

\begin{figure}
    \centering
    \begin{tikzpicture}
    \draw(-6,-3)--(6,-3)--(6,3.8)--(-6,3.8)--(-6,-3);
    \draw(-6,3)--(0,3) node [midway, above] {Tree $T$ in $\st_{1,A}^{-1}(G)$};
    \draw (0,3)--(6,3) node [midway, above] {Tree $\ft_aT$ in $\st_{1,B}^{-1}(\ft_aG)$};
    \draw(0,3.8)--(0,-3);
    \draw[dashed] (-5,0)--(-4,0) node [pos = -0.2] {$b$};
    \draw[line width=2pt](-3,-1)--(-4,0) node [midway, below left] {$y$};
    \draw[line width = 2pt](-4,0)--(-3,1) node [midway, above left]{$x$};
    \draw[dashed] (-3,1)--(-2,1) node[pos = 1.2] {$1$};
    \draw[dashed] (-3,-1)--(-2,-1) node[pos = 1.2] {$1^*$};
    \draw[dashed] (-4,-2)--(-3,-1) node [pos=-0.2] {$a$};
    \draw[dashed] (-4,2)--(-3,1) node [pos = -0.2] {$c$};
    \draw[fill=black] (-3,1) circle (3pt);
    \draw[fill=black] (-4,0) circle (3pt);
    \draw[fill=black] (-3,-1) circle (3pt);
    
    \draw[dashed] (1,0)--(2,0) node [pos = -0.2] {$b$};
    \draw[line width=2pt](2,0)--(3,1) node [midway, above left] {$s$};
    \draw[dashed] (3,1) --(4,1) node [pos =1.2] {$1$};
    \draw[dashed] (2,0) --(3,-1) node [pos =1.2] {$1^*$};
    \draw[dashed] (2,2)--(3,1) node [pos = -0.2] {$c$};
    \draw[fill=black] (3,1) circle (3pt);
    \draw[fill=black] (2,0) circle (3pt);
    \end{tikzpicture}
    \caption{Trees of Example \ref{ex: weak properness of forgetting the marking}}
    \label{fig: wp forgetting the marking trees}
\end{figure}
\end{example}
\subsection{Clutching morphisms.} Let $g,h\geq 0$ be integers, and let $A$ and $B$ be finite sets with $2g+\#A-2>0$ and $2h+\#B-2>0$. Additionally, it is also assumed that $A\cap B= \{c\}$, so that $A\Delta B = (A\cup B)\backslash \{c\}$. We now construct a proper morphism of linear poic-fibrations $\st_{g,A}\times\st_{h,B}\to \st_{g+h,A\Delta B}$, given by joining two graphs at the $c$-marked leg. We proceed meticulously as in the previous subsection.
\begin{const}
    Suppose $G$ is an object of $\bbG_{g,A}$ and $H$ is an object of $\bbG_{h,B}$. We will describe an object $\kappa(G,H)$ of $\bbG_{g+h,A\Delta B}$ motivated by our description. Let $V_c\in V(G)$ denote the vertex incident to $\ell_c(G)$, and let $F(\kappa(G,H))$ denote the set obtained from $F(G)\sqcup F(H)$ by identifying the sets $\ell_c(G)\cup\partial\ell_c(G)\subset F(G)$ and $\ell_c(H)\cup\partial\ell_c(H)\subset F(H)$ into the single element $V_c$, that is:
    \begin{equation*}
        F(\kappa(G,H)) = \left(F(G)\backslash \ell_c(G)\right)\sqcup\left(F(H)\backslash(\ell_c(H)\cup \partial\ell_c(H))\right).
    \end{equation*}
    Observe that the following maps are well-defined
     \begin{align*}
        &r_{\kappa(G,H)}\colon  F(\kappa(G,H))\to F(\kappa(G,H)),& X\mapsto &\begin{cases}
            r_G(X), &\textnormal{ if }X\in F(G), \\
            r_H(X), &\textnormal{if }X\in F(H)\backslash r_H^{-1}(\partial\ell_c(H)),\\
            V_c, &\textnormal{if }X\in r_H^{-1}(\partial\ell_c(H)),
        \end{cases}\\
        &\iota_{\kappa(G,H)}\colon F(\kappa(G,H))\to F(\kappa(G,H)), &X \mapsto &\begin{cases}
            \iota_G(X), &\textnormal{ if }X\in F(G),\\
            \iota_H(X), &\textnormal{ if }X\in F(H),
        \end{cases}
    \end{align*}
    and, furthermore, satisfy $\iota_{\kappa(G,H)}\circ r_{\kappa(G,H)}=r_{\kappa(G,H)}$.\\
    This means then that the triple $(F(\kappa(G,H)),r_{\kappa(G,H)},\iota_{\kappa(G,H)})$ actually defines a graph $\kappa(G,H)$.The edges of this graph are just the union of those of $G$ and $H$ (that is $E(\kappa(G,H)) = E(G)\cup E(H)$). Furthermore, if we identify $\bbR^{E(\kappa(G,H))}$ with the product $\bbR^{E(G)}\times\bbR^{E(H)}$, then the cone of metrics $\sigma_{\kappa(G,H)}$ naturally coincides with the product $\sigma_G\times \sigma_H$, so we let $     \eta_{\kappa(G,H)}$ denote the identity map $\sigma_G\times\sigma_H\to \sigma_{\kappa(G,H)}$. In addition, we remark that the markings of $G$ and $H$ clearly give rise to a marking 
    \begin{equation*}
        \ell\colon A\Delta B\to L(\kappa(G,H)).
    \end{equation*}
    If we now turn our attention towards morphisms, let $f_1\colon G\to G^\prime$ and $f_2\colon H\to H^\prime$ be morphisms of their respective categories. Then the map
    \begin{equation*}
        F(\kappa(G,H))\to F(\kappa(G^\prime,H^\prime)), X \mapsto \begin{cases}
            f_1(X), &\textnormal{ if }X\in F(G), \\
            f_2(X), &\textnormal{ if }X\in F(H),
            \end{cases}
    \end{equation*}
    is well-defined and gives rise to a morphism $\kappa(f_1,f_2)\colon \kappa(G,H)\to \kappa(G^\prime,H^\prime)$.
\end{const}
\begin{defi}
    Following the notation of the above construction, we call the graph $\kappa(G,H)$ \emph{the clutching of $G$ and $H$ at $c$}.
\end{defi}
\begin{lem}
    The associations of the above construction (and following the notation thereof) give rise to a functor
    \begin{equation*}
        {\kappa}\colon\bbG_{g,A}^\op\times \bbG_{h,B}^\op\to \bbG_{g+h,A\Delta B}^\op, (G,H)\mapsto \kappa(G,H).
    \end{equation*}
    In addition, the maps $\eta_{\kappa,(\bullet,\bullet)}$ define a natural transformation 
    \begin{equation*}
        \eta_{\kappa}\colon \Mtrop_{g,A}\times\Mtrop_{h,B} \implies \sigma_{g+h,A\Delta B}\circ \kappa.
    \end{equation*}
\end{lem}   
 \begin{proof}
     Both statements of the lemma are routine checks, that we omit for the sake of brevity.
 \end{proof}
 
\begin{defi}
 The previous lemma shows that we obtain a morphism of poic-spaces
 \begin{equation}
     \kappa\colon \Mtrop_{g,A}\times\Mtrop_{h,B}\to  \Mtrop_{g+h,A\Delta B}.
 \end{equation}
We call this morphism of poic-spaces the \emph{clutching morphism}.
\end{defi}
We seek to improve this to a morphism of linear poic-fibrations. 
\begin{const}
Let $A$ and $B$ be as before. Consider the following integral linear map $K_{A,B}\colon \bbR^{\binom{A}{2}}\oplus\bbR^{\binom{B}{2}}\to \bbR^{\binom{A\Delta B}{2}}$ given at $(\vec{v},\vec{w})\in \bbR^{\binom{A}{2}}\oplus \bbR^{\binom{B}{2}}$ by
    \begin{equation*}
        K_{A,B}(\vec{v},\vec{w})_{\{a,b\}}:= \begin{cases}
            {v}_{\{a,b\}},&\{a,b\} \subset A,\\
            {w}_{\{a,b\}},&\{a,b\} \subset B,\\
            {v}_{\{a,c\}}+{w}_{\{c,b\}},&a\in  A, b\in B.
        \end{cases}
    \end{equation*}
    It is readily checked that this map descends to an integral linear map $K_{A,B}\colon Q_A\oplus Q_B\to Q_{A\Delta B}$, and in particular to a morphism of linear poic-complexes $K_{A,B}\colon\underline{\left(N_{\dist_A}\oplus N_{\dist_B}\right)_\bbR}\to \underline{\left(N_{\dist_{A\Delta B}}\right)_\bbR}$. Furthermore, $K_{A,B}\circ (\dist_A\times\dist_B)= \dist_{A\Delta B}\circ \kappa $, so that letting $\kappa_\tint:= K_{A,B}$, we obtain a morphism of linear poic-complexes
    \begin{equation}
        \kappa\colon \Mtrop_{0,A}\times \Mtrop_{0,B} \to \Mtrop_{0,A\Delta B}.\label{eq: clutching rational}
    \end{equation}
\end{const}

\begin{nota}
    Let $j = g+h$ and $C=A\Delta B$, we set 
    \begin{align*}
     &&\mathbbm{g}=\{1,\dots,g,1^*,\dots,g^*\},&&\mathbbm{h}=\{1,\dots,h,1^*,\dots,h^*\},&&\mathbbm{j}=\{1,\dots,j,1^*,\dots,j^*\},
    \end{align*}
    and regard $\mathbbm{g}$ and $\mathbbm{h}$ as disjoint. We identify $\mathbbm{j}$ as the disjoint union $\mathbbm{g}\sqcup \mathbbm{h}$, by viewing $\mathbbm{g}\subset \mathbbm{j}$ and identifying $\mathbbm{h}$ with the subset (by $n\mapsto g+n$ and $n^*\mapsto (g+n)^*$, for $1\leq n\leq h$)
    \begin{equation*}
        \{g+1,\dots,g+h,(g+1)^*,\dots,(g+h)^*\}\subset \mathbbm{j}.
    \end{equation*}
    With the above, we regard the morphism \eqref{eq: clutching rational} as 
    \begin{equation*}
        \kappa\colon \Mtrop_{0,A\sqcup \mathbbm{g}}\times \Mtrop_{0,B\sqcup\mathbbm{h}} \to \Mtrop_{0,C \sqcup \mathbbm{j}}.
    \end{equation*}
    Additionally, we denote the natural map $\bbR^g\times\bbR^h\to\bbR^{g+h}$ by $\Id_g\oplus\Id_h$.
\end{nota} 

\begin{lem}\label{lem: clutching morphism of fibrations}
    Let $\mathfrak{k}^\pcomplexes$ denote the morphism of linear poic-complexes 
    \begin{equation*}
        \kappa\times (\Id_g\oplus\Id_h)\colon \ST_{g,A}\times\ST_{h,B}\to \ST_{j,C},
    \end{equation*}
    and let $\mathfrak{k}^\pspaces$ denote the morphisms of poic-spaces $\kappa\colon \Mtrop_{g,A}\times\Mtrop_{h,B}\to\Mtrop_{j,C}$. These define a proper morphism of poic-fibrations 
    \begin{equation}
    \mathfrak{k}\colon\st_{g,A}\times\st_{h,B}\to \st_{j,C}.\label{eq: clutching morphism fibrations}
    \end{equation}
\end{lem}
\begin{proof}
    The properness of $\mathfrak{k}^\pcomplexes$ just depends on the properness of $\kappa\colon \Mtrop_{0,A\sqcup\mathbbm{g}}\times\Mtrop_{0,B\sqcup\mathbbm{h}}\to\Mtrop_{0,C\sqcup\mathbbm{j}}$. But this is clear, since the $\eta_{\mathfrak{k}^c,\bullet}$ are isomorphisms of poics. The lifting property of $\mathfrak{k}^\pspaces$ on isomorphisms follows by keeping track of the images of the subgraphs of $\kappa(G,H)$ generated by $G$ and $H$ (and reconstructing the respective isomorphisms in this way). To close, we observe that the relation $\st_{j,C}\circ \mathfrak{k}^\pcomplexes=\mathfrak{k}^\pspaces\circ \left(\st_{g,A}\times\st_{h,B}\right)$ is readily checked.
\end{proof}
\begin{defi}
 We also call the proper morphism of linear poic-fibrations \eqref{eq: clutching morphism fibrations} \emph{the clutching morphism}.
\end{defi} 

\begin{example}

    An example of the clutching morphism is explicitly depicted in Figure \ref{fig: clutching}, where we let $A=\{a_1,a_2,a_3,c\}$ and $B=\{b_1,b_2,c\}$. In this case, we have that $A\Delta B=\{a_1,a_2,a_3,b_1,b_2\}$, and from objects $G$ of $\bbG_{2,A}$ and $H$ of $\bbG_{3,B}$, we obtain an object $\kappa(G,H)$ of $\bbG_{5,A\Delta B}$. For size considerations we omit the labelling of vertices and edges in these depictions.

    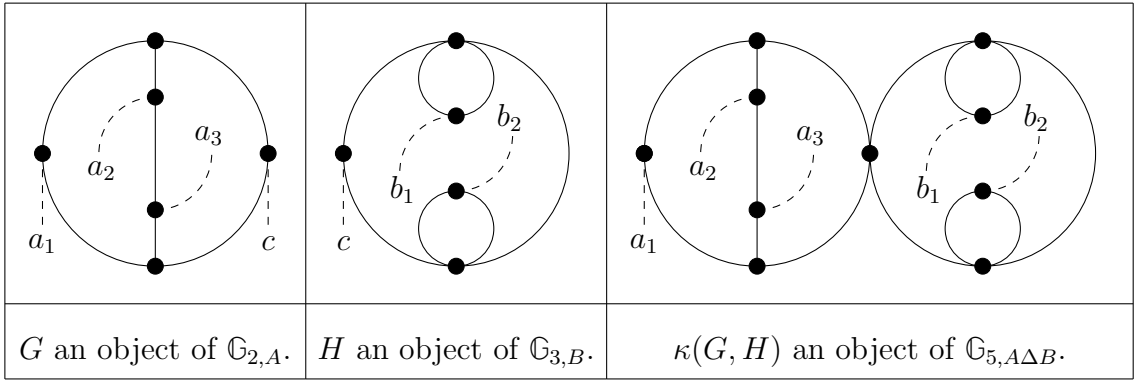
\begin{figure}[htbp]
        \begin{tikzpicture}
            \draw[dashed](-6.5,0)--(-6.5,-1) node [pos=1.2] {$a_1$};
            \draw[dashed] (-5,0.75) arc (90:180:0.75) node [pos = 1.2] {$a_2$};
            \draw[dashed] (-5,-0.75) arc (-90:0:0.75) node [pos = 1.2] {$a_3$};
            \draw[dashed](-3.5,0)--(-3.5,-1) node [pos=1.2] {$c$}; 
            \draw[fill=none](-5,0) circle (1.5);
            \draw[fill=black](-5,1.5) circle (3pt);
            \draw[fill=black](-5,-1.5) circle (3pt);
            \draw[fill=black](-6.5,0) circle (3pt);
            \draw[fill=black](-5,0.75) circle (3pt);
            \draw[fill=black](-5,-0.75) circle (3pt);
            \draw[fill=black](-3.5,0) circle (3pt);
            \draw[](-5,1.5)--(-5,-1.5);

            \draw[fill=none](-1,0) circle (1.5);
            \draw[fill=none](-1,1) circle (0.5);
            \draw[fill=none](-1,-1) circle (0.5);
            \draw[dashed] (-1,0.5) arc (90:180:0.75) node [pos = 1.2] {$b_1$};
            \draw[dashed] (-1,-0.5) arc (-90:0:0.75) node [pos = 1.2] {$b_2$};
            \draw[dashed] (-2.5,0)--(-2.5,-1) node [pos=1.2] {$c$};
            \draw[fill=black](-2.5,0) circle (3pt);
            \draw[fill=black](-1,1.5) circle (3pt);
            \draw[fill=black](-1,0.5) circle (3pt);
            \draw[fill=black](-1,-0.5) circle (3pt);
            \draw[fill=black](-1,-1.5) circle (3pt);

            \draw[dashed](1.5,0)--(1.5,-1) node [pos=1.2] {$a_1$};
            \draw[dashed] (3,0.75) arc (90:180:0.75) node [pos = 1.2] {$a_2$};
            \draw[dashed] (3,-0.75) arc (-90:0:0.75) node [pos = 1.2] {$a_3$};
            \draw[dashed] (6,0.5) arc (90:180:0.75) node [pos = 1.2] {$b_1$};
            \draw[dashed] (6,-0.5) arc (-90:0:0.75) node [pos = 1.2] {$b_2$};
            \draw[fill=none](3,0) circle (1.5);
            \draw[](3,1.5)--(3,-1.5);
            \draw[fill=none](6,0) circle (1.5);
            \draw[fill=none](6,1) circle (0.5);
            \draw[fill=none](6,-1) circle (0.5);
            \draw[fill=black](3,1.5) circle (3pt);
            \draw[fill=black](3,-1.5) circle (3pt);
            \draw[fill=black](1.5,0) circle (3pt);
            \draw[fill=black](3,0.75) circle (3pt);
            \draw[fill=black](3,-0.75) circle (3pt);
            \draw[fill=black](1.5,0) circle (3pt);
            \draw[fill=black](4.5,0) circle (3pt);
            \draw[fill=black](6,1.5) circle (3pt);
            \draw[fill=black](6,0.5) circle (3pt);
            \draw[fill=black](6,-0.5) circle (3pt);
            \draw[fill=black](6,-1.5) circle (3pt);

            \draw (-7,2)--(8,2);
            \draw (-7,2)--(-7,-3);
            \draw (-3,2)--(-3,-3);
            \draw (1,2)--(1,-3);
            \draw (8,2)--(8,-3);
            \draw (-7,-2)--(8,-2);
            \draw (-7,-3)--(-3,-3) node [midway, above] {$G$ an object of $\bbG_{2,A}$};
            \draw (-3,-3)--(1,-3) node [midway, above] {$H$ an object of $\bbG_{3,B}$};
            \draw (1,-3)--(8,-3) node [midway, above] {$\kappa(G,H)$ an object of $\bbG_{5,A\Delta B}$};

        \end{tikzpicture}
        \caption{The morphism $\kappa$ in action}
        \label{fig: clutching}
    \end{figure}
\end{example}
\subsection{Mumford curves and cycle rigidification.} In \cite{CavalieriGrossMarkwig}, the authors approach similar constructions regarding the moduli spaces of tropical curves. It is worth mentioning that they consider discrete graphs with non-trivial weight functions on their vertices, something that we do not consider. They define the notion of a tropical space, and the category of tropical spaces, which they provide with the structure of a site. They additionally define the notion of a \emph{family of $n$-marked genus-$g$ stable
tropical curves} over a fixed tropical space, and introduce the category $\overline{\mathscr{M}}_{g,n}$ fibered in groupoids over the category of tropical spaces given by
\begin{equation*}
    \overline{\mathscr{M}}_{g,n}= \{ \textnormal{ families of }n\textnormal{-marked genus-}g \textnormal{ stable tropical curves over } B\},
\end{equation*}
where $B$ is a tropical space. This fibered category is a stack with respect to the specified site structure, that is not geometric and it is not possible to provide it with an atlas. The issue being topological, since there cannot exist a local isomorphism onto $\overline{\mathscr{M}}_{g,n}$. Therefore, the intersection theory they define for tropical spaces cannot be applied to $\overline{\mathscr{M}}_{g,n}$, but nonetheless they are able to define divisors and line bundles on this stack, and they settle by considering tropical cycles on tropical spaces $T$ that are equipped with a morphism $T\to \overline{\mathscr{M}}_{g,n}$. However, this stack contains the open substack $\mathscr{M}^{\Mf}_{g,n}$ of \emph{Mumford curves} that does admit an atlas. 

A \emph{Mumford curve} in \cite{CavalieriGrossMarkwig} is a tropical curve, where the genus function of the underlying combinatorial type (the discrete graph) is identically zero. Namely, what we call a tropical curve in this paper. In terms of spaces, our $\cM^{\trop}_{g,n}$ corresponds identically to their $\mathscr{M}^{\Mf}_{g,n}$. The atlas they construct for $\mathscr{M}^{\Mf}_{g,n}$ is by means of cycle rigidifications. That is, they construct a space $\mathscr{V}^{\Mf}_{g,n}$ whose points parametrize cycle rigidified tropical curves. A cycle rigidified graph is simply a graph together with an oriented cycle basis consisting of primitive cycles (those that are not subsets of another cycle). We briefly explain how to interpret the space $\mathscr{V}^{\Mf}_{g,n}$ of cycle rigidified Mumford curves in terms of our framework.
\begin{const}
    Suppose $g,n\geq 0$ are such that $n+2g-2>0$. Let $\bbV_{g,n}$ denote the category specified by:
    \begin{itemize}
        \item The objects of $\bbV_{g,n}$ consist of tuples $(G,c_1^G,\dots,c_g^G)$, where $G$ is an object of $\bbG_{g,n}$ and $(c_i^G)_{1\leq i\leq g}$ is an oriented cycle basis of $H_1(G,\bbZ)$ consisting of primitive cycles.
        \item For two objects $(G,c_i^G)$ and $(H,c_i^H)$, the set of morphisms $\Hom_{\bbV_{g,n}}((G,c_i^G),(H,c_i^H))$ consists of the contractions $f\colon G\to H$ such that $f_*c_i^G=c_i^H$.
        \item Composition of morphisms is just composition of maps.
    \end{itemize}
    We remark that $\bbV_{g,n}$ is essentially finite, and there are no non-trivial automorphisms at objects of $\bbV_{g,n}$ (see Lemma 4.2 of \cite{CavalieriGrossMarkwig}). Naturally, the association
    \begin{equation*}
        {\drgdf_{g,n}}\colon \bbV_{g,n}\to\bbG_{g,n}, (G,(c_i^G))\mapsto G,
    \end{equation*}
    (here $\drgdf$ stands for derigidification) is a functor, and by setting $\mathrm{V}_{g,n}:=\Mtrop_{g,n}\circ {\drgdf_{g,n}}$, we obtain a poic-space $\mathrm{V}_{g,n}$ with a morphism of poic-spaces $\drgdf\colon\mathrm{V}_{g,n}\to\Mtrop_{g,n}$. This is just a poic-space, because the category $\bbV_{g,n}$ is not thin. It is clear from construction that the realization of this space is naturally homeomorphic to $\mathscr{V}^{\Mf}_{g,n}$.
\end{const}
\begin{nota}
    Let $\bbR_{\neq0}$ denote the linear poic-complex given as $\bbR_{>0}\sqcup \bbR_{<0}$, where we regard these poics as subcones of $\underline{\bbR}$. The underlying category of this poic-complex consists of two objects, which we will denote by $\pm 1$.
\end{nota}
Letting $g,n\geq0$ as above, we now construct a poic-fibration over $\bbV_{g,n}$ in the same spirit as $\st_{g,n}$.
\begin{const}
    Consider the linear poic-complex $\Mtrop_{0,n+2g}\times (\bbR_{\neq0})^g$. So an object of the underlying category is a tuple $(T, (s_i)_{1\leq i\leq g})$ where $s_i = \pm1$. To such an object we associate the cycle rigidified graph $(\st_{g,n}(T), (c^{s_i})_{1\leq i\leq g})$, where $c^{s_i}$ is the oriented primitive cycle defined by the primitive cycle created by the added edge $e_i\in E(\st_{g,n}(T))$ with the following orientation:
    \begin{itemize}
        \item if $s_i=1$, then the orientation is from the vertex incident to $\partial\ell_{n+i}(T)$ towards the vertex incident to $\ell_{n+g+i}(T)$, 
        \item if $s_i=-1$, then the reverse orientation.
    \end{itemize}
    The morphisms are assigned just as with $\st_{g,n}$, and this association can be seen to give rise to a well-defined functor. Further, this functor can be improved to a morphism of poic-spaces
    \begin{equation}
        \rgdfst_{g,n}\colon \Mtrop_{0,n+2g}\times (\bbR_{\neq0})^g\to \mathrm{V}_{g,n},\label{eq: cycle rigidified poic-fibration}
    \end{equation}
    by letting the natural transformation be just like in \eqref{eq: poic morphisms from spanning tree cover} (here $\rgdfst$ stands for rigidification of spanning trees). As in Proposition \ref{prop: spanning tree fibration}, this is a poic-fibration.
\end{const}

We can now take the realization of \eqref{eq: cycle rigidified poic-fibration} to obtain a continuous map
\begin{equation}
        |\rgdfst_{g,n}|\colon \cM^\trop_{0,n+2g}\times (\bbR_{\neq0})^g\to \mathscr{V}^{\Mf}_{g,n},
\end{equation}
which is a morphism of tropical spaces (where $\cM^\trop_{0,n+2g}\times (\bbR_{\neq0})^g$ has the tropical space structure given by its natural embedding to $Q_{n+2g}\oplus\bbR^g$). Under this morphism, we can identify tropical cycles of the tropical space $\mathscr{V}^{\Mf}_{g,n}$ as tropical cycles of the tropical space $\cM^\trop_{0,n+2g}\times (\bbR_{\neq0})^g$ that are invariant with respect to the map $|\rgdfst_{g,n}|$. This coincides with our constructions, so to conclude that we can compute the tropical cycles on the tropical space $\mathscr{V}^{\Mf}_{g,n}$ in terms of our framework, it suffices to explain how to compute the tropical cycles on the tropical space $\cM^\trop_{0,n+2g}\times (\bbR_{\neq0})^g$ in terms of our framework. Since the tropical space  $\cM^\trop_{0,n+2g}\times (\bbR_{\neq0})^g$ is embedded in a real vector space, the tropical cycles of this tropical space coincide with the classical (as in \cite{AllermannRau}) tropical cycles. Therefore, following an analogous construction to Construction \ref{const: polyhedral cycles}, we can modify $\cM^\trop_{0,n+2g}\times (\bbR_{\neq0})^g$ and produce, in this way, a linear poic-complex whose cycles are the same as the tropical cycles of this tropical space.
It is possible to follow a similar pattern of the construction of $\rgdfst_{g,n}$ for $\Mtrop_{0,n+2g}\times \underline{\bbR}^g$ and $\overline{\mathscr{M}}_{g,n}$. Of course this does not solve the previously mentioned topological issue of providing the stack $\overline{\mathscr{M}}_{g,n}$ with an atlas, because the underlying map is by no means a local isomorphism. However, this construction can be understood in a similar combinatorial vein as what has been done for $\cM^{\trop}_{g,n}$. In this case, instead of $\bbG_{g,n}$ one must consider the category of genus-$g$ $n$-marked stable graphs, and perform the analogous constructions. Since we make use of the specific poic-fibration $\st_{g,n}$ in upcoming work, we defer the details of this case for future work.

\section{Appendix}
\subsection{Equivalent poic-complexes and subdivisions.} The notion of isomorphism of poic-complexes is quite restrictive, and it is often the case that we can change the underlying category by an equivalent one, in order to simplify some computations. Let $\Phi\colon C_\Phi\to\POIC$ denote a poic-complex. 
\begin{lem}\label{lem: equivalence poic-complexes.}
    If $D$ is a category with a functor $F:D\to C_\Phi$ that is an equivalence, then $\Phi\circ F$ is also a poic-complex. Moreover, the functor $F$ induces a morphism of poic-complexes $\Phi\circ F\to \Phi$ whose realization is a homeomorphism, which is piecewise linear at the cones.
\end{lem}
\begin{proof}
    An equivalence is given by an essentially surjective fully faithful functor. The functor being fully faithful implies that $D$ is thin, and together with essential surjectivity, it also implies that $D$ is essentially finite. Hence, the tuple $\Phi\circ F$ is a poic-complex, and the identity natural transformation $\Id:(\Phi\circ F)\implies \Phi\circ F$ gives rise to a morphism of poic-complexes $\Phi\circ F\to \Phi$, where at each cone of $\Phi\circ F$ it is an (integral linear) isomorphism. Since $F$ is essentially surjective and fully faithful, it follows from the previous fact that the induced map between the realizations is a homemorphism with the prescribed properties. 
\end{proof}
\begin{defi}
    We say that a morphism of poic-complexes $E\colon \Sigma\to \Phi$ is an \emph{equivalence}, if the underlying functor $E$ is an equivalence, and $\eta_E$ consists of isomorphisms. In this case, the morphism $E$ induces an isomorphism between the corresponding realizations. 
\end{defi}
\begin{lem}\label{lem: precomposition of subd with eq is subd.}
    If $E\colon\Sigma\to\Phi$ is an equivalence and $\Phi_X$ is a linear poic-complex, then $E$ is a proper morphism and 
    \begin{equation*}
        E_*\colon M_k(\Sigma_{X\circ E})\to M_k(\Phi_X)
    \end{equation*}
    is an isomorphism for every integer $k\geq 0$. Moreover, if $S\colon\Sigma^\prime \to \Sigma$ is a subdivision, then $S\circ E\colon \Sigma^\prime\to \Phi$ is also a subdivision.
\end{lem}
\begin{proof}
    An equivalence $E\colon\Sigma\to \Phi$ establishes a bijection between isomorphism classes, and therefore the corresponding weight groups are isomorphic. Since $\eta_E\colon \Sigma \implies \Phi\circ F_E$ consists of isomorphisms, it follows that the balancing conditions are the same, so the Minkowski weights are isomorphic. The statement on subdivisions is a routine check.
\end{proof}
\begin{lem}\label{lem: poic-complex equiv to poset poic-complex}
    Any poic-complex $\Phi$ is equivalent to a poic-complex whose underlying category is a poset.
\end{lem}
\begin{proof}
    Let $\Phi$ be an arbitrary poic-complex. By definition, the isomorphism classes $[\Phi]$ form a finite set. For each class $x\in [\Phi]$, let 
    \begin{equation*}
        \texttt{S} =\{ s_x : s_x \textnormal{ is an object of } C_\Phi \textnormal{ and }[s_x]=x\}
    \end{equation*}
    be a full set of representatives\footnote{This is instance of the axiom of choice is commonly known as the global axiom of choice, and we make use of it. If the reader prefers, it can instead be assumed that every essentially finite category that we consider comes with a specified choice of representatives. This can be granted in the cases of interest.} of $[\Phi]$. Let $C_{\texttt{S}(\Phi)}$ denote the full subcategory of $C_\Phi$ given by $\texttt{S}$. By construction, this category is actually a poset and the inclusion $I:C_{\texttt{S}(\Phi)}\to C_\Phi$ is a fully faithful essentially surjective functor. Therefore, Lemma \ref{lem: precomposition of subd with eq is subd.} shows that $\Phi\circ I$ is a poi-complex an equivalence of categories, and by definition the inclusion morphism $\Phi\circ I\to \Phi$ is an equivalence . 
\end{proof}

\begin{lem}\label{lem: equiva poic-complexes equivalent subdivisions}
    Suppose $\Sigma$ and $\Phi$ are poic-complexes, where the underlying category of $\Sigma$ is a given by a (finite) poset. A morphism $E\colon\Sigma\to\Phi$ that is an equivalence induces an equivalence of categories $\Subd(\Sigma)\to \Subd(\Phi)$. In addition, if $\Phi_X$ is a linear poic-complex, then the morphism induces an isomorphism $E_*\colon Z_k(\Sigma_{X\circ E})\to Z_k(\Phi_X)$ for any integer $k\geq0$. 
\end{lem}
\begin{proof}
    Observe that the following association is a well defined functor:
    \begin{equation*}
        E_*\colon \Subd(\Sigma)\to \Subd(\Phi), S\mapsto S\circ E.
    \end{equation*}
    We want to show that $E_*$ is an essentially surjective fully faithful functor. The fully faithfulness follows from the fact that $E$ is an equivalence, so we focus on the essential surjectivity. For this, consider given a subdivision $S^\prime\colon\Phi^\prime\to \Phi$, we will define a subdivision $E^{-1}S^\prime\colon E^{-1}\Phi^\prime\to \Sigma$ such that $E^*\left( E^{-1}S^\prime\right) \cong \Phi^\prime$ in $\Subd(\Phi)$. Since $E$ is an equivalence and $C_\Sigma$ is a poset, the functor $E$ establishes a bijection $b_E: C_\Sigma\to [\Phi]$. We let $b_G:[\Phi]\to C_\Sigma$ denote its inverse. Observe that for a cone $s$ of $\Phi^\prime$, there is a unique object $p\in C_\Sigma$ with $E(p)\cong S(s)$, namely $p = b_G([S(s)])$. In fact, this defines a functor
    \begin{equation}
        G:C_{\Phi^\prime}\to C_\Sigma, s\mapsto b_G([S(s)]).\label{eq: the inverse}
    \end{equation}
    To construct the poic-complex $E^{-1}\Phi^\prime$, we set $C_{E^{-1}\Phi^\prime} = C_{\Phi^\prime}$, and for an object $s\in C_{E^{-1}\Phi^\prime}$, we define the poic ${E^{-1}\Phi^\prime}(s)$ as the following fibre square
    \begin{equation*}
        \begin{tikzcd}
                    &&{\Phi^\prime}(s)\arrow{d}{\eta_{S,s}}\\
                    {\Sigma}(G(s))\arrow{rr}[swap]{\eta_{E,G(s)}}&&{\Phi}(S(s)).
        \end{tikzcd}
    \end{equation*}
    In other words, the poic ${E^{-1}\Phi^\prime}(s)$ is the subcone of ${E}(G(s))$ given by the image of ${\Phi^\prime}(s)$ under $\eta_{S,s}$. This definition is naturally functorial, so that $E^{-1}\Phi^\prime$ is a poic-complex. The functor $G:C_{E^{-1}\Phi^\prime}\to C_\Sigma$ and the natural inclusions define a subdivision $S^{-1}E:E^{-1}\Phi^\prime\to \Sigma$. To finalize, we remark that the identity functor $C_{\Phi^\prime}\to C_{\Phi^\prime} = C_{E^{-1}\Phi^\prime}$ and the natural transformation $\eta_S$ (and its inverses defined on the corresponding image cones) define an isomorphism of subdivisions $E_*\left(E^{-1}\Phi^\prime\right) \cong \Phi^\prime$. The isomorphism between the cycle groups follows from Lemma \ref{lem: precomposition of subd with eq is subd.} and this equivalence.
\end{proof}
\subsection{Poic-complexes coming from partially open fans.}\label{appendix: poic-complexes coming from partially open fans} Suppose $N$ is a finite rank lattice and let $V$ denote the real vector space $V:=N\otimes_\bbZ\bbR$. A \emph{partially open fan $\Phi$ in $V$} is a partially open polyhedral complex $\Phi$ in $V$ (see \cite{GathmannOchseMSCTV} and \cite{GathmannMarkwigOchseTMSSMC}) whose polyhedra consist of partially open integral cones of $V$. Naturally, any partially open fan in $V$ gives rise to a poic-complex (with the tautological association) and, furthermore, to a linear poic-complex through the embedding map to $V$. We call this linear poic-complex structure on the partially open fan $\Phi$ in $V$ the \emph{natural linear poic-complex structure in $\Phi$}. The support of a partially open fan $\Phi$ in $V$ is the subset $|\Phi|\subset V$ given by
\begin{equation*}
    |\Phi| := \bigcup_{\sigma\in \Phi}\sigma \subset V.
\end{equation*}
\begin{defi}
    Suppose $\Phi$ is a partially open fan in $V$. A partially open fan $\Psi$ in $V$ is called an \emph{honest subdivision of $\Phi$} if any cone of $\Phi$ is a union of cones of $\Psi$ and $|\Psi|=|\Phi|$. If $\Psi$ is an honest subdivision of $\Phi$, then mapping a cone of $\Psi$ to the minimal cone of $\Phi$ it lies in determines a subdivision of the corresponding poic-complexes. We let $\textnormal{HnstSubd}(\Phi)$ denote the category of honest subdivisions of $\Phi$, where the maps are commuting subdivision maps.
\end{defi}

\begin{lem}\label{lem: partially open fan subd equiv to hnstsubd}
    Suppose $\Phi$ is a partially open fan in $V$. Any subdivision $S\colon \Phi^\prime\to\Phi$ (where $\Phi^\prime$ is an arbitrary poic-complex) is equivalent to an honest subdivision $\underline{\Phi^\prime}$ of $\Phi$. Furthermore, this association gives rise to an equivalence of categories 
    \begin{equation*}
        \underline{\bullet}\colon \Subd(\Phi)\to \textnormal{HnstSubd}(\Phi).
    \end{equation*}
\end{lem}
\begin{proof}
    A cone $s$ of the poic-complex $\Phi^\prime$ determines a subcone of $S(s)\in\Phi$, namely $\eta_{S,s}\left({\Phi^\prime}(s)\right)$. We remark that if $s\cong s^\prime$, then the corresponding cones are the same, and hence the collection
    \begin{equation*}
        \{\eta_{S,s}\left({\Phi^\prime}(s)\right) : s \textnormal{ is a cone of }\Phi^\prime\}
    \end{equation*}
    is a finite set of partially open integral cones of $V$. The conditions on subdivisions of poic-complexes imply that this collection is a partially open fan $\underline{\Phi^\prime}$ in $V$ that is also an honest subdivision of $\Phi$. Let $I\colon \underline{\Phi^\prime}\to\Phi$ denote the subdivision map induced by the honest subdivision. The association
    \begin{equation*}
        {E_{\Phi^\prime}}\colon C_{\Phi^\prime}\to \underline{\Phi^\prime}, s\mapsto \eta_{S,s}\left({\Phi^\prime}(s)\right),
    \end{equation*}
    where the morphisms of $\Phi^\prime$ are mapped to the corresponding face relations, is a functor, which must be an equivalence of categories. It can readily be observed that $I\circ {E_{\Phi^\prime}}=S$ and, in addition, the identity maps give rise now to a natural transformation $\eta_{E_{\Phi^\prime}}: {\Phi^\prime}\implies{\underline{\Phi^\prime}}\circ {E_{\Phi^\prime}} $, which actually consists of isomorphisms. Hence, the tuple $E_{\Phi^\prime}\colon \Phi^\prime\to\underline{\Phi^\prime}$ is the sought-after equivalence. By construction, this association
    \begin{equation*}
        \underline{\bullet}\colon\Subd(\Phi)\to \textnormal{HnstSubd}(\Phi)
    \end{equation*}
    is functorial and gives a fully faithful functor. In addition, if $S\colon \Phi^\prime\to\Phi$ was an honest subdivision, then $\underline{\Phi^\prime}$ is this same honest subdivision. Thus it is also essentially surjective, and therefore an equivalence.
\end{proof}

Of course, if $\Phi$ is a partially open fan, then $\textnormal{HnstSubd}(\Phi)$ is a small category. We close with a proof of Lemma \ref{lem: classical cycles and our cycles}.
\classicaltint*
\begin{proof}
    It is sufficient to consider honest subdivisions of $\Phi_\cX$ to compute its tropical cycles. Any honest subdivision of $\Phi_\cX$ gives rise to a subdivision of $\cX$, by intersecting it with the hyperplane $z=1$. Since honest subdisivions of $\Phi_\cX$ are given by partially open fans in the underlying vector space, it follows that taking intersection and keeping the corresponding weights will produce a Minkowski of one dimension less. To go the other way around, we consider a subdivision whose recessions cones for a fan (\cite{GilSombra}), and then extend the weights correspondingly (namely, following the notation of the construction, we can set $\omega(\sigma_\xi):=\omega(\xi)$). This produces a Minkowski weight of $\Phi_\cX$ of one dimension higher, and these processes are inverse to each other.
\end{proof}

\subsection{The ${\ord(\Phi)}$-construction.} We assume, for the sake of simplicity, that any two zero-dimensional cones of a poic-complex are isomorphic. This is not a restrictive condition and simplifies our work in this subsection. We will explain how to construct, for an arbitrary poic-complex whose underlying category is a poset, a subdivision that is a partially open simplicial fan in a specific vector space. This subdivision has the feature that at each cone of the poic-complex we obtain its barycentric subdivision. We first deal with the case where $\Phi$ consists of closed convex cones, and then explain how to do the general case. 

\begin{nota}
    For a closed simplicial cone $\sigma$, we let $n_\sigma$ denote the vector lying in the interior of the underlying polyhedral cone given by the sum of the integral generators of the rays of $\sigma$. 
\end{nota}

\begin{const}
    Suppose $\Phi$ is a poic-complex of closed convex polyhedral cones, and whose underlying category is a poset (in this case the underlying set is the set of isomorphism classes). We let $\inf\Phi$ denote the unique zero dimensional cone of $\Phi$. Let $C_{\ord(\Phi)}$ denote the poset of totally ordered chains of $C_\Phi$ of non-trivial cones, and let $N_{{\ord(\Phi)}}$ denote the free abelian group on $\Phi$, that is $N_{{\ord(\Phi)}} :=\bbZ^{\Phi}$. We set $V_{{\ord(\Phi)}}:=N_{{\ord(\Phi)}}\otimes_\bbZ \bbR = \bbR^{\Phi}$, and construct a simplicial fan in this vector space. Consider the following closed integral polyhedral cones of $V_{{\ord(\Phi)}}$ defined by the chains in $C_{\ord(\Phi)}$:
    \begin{itemize}
        \item At chains $\{s\}$ of length $1$, we define $\sigma_{s}$ as the one dimensional closed integral polyhedral cone of $V_{{\ord(\Phi)}}$ generated by the basis vector corresponding to $s$.
        \item At chains $c=\{s_1<s_2<\dots<s_n\}$, we define $\sigma_{s_1<\dots<s_n}$ as the closed integral polyhedral cone of $V_{{\ord(\Phi)}}$ generated by the rays $\sigma_{s_1},\dots,\sigma_{s_n}$.
    \end{itemize}
    Containment relations of chains in $C_{\ord(\Phi)}$ become face relations of the corresponding cones. Let $\underline{o}$ denote the closed integral polyhedral cone of $V_{{\ord(\Phi)}}$ given by the origin. By construction, the collection
    \begin{equation*}
        {\ord(\Phi)} := \{\sigma_c : c\in C_{\ord(\Phi)}\}\cup\{\underline{o}\}
    \end{equation*}
    is a fan in $V_{{\ord(\Phi)}}$. Notice that for positive $k$, the $k$-dimensional cones of ${\ord(\Phi)}$ correspond to the $k$-length chains of $\Phi$. We remark that the function
    \begin{equation*}
        \max : {\ord{\Phi}}\to \Phi, \begin{cases}
            \sigma_c\mapsto \max c, & c\in C_{\ord(\Phi)},\\
            \underline{o}\mapsto\inf\Phi,
        \end{cases}
    \end{equation*}
    is actually a functor. We upgrade this functor to a morphism of poic-complexes by defining the natural transformation as follows:
    \begin{itemize}
        \item At the trivial cone $\underline{o}$, it is the only possible morphism between zero dimensional poics.
        \item At the one-dimensional cones $\sigma_s$ with $s\in\Phi$, it is defined as the linear map $\eta_{\max,\sigma_s}\colon \sigma_s \to \Phi(s)$ given by mapping the generator of this ray to the vector $n_{\Phi(s)}$.
        \item At chains $c=\{s_1<\dots <s_n\}$ with $n>1$, it is defined as the linear map
        \begin{equation*}
            \eta_{\max,\sigma_c}\colon \sigma_c\to \Phi(\max c),
        \end{equation*}
        defined by the compositions $\sigma_{s_i}\xrightarrow{\eta_{\max, s_i}} \Phi(s_i)\xrightarrow{\Phi(s_i\leq \max c)}\Phi(\max c)$.
    \end{itemize}
    With the above, we obtain a morphism of poic-complexes $\max:{\ord(\Phi)}\to \Phi$.
\end{const}
\begin{lem}
    If $\Phi$ is a poic-complex of closed convex polyhedral cones whose underlying category is a poset, then the morphism $\max:{\ord(\Phi)}\to \Phi$ is a subdivision.
\end{lem}
\begin{proof}
    The above construction provides the barycentric subdivison at each cone of $\Phi$. The proof is a long, but routine, check which we omit.
\end{proof}

We use the above construction to obtain an analogous result for general poic-complexes whose underlying categories are posets. The idea is to first construct a poic-complex of closed convex polyhedral cones whose underlying category is also a poset that additionally contains our starting poic-complex as a full subcategory (and as subcones), then perform the previous construction in the new one, and finally just remove faces accordingly.

\begin{const}
    Suppose $\Phi$ is a poic-complex whose underlying category is a poset. For each cone $s\in\Phi$, let $\overline{\Phi(s)}$ denote the closed poic defined by $\Phi(s)$. For a cone $s$ of $\Phi$, we define $\add(s)$ as the set of faces $\tau\leq \overline{\Phi(s)}$ ($\add(s)$ stands for additional faces of $s$) subject to the condition: there does not exist $t\lneq s$ such that $\tau$ is a face of the image of $\overline{\Phi(t)}$ under $\Phi(t\lneq s)$. These come with a natural relation given by inclusion of faces.\\
    We now define the poset $C_{\overline{\Phi}}$ inductively as follows. Let $n = \max \{\dim s : s\in C_\Phi\}$, and set $C_{0} = C_\Phi$. Define 
    \begin{equation*}
        C_{i+1} := C_i \cup \bigcup_{s\in C_\Phi, d(s) = i+1} \add(s),
    \end{equation*}
    so that $C_{\overline{\Phi}}:=C_n$, and its order relation is given by extension of the following:
    \begin{itemize}
        \item the order relation of $C_\Phi$.
        \item the order relation (face relation) of each $\add(s)$, for each cone $s$ of $\Phi$ .
        \item setting that $\tau\leq s$, for every $\tau\in \add(s)$ and $s\in C_\Phi$.
    \end{itemize}
    We now set 
    \begin{equation*}
        {\overline{\Phi}}:C_{\overline{\Phi}}\to \POIC, p\mapsto \begin{cases}
            \overline{\Phi(p)},& p\in C_\Phi,\\
            p, &p\in \add(s),\textnormal{ for some }s\in C_\Phi,
        \end{cases}
    \end{equation*}
    this preserves the face relations, and hence makes the tuple $\overline{\Phi}:=(C_{\overline{\Phi}},{\overline{\Phi}})$ into a poic-complex. Furthermore, this is a poic-complex of closed polyhedral cones whose underlying category is a poset. To finalize, we consider $\ord(\overline{\Phi})$, which is a fan that subdivides $\overline{\Phi}$, and take the corresponding partially open integral cones to obtain a partially open fan that subdivides $\Phi$.
\end{const}
In general, when $\Phi$ is a poic-complex whose underlying category is a poset, we will refer to this construction as the \emph{$\ord(\Phi)$-construction}. For general poic-complexes, namely if we drop the assumption on zero dimensional cones, we obtain an analogous result by looking at the maximal subcomplexes that satisfy this property. There are finitely many, and for each of these we would obtain such a subdivision.

\subsection{Proofs of Lemmas \ref{lem: subdivisions are essentially small for simplicial}, \ref{lem: modern cycles and our cycles}, \ref{lem: P-fine subdivisions are final}, and \ref{lem: refining to pi-compatible subdivisions}}\label{appendix: proofs of lemmas}

\sbdvs*

\begin{proof}
    From Lemmas \ref{lem: poic-complex equiv to poset poic-complex} and \ref{lem: equiva poic-complexes equivalent subdivisions}, it can be assumed, without loss of generality, that $\Phi$ is a poic-complex whose underlying category is a poset. The result follows by applying the $\ord(\Phi)$-construction and Lemma \ref{lem: partially open fan subd equiv to hnstsubd} (just as at the end of the previous subsection, we must look at the maximal subcomplexes that satisfy the assumption on zero dimensional cones).
\end{proof}

\moderntint*
\begin{proof}
The fan $\ord(\mathrm{p}(\Sigma))$ is a closed fan, and hence a cone complex. In addition, the morphism $\phi^\Sigma$ makes it a weakly embedded cone complex. Notice that $\ord(\mathrm{p}(\Sigma))$ is isomorphic (as weakly embedded cone complexes) to the barycentric subdivision of $\Sigma$, and hence have isomorphic groups of tropical cycles of weakly embedded cone complexes. Proper subdivisions of this weakly embedded cone complex $\ord(\mathrm{p}(\Sigma))$ are the same as honest subdivisions of the underlying fan. Hence, Lemma \ref{lem: partially open fan subd equiv to hnstsubd} shows that the group of tropical cycles of this weakly embedded cone complex is isomorphic to the group of tropical cycles of this fan.
\end{proof}

\pfnsbdvs*

\begin{proof}
    Because of Lemma \ref{lem: poic-complex equiv to poset poic-complex}, we assume without loss of generality that the underlying category of $\Psi$ is a poset. Then can we proceed just as in Construction 2.24 of \cite{GathmannKerberMarkwigTFMSTC}.
\end{proof}

\pfnpicmptblsbdvsns*

\begin{proof}
    If $E\colon \Sigma\to\Phi$ is an equivalence, then:
    \begin{itemize}
        \item $\pi\circ E\colon \Sigma\to \cX$ is a poic-fibration.
        \item Any $(\pi\circ E)$-compatible subdivision is, after pre-composition with $E$ (see Lemma \ref{lem: precomposition of subd with eq is subd.}) a $\pi$-compatible subdivision.
    \end{itemize}
    From Lemma \ref{lem: poic-complex equiv to poset poic-complex}, it follows that there is an equivalence $E\colon \Sigma\to \Phi$ where the underlying category of $\Sigma$ is a poset. We can then apply the functor $E^*$ (see the proof of Lemma \ref{lem: equiva poic-complexes equivalent subdivisions}), and obtain a subdivision $E^*S\colon E^*\Phi^\prime\to \Sigma$ such that $\left(E^*S\right) \circ E \cong S$. Applying once again Lemma \ref{lem: poic-complex equiv to poset poic-complex}, we obtain an equivalence $E^\prime\colon \Sigma^\prime\to E^*\Phi^\prime$. The composition $E^*S\circ E^\prime$ is then a subdivision $S_E\colon \Sigma^\prime\to\Sigma$, where the underlying categories of both poic-complexes are finite posets. Since any subdivision $S^\prime_E\colon \Sigma^{\prime\prime}\to\Sigma^\prime$ gives rise to a subdivision $E^\prime\circ S^\prime_E\colon\Sigma^{\prime\prime}\to E^*\Phi^\prime$ and $S^\prime_E\circ E \cong S$, it is sufficient to show the lemma when the underlying categories of both $\Phi^\prime$ and $\Phi$ are finite posets.\\
    Let $n=\max_{s\in\Phi} d(s)$. Observe that for $s\in\Phi$, the subdivision $S^{-1}(s)$ can be regarded as a full subcategory of $\Phi^\prime$. Such identifications will be used throughout the proof. Since the underlying category of $\Phi$ is a poset, the set of isomorphism classes $[\Phi]$ is simply $\Phi$. The poic-fibration $\pi\colon\Phi\to\cX$ induces a partition of $\Phi$ indexed by $[\cX]$:
    \begin{equation*}
        \Phi = \bigsqcup_{[x]\in\cX} \Phi_{[x]},
    \end{equation*}
    where for $[x]\in\cX$, the set $\Phi_{[x]}$ consists of the $s\in \Phi$ with $[\pi(s)]=[x]$. Now, we make the following choices:
    \begin{itemize}
        \item For each $[x]\in [\cX]$ a $s_{[x]}\in \Phi_{[x]}$.
        \item For every $s\in\Phi_{[x]}$ a fixed isomorphism $f_s\colon\pi(s_{[x]})\to \pi(s)$ of $\cX$.
    \end{itemize}
    We will construct subdivisions $S_i\colon \Phi^\prime_i\to\Phi(\leq i)$ (for $2\leq i\leq \max_{s\in\Phi} d(s)$) inductively by dimension, with the goal that $\Phi^\prime_i$ is stable under isomorphisms between cones of $\cX$ of dimension $\leq i$. 
    \begin{itemize}
        \item For $i=2$, we consider $[x]\in[\cX](2)$. The automorphisms of $\pi(s_{[x]})$ act on ${\Phi}(s_x)^0$, and the subdivision $S^{-1}(s_x)$ can then be pushed through by any of these automorphisms of $\pi(s_{[x]})$. This gives multiple new subdivisions of $\Phi(s_{[x]})^0$, and taking the intersection of all of these produces a subdivision that will be stable under the automorphisms of $\pi(s_{[x]})$. We do this for every $[x]\in[\cX](2)$, and then transport these subdivisions to all $s\in\Phi$ with $d(s)=2$ by means of the maps induced by the isomorphisms $f_s\colon \pi(s)\to \pi(s_{[\pi(s)]})$. Since we are dealing with cones of dimension $2$, subdivisions of these cones can be regarded as honest subdivisions of $\Phi(s)$. Keeping track of the corresponding face relations gives rise to a poset $C_{\Phi^\prime_2}$, and taking corresponding subcones will make this poset into a poic-complex $\Phi^\prime_2$ with a morphism $S_2\colon \Phi^\prime_2\to\Phi^\prime(\leq 2)$. This is already stable with respect to isomorphisms between $2$-dimensional cones of $\cX$ by construction (any isomorphism will factor as a composition of the $f_s$ with some automorphism of the $\pi(s_x)$). The reason for dimension $2$ is just that this is the smallest where the previous construction is non-trivial.
        
        \item Given $S_i\colon \Phi_i^\prime\to \Phi^\prime(\leq i)$, we construct $\Phi_{i+1}^\prime$. First, we extend $S_i$ to a subdivision of $\Phi^\prime(\leq i+1)$ by making stellar subdivisions, also called star subdivisions, (along arbitrary rays) at every $(i+1)$-dimensional cone of $\Phi^\prime$, and then linearly extending from the faces inwards. This produces a subdivision $\hat{S}_{i}\colon \hat{\Phi}_i^\prime\to \Phi^\prime(\leq i+1)$. Since stellar subdivisions do not affect the boundary of the cones, this $\hat{S}_i$ is still stable under isomorphisms between objects of $\cX$ of dimension $\leq i$. Consider now $[x]\in[\cX](i+1)$ and mimic the previous procedure. The automorphisms of $\pi(s_{[x]})$ act on ${\Phi}(s_{[x]})^0$, and we push through these the subdivision $(S\circ\hat{S}_i)^{-1}(s_{[x]})$. This gives multiple new subdivisions of $\Phi(s)^0$, which we intersect and obtain a subdivision of $\Phi(s_{[x]})^0$ stable under the automorphisms of $\pi(s_{[x]})$. We repeat this for every $[x]\in[\cX](i+1)$, and then transport these subdivisions to all the $s\in\Phi$ with $d(s)=i+1$ by means of the maps induced by the isomorphisms $f_s\colon \pi(s)\to \pi(s_{[\pi(s)]})$. Since $\hat{\Phi}^\prime_i$ is stable under isomorphisms between objects of $\cX$ of dimension $\leq i$, it follows that these subdivisions can be extended to the whole $\Phi(s)$. By keeping track of the corresponding face relations we obtain a finite poset $C_{\Phi^\prime_{i+1}}$, and taking corresponding subcones makes this poset into a poic-complex $\Phi^\prime_{i+1}$ with a morphism $S_{i+1}\colon \Phi^\prime_{i+1}\to\Phi^\prime(\leq i+1)$. By construction, this subdivision is stable under isomorphisms between objects of $\cX$ of dimension $\leq i+1$.    
    \end{itemize}
    
    \end{proof}

\end{document}